\documentclass[]{scrartcl}

\usepackage[utf8]{luainputenc}
\usepackage[USenglish]{babel}
\usepackage{csquotes}

\usepackage[a4paper,top=27mm,bottom=20mm,inner=25mm,outer=20mm]{geometry}

\usepackage[%
  backend=bibtex,bibencoding=ascii,
  style=numeric-comp,
  giveninits=true, uniquename=init, 
  natbib=true,
  url=true,
  doi=true,
  isbn=false,
  backref=false,
  maxnames=99,
  ]{biblatex}
\usepackage{amsmath}
\allowdisplaybreaks
\numberwithin{equation}{section}
\usepackage{amssymb}
\usepackage{commath}
\usepackage{mathtools}
\usepackage{bbm}
\usepackage{nicefrac}
\usepackage{subdepth}
\usepackage{adjustbox}

\usepackage{siunitx}
\usepackage{algorithm}
\usepackage{algpseudocode}

\usepackage{amsthm}
\usepackage{thmtools}
\usepackage{etoolbox}
\makeatletter
\patchcmd{\thmt@setheadstyle}
 {\bgroup\thmt@space}
 {\thmt@space}
 {}{}
\patchcmd{\thmt@setheadstyle}
 {\egroup\fi}
 {\fi}
 {}{}
\makeatother
\declaretheoremstyle[
  bodyfont=\normalfont\itshape,
  headformat=\NAME\ \NUMBER\NOTE,
]{myplain}
\declaretheoremstyle[
  headformat=\NAME\ \NUMBER\NOTE,
]{mydefinition}
\newcommand{\envqed}{{\lower-0.3ex\hbox{$\triangleleft$}}}
\declaretheorem[style=myplain,numberwithin=section]{theorem}

\declaretheorem[style=mydefinition,numberlike=theorem,qed=\envqed]{definition}
\declaretheorem[style=mydefinition,numberlike=theorem,qed=\envqed]{remark}
\declaretheorem[style=mydefinition,numberlike=theorem,qed=\envqed]{example}

\usepackage[plainpages=false,pdfpagelabels,hidelinks,unicode]{hyperref}

\usepackage{color}
\usepackage{graphicx}
\usepackage[small]{caption}
\usepackage{subcaption}

\begingroup\expandafter\expandafter\expandafter\endgroup
\expandafter\ifx\csname pdfsuppresswarningpagegroup\endcsname\relax
\else
\fi

\usepackage{booktabs}
\usepackage{rotating}
\usepackage{multirow}

\usepackage{enumitem}

\usepackage{ifluatex}
\ifluatex
  \usepackage[no-math]{fontspec}
\else
  \usepackage[T1]{fontenc}
\fi
\usepackage{newpxtext,newpxmath}

\usepackage{xparse}

\let\epsilon\varepsilon
\let\phi\varphi
\let\rho\varrho

\usepackage{xspace}

\providecommand\R{}
\renewcommand{\R}{\mathbb{R}}

\NewDocumentCommand{\RK}{o m O{\the\numexpr#2-1\relax} m O{} O{} o}{%
  \IfValueTF{#1}{#1}{RK}%
  #2(#3)#4%
  \ifblank{#6}{}{\textsubscript{F}}%
  \ifblank{#5}{}{[#5]}%
  \IfValueT{#7}{#7}%
}

\renewcommand{\Re}{\operatorname{Re}}

\newcommand{\dx}{\Delta x}

\newcommand{\diag}{\operatorname{diag}}

\renewcommand{\vec}[1]{\pmb{#1}}

\newcommand{\xmin}{x_\mathrm{min}}
\newcommand{\xmax}{x_\mathrm{max}}

\newcommand{\tL}{\vec{t}_L}
\newcommand{\tR}{\vec{t}_R}

\newcommand{\SAT}{\vec{\mathrm{SAT}}}
\newcommand{\BTS}{\vec{\mathrm{BTs}}}

\newcommand{\Ngeo}{N_{\mathrm{geo}}}

\NewDocumentCommand{\opD}{m+g}{%
  \IfNoValueTF{#2}
    {D_{#1}}
    {D_{#1,#2}}%
}
\NewDocumentCommand{\opDsplit}{m+g}{%
  \IfNoValueTF{#2}
    {\widetilde{D}_{#1}}
    {\widetilde{D}_{#1,#2}}%
}
\NewDocumentCommand{\opM}{g}{%
  \IfNoValueTF{#1}
    {M}
    {M_{#1}}%
}
\NewDocumentCommand{\opQ}{g}{%
  \IfNoValueTF{#1}
    {Q}
    {Q_{#1}}%
}
\NewDocumentCommand{\opI}{g}{%
  \IfNoValueTF{#1}
    {I}
    {I_{#1}}%
}
\NewDocumentCommand{\opV}{g}{%
  \IfNoValueTF{#1}
    {V}
    {V_{#1}}%
}
\NewDocumentCommand{\opB}{g}{%
  \IfNoValueTF{#1}
    {B}
    {B_{#1}}%
}
\NewDocumentCommand{\opR}{g}{%
  \IfNoValueTF{#1}
    {R}
    {R_{#1}}%
}
\NewDocumentCommand{\opN}{m+g}{%
  \IfNoValueTF{#2}
    {N_{#1}}
    {N_{#1,#2}}%
}

\NewDocumentCommand{\fnum}{g}{%
  \IfNoValueTF{#1}
    {f^{\mathrm{num}}}
    {f^{\mathrm{num,#1}}}%
}
\NewDocumentCommand{\vecfnum}{g}{%
  \IfNoValueTF{#1}
    {\vec{f}^{\mathrm{num}}}
    {\vec{f}^{\mathrm{num,#1}}}%
}
\NewDocumentCommand{\vecfcorr}{g}{%
  \IfNoValueTF{#1}
    {\vec{f}^{\mathrm{corr}}}
    {\vec{f}^{\mathrm{corr,#1}}}%
}
\NewDocumentCommand{\fvol}{g}{%
  \IfNoValueTF{#1}
    {f^{\smash{\mathrm{vol}}}}
    {f^{\smash{\mathrm{vol,#1}}}}%
}

\newcommand{\orcid}[1]{ORCID:~\href{https://orcid.org/#1}{#1}}
\usepackage{authblk}

\newenvironment{keywords}{\par\textbf{Key words.}}{\par}
\newenvironment{AMS}{\par\textbf{AMS subject classification.}}{\par}

\title{On the robustness of high-order upwind summation-by-parts methods for nonlinear conservation laws}

\author[1]{Hendrik~Ranocha\thanks{\orcid{0000-0002-3456-2277}}}
\affil[1]{Institute of Mathematics, Johannes Gutenberg University Mainz, Germany}

\author[2]{Andrew~R.~Winters\thanks{\orcid{0000-0002-5902-1522}}}
\affil[2]{Department of Mathematics; Applied Mathematics, Linköping University, Sweden}

\author[3,4,5]{Michael~Schlottke-Lakemper\thanks{\orcid{0000-0002-3195-2536}}}
\affil[3]{High-Performance Scientific Computing, University of Augsburg, Germany}
\affil[4]{Applied and Computational Mathematics, RWTH Aachen University, Germany}
\affil[5]{High-Performance Computing Center Stuttgart (HLRS), University of Stuttgart, Germany}

\author[1,6]{Philipp~Öffner\thanks{\orcid{0000-0002-1367-1917}}}
\affil[6]{Institute of Mathematics, TU Clausthal, Germany}

\author[7]{Jan~Glaubitz\thanks{\orcid{0000-0002-3434-5563}}}
\affil[7]{Aeronautics and Astronautics, Massachusetts Institute of Technology, USA}

\author[8,9]{Gregor~J.~Gassner\thanks{\orcid{0000-0002-1752-1158}}}
\affil[8]{Department of Mathematics and Computer Science, University of Cologne, Germany}
\affil[9]{Center for Data and Simulation Science, University of Cologne, Germany}

\date{September 18, 2024} 

\makeatletter
\hypersetup{pdfauthor={Hendrik Ranocha et al.}} 
\hypersetup{pdftitle={On the robustness of high-order upwind summation-by-parts methods for nonlinear conservation laws}} 
\makeatother

\begin{document}

\maketitle

\begin{abstract}
\noindent
  We use the framework of upwind summation-by-parts (SBP) operators developed
by Mattsson (2017, \href{https://doi.org/10.1016/j.jcp.2017.01.042}{doi:10.1016/j.jcp.2017.01.042}) and
study different flux vector splittings in this context. To do so, we
introduce discontinuous-Galerkin-like interface terms for multi-block upwind
SBP methods applied to nonlinear conservation laws.
We investigate the behavior of the upwind SBP methods for flux vector splittings
of varying complexity on Cartesian as well as unstructured curvilinear multi-block meshes.
Moreover, we analyze the
local linear/energy stability of these methods following
Gassner, Svärd, and Hindenlang (2022, \href{https://doi.org/10.1007/s10915-021-01720-8}{doi:10.1007/s10915-021-01720-8}).
Finally, we investigate the robustness of upwind SBP methods for challenging
examples of shock-free flows of the compressible Euler equations such as
a Kelvin-Helmholtz instability and the inviscid Taylor-Green vortex.

\end{abstract}



\section{Introduction}
\label{sec:introduction}

Stability and robustness are crucial properties of numerical methods for
conservation laws to obtain reliable simulations, in particular for
under-resolved flows. At the same time, high-order methods can be very
efficient and fit well to modern hardware. However, it is non-trivial to
ensure robustness of high-order methods without destroying their
high-order accuracy.

Over the last decade, entropy-based methods have emerged as a popular choice
to construct robust high-order methods in a wide range of applications.
Built from the seminal work of Tadmor
\cite{tadmor1987numerical,tadmor2003entropy}, high-order extensions
have been developed in \cite{lefloch2002fully,fisher2013high}. These
flux differencing schemes work well for under-resolved flows, e.g.,
\cite{gassner2016split,rojas2021robustness,chan2022entropy,klose2020assessing,sjogreen2018high}.
However, some doubts have been raised recently within the high-order
community by \citet{gassner2022stability}. In their article, the authors
demonstrated critical failures of high-order entropy-dissipative methods
for a conceptually simple setup of the 1D compressible Euler equations;
with constant velocity and pressure, these equations reduce to simple linear advection
of the density. Central schemes without any entropy properties perform well
in this case but crash for demanding simulations of under-resolved flows
such as the inviscid Taylor-Green vortex. In contrast, entropy-stable flux
differencing methods work well for the Taylor-Green vortex but fail for the
apparently simple advection example.

Failures due to positivity issues can be fixed by adding
invariant domain preserving techniques, e.g.,
\cite{rueda2023monolitic,pazner2021sparse,maier2021efficient,guermond2022implementation}.
However, it is desirable to combine such shock-capturing and invariant
domain preserving approaches with a good baseline scheme such that the
amount of additional dissipation can be kept low \cite{rueda2021subcell}.
Thus, we are interested in high-order baseline schemes that come already
with some built-in dissipation everywhere, not only at element interfaces
as typical in discontinuous Galerkin (DG) methods. At the same time, we would
like to avoid having additional parameters in the schemes that need to be
tuned manually.

Many high-order methods with some provable stability properties can be
obtained in the general framework of summation-by-parts (SBP) operators.
SBP operators were originally developed for finite difference methods
\cite{kreiss1974finite,strand1994summation}.
They are the basis of entropy-stable flux differencing methods by mimicking
integration by parts discretely. Many common numerical methods can be
formulated using SBP operators, e.g.,
finite volume methods \cite{nordstrom2001finite,nordstrom2003finite},
continuous Galerkin methods \cite{hicken2016multidimensional,hicken2020entropy,abgrall2020analysisI,abgrall2023analysis},
DG methods \cite{gassner2013skew,carpenter2014entropy,chan2018discretely},
and flux reconstruction methods \cite{huynh2007flux,ranocha2016summation}.
Further information and background material on SBP operators is collected
in the review articles \cite{svard2014review,fernandez2014review}.

Classical SBP operators can be used to design numerical schemes that are
provably stable. Typically, SBP methods are based on central-type
discretizations in the interior and weak imposition of boundary data
using simultaneous approximation terms (SATs)
\cite{carpenter1994time,carpenter1999stable}
that introduce some dissipation.
In a multi-block finite difference or DG setting,
such SATs are also used to couple the blocks/elements weakly and introduce
additional dissipation --- but only at interfaces, not in the interior of
the block/elements. To obtain additional dissipation everywhere, artificial
dissipation operators can be used \cite{mattsson2004stable}. These operators
can be combined with a user-chosen amount of dissipation and it may be
non-trivial to choose an appropriate amount of dissipation.

Combining classical SBP operators and artificial dissipation can be
interpreted as upwinding \cite{svard2005steady,mattsson2007high}.
\citet{mattsson2017diagonal} introduced a general definition of upwind
SBP operators and constructed a range of schemes with good numerical
properties, resulting in a parameter-free combination of central-type
SBP operators and artificial dissipation. These upwind SBP operators have
been used successfully for a range of applications such as the shallow
water equations \cite{lundgren2020efficient},
atmospheric flows \cite{rydin2018high},
and scalar conservation laws \cite{stiernstrom2021residual}.
They have also been extended to staggered grids in
\cite{mattsson2018compatible}. Their relations to DG methods have been
discussed in \cite{ranocha2021broad,ortleb2023stability}.

To apply upwind SBP operators to nonlinear conservation laws, a flux
vector splitting is required \cite{mattsson2017diagonal}. Across the literature
\cite{svard2005steady,mattsson2007high,mattsson2017diagonal,lundgren2020efficient,rydin2018high,stiernstrom2021residual}
the numerical testing is predominantly done with
Lax-Friedrichs type splittings.
For many numerical schemes, such Lax-Friedrichs type splittings are not
ideal. As stated by \citet[Remark~4.2]{stiernstrom2021residual}, many other
flux vector splittings are available but have not been studied in detail
with upwind SBP operators so far. One of the goals of this article is to
fill this gap and investigate the impact of different flux vector splittings
on robustness for challenging examples on Cartesian and curvilinear meshes.

To do so, we first review upwind SBP operators \cite{mattsson2017diagonal}
and classical flux vector splittings \cite[Chapter~8]{toro2009riemann} in
Section~\ref{sec:basics}.
These flux vector splitting methods have been widely developed and used
in the last century
\cite{steger1979flux,vanleer1982flux,hanel1987accuracy,liou1991high,coirier1991numerical,buning1982solution}
but were abandoned in favor of other techniques due to their significant
amount of numerical dissipation \cite{vanleer1991flux}. We will see that
the combination of flux vector splitting techniques with high-order
difference operators does not lead to an excessive amount of artificial
dissipation.

Next, we formulate high-order upwind SBP methods for nonlinear problems
in Section~\ref{sec:formulations} based on the seminal works of Mattsson
and collaborators \cite{mattsson2017diagonal,svard2005steady,mattsson2007high,lundgren2020efficient,rydin2018high,stiernstrom2021residual}. To enable an investigation across a range of different flux
vector splittings in multi-block finite difference methods, we need to
introduce appropriate SATs. To do so, we start with a classical upwind
SBP formulation and introduce interface terms as in DG
methods --- using numerical fluxes resulting from the flux vector splitting.
We then discuss the relation of this formulation to the construction of
global upwind SBP operators as done in \cite{ranocha2021broad}.

In the final part of Section~\ref{sec:formulations}, we consider the
upwind SBP methods on unstructured curvilinear multi-block meshes.
The formulation in generalized coordinates reveals a subtle interplay between
the finite difference operator and the particular flux vector splitting.
Moreover, we demonstrate that these subtleties are not an issue for Lax-Friedrichs
type splittings; however, they are present for more sophisticated splitting techniques.

Afterwards, we follow Gassner, Svärd, and Hindenlang
\cite{gassner2022stability} and analyze the local linear/energy stability
properties of upwind SBP methods in Section~\ref{sec:stability}.
In particular, we prove local linear/energy stability for Burgers' equation
in the setting where \citet{gassner2022stability} observed stability issues
for entropy-stable methods based on classical SBP operators.

In Section~\ref{sec:experiments}, we investigate the behavior of upwind
SBP methods with different flux vector splittings numerically.
We begin with 1D convergence tests, verify the local linear/energy stability
results, and then proceed to 2D and 3D simulations of under-resolved flows
on Cartesian meshes.
In particular, we consider shock-free setups for the compressible Euler
equations and study the robustness for two challenging setups: a
Kelvin-Helmholtz instability and the inviscid Taylor-Green vortex.
We further study the convergence and free-stream preservation properties
on unstructured curvilinear meshes with different flux vector splittings.
Finally, we summarize our findings and provide an outlook on further
research in Section~\ref{sec:summary}.

\section{Review of upwind SBP operators and flux vector splitting}
\label{sec:basics}

Consider a hyperbolic conservation law
\begin{equation}
\label{eq:HCL-1D}
  \partial_t u(t, x) + \partial_x f\bigl( u(t, x) \bigr) = 0,
  \qquad t \in (0, T), x \in (\xmin, \xmax),
\end{equation}
with conserved variable $u$ and flux $f$ in one space dimension, equipped
with appropriate initial and boundary conditions.
For now, we concentrate on the 1D setting to describe the overall methodologies.
Extension of the method to multiple space dimensions is done using a tensor product structure.
We delay a detailed discussion of the continuous and discrete formulations in generalized
curvilinear coordinates to Section~\ref{sec:curvilinear}.

In this section, we review classical flux vector splitting techniques,
the basic idea of upwind SBP methods, and collect some useful properties of
upwind SBP operators for reference. All these concepts and results are
well-known in the literature, but we collect them here to make the article
self-contained.

\subsection{Flux vector splitting}

The classical flux vector splitting approach \cite[Chapter~8]{toro2009riemann}
to create (semi-)discretizations of the conservation law \eqref{eq:HCL-1D}
begins with an appropriate splitting of the flux $f$ such that
\begin{equation}
  \label{eq:flux_vec_split}
  f(u) = f^-(u) + f^+(u),
\end{equation}
where the eigenvalues $\lambda^{\pm}_i$ of the Jacobians
$J^{\pm} = \partial_u f^{\pm}$ satisfy
\begin{equation}
  \forall i\colon \quad \lambda^-_i \le 0, \; \lambda^+_i \ge 0.
\end{equation}
There is a great deal of freedom in the construction of a flux vector splitting \eqref{eq:flux_vec_split}
to create an upwind scheme. The design of $f^-(u)$ and $f^+(u)$ typically relies on
the mathematically sound characteristic theory for hyperbolic partial differential equations.
Depending on how one treats the different characteristics, for instance separating the
convective and pressure components of the compressible Euler equations, leads to a wide
variety of flux vector splittings, e.g.,~\cite{vanleer1982flux,hanel1987accuracy,liou1991high,liou1996sequel,zha1993numerical}.
Because the flux vector splitting separates the upwind directions with which
solution information propagates, the resulting scheme does not require the (approximate) solution
of a Riemann problem.
This makes flux vector splitting based algorithms particularly attractive due to their simplicity and ability
to approximate shock waves.
To demonstrate this simplicity, consider a classical first-order
finite volume method of the form
\begin{equation}
  \partial_t \vec{u}_i + \frac{1}{\dx} \left(
    \fnum(\vec{u}_{i}, \vec{u}_{i+1}) - \fnum(\vec{u}_{i-1}, \vec{u}_{i})
  \right)
  = 0
\end{equation}
with numerical flux $\fnum$. In the flux vector splitting approach, the
numerical flux is chosen as
\begin{equation}
  \fnum(u_l, u_r) = f^+(u_l) + f^-(u_r).
\end{equation}
Thus, the chosen splitting determines the scheme completely.

\begin{example}
\label{ex:global-Lax-Friedrichs}
  The global Lax-Friedrichs splitting requires a global upper bound
  $\lambda$ on the possible wave speeds and uses
  \begin{equation}
    f^\pm(u) = \frac{1}{2} \left( f(u) \pm \lambda u \right).
  \end{equation}
  This results in the numerical flux
  \begin{equation}
    \fnum(u_l, u_r)
    =
    f^+(u_l) + f^-(u_r)
    =
    \frac{1}{2} \left( f(u_l) + f(u_r) \right)
    - \frac{\lambda}{2} (u_r - u_l).
  \end{equation}
  This splitting has predominantly been used in previous works on upwind SBP operators,
  e.g., \cite{mattsson2017diagonal,stiernstrom2021residual,lundgren2020efficient}.
\end{example}

Next, we present some examples for the 1D compressible Euler equations
\begin{equation}
  \partial_t \begin{pmatrix} \rho \\ \rho v \\ \rho e \end{pmatrix}
  + \partial_x \begin{pmatrix} \rho v \\ \rho v^2 + p \\ (\rho e + p) v \end{pmatrix}
  = 0
\end{equation}
of an ideal gas with density $\rho$, velocity $v$, total energy density
$\rho e$, and pressure
\begin{equation}
  p = (\gamma - 1) \left( \rho e - \frac{1}{2} \rho v^2 \right),
\end{equation}
where the ratio of specific heats is usually chosen as $\gamma = 1.4$.
To the best of our knowledge, the splittings described in the following
examples have not been combined
with upwind SBP operators in the existing literature
\cite{svard2005steady,mattsson2007high,mattsson2017diagonal,lundgren2020efficient,rydin2018high,stiernstrom2021residual} or only in less detail.

\begin{example}
\label{ex:Steger-Warming}
  To describe the Steger-Warming splitting \cite{steger1979flux}, we use
  the standard notation
  \begin{equation}
    \lambda_i^\pm = \frac{\lambda_i \pm |\lambda_i|}{2}
  \end{equation}
  for the positive/negative part of an eigenvalue $\lambda_i$.
  The wave speeds of the 1D Euler equations are
  \begin{equation}
    \lambda_1 = v - a,
    \quad
    \lambda_2 = v,
    \quad
    \lambda_3 = v + a,
  \end{equation}
  where the speed of sound is $a = \sqrt{\gamma p / \rho}$. Then, the flux
  splitting of Steger and Warming is given by
  \begin{equation}
    f^\pm
    =
    \frac{\rho}{2 \gamma}
    \begin{pmatrix}
      \lambda_1^\pm + 2 (\gamma - 1) \lambda_2^\pm + \lambda_3^\pm \\
      (v - a) \lambda_1^\pm + 2 (\gamma - 1) v \lambda_2^\pm + (v - a) \lambda_3 ^\pm \\
      (H - v a) \lambda_1^\pm + (\gamma - 1) v^2 \lambda_2^\pm + (H + v a) \lambda_3^\pm
    \end{pmatrix},
  \end{equation}
  where $H = (\rho e + p) / \rho = v^2 / 2 + a^2 / (\gamma - 1)$
  is the enthalpy, see also \cite[Section~8.4.2]{toro2009riemann}.
\end{example}

\begin{example}
\label{ex:van-Leer-Hanel}
  Next, we describe the van Leer-Hänel splitting
  \cite{vanleer1982flux,hanel1987accuracy,liou1991high}
  based on a splitting of van Leer with a modification of the energy flux
  proposed by Hänel et al.\ and the ``p4'' splitting of the pressure proposed
  by Liou and Steffen.
  First, we introduce the signed Mach number $M = v / a$ and the pressure
  splitting
  \begin{equation}
    p^\pm = \frac{1 \pm \gamma M}{2} p.
  \end{equation}
  The fluxes are given by
  \begin{equation}
    f^\pm
    =
    \pm \frac{\rho a (M \pm 1)^2}{4}
    \begin{pmatrix}
      1 \\
      v \\
      H
    \end{pmatrix}
    +
    \begin{pmatrix}
      0 \\
      p^\pm \\
      0
    \end{pmatrix},
  \end{equation}
  where $H = (\rho e + p) / \rho = v^2 / 2 + a^2 / (\gamma - 1)$
  is again the enthalpy.
\end{example}

\subsection{Upwind SBP operators}

In this article, we focus on a collocation setting as in classical finite
difference methods. Thus, we consider a grid $\vec{x} = (\vec{x}_i)_{i=1}^N$
with nodes $\vec{x}_i$ and use pointwise approximations such as
$\vec{u}_i = u(\vec{x}_i)$ and $\vec{1} = (1, \dots, 1)^T$. We also assume that
the grid includes the boundary nodes of the domain, i.e.,
\begin{equation}
  \vec{x}_1 = \xmin, \qquad \vec{x}_N = \xmax.
\end{equation}
Then, classical SBP operators are constructed to mimic integration-by-parts,
cf. \cite{svard2014review,fernandez2014review}.

\begin{definition}
  An \emph{SBP operator} on the domain $[\xmin, \xmax]$ consists of
  a grid $\vec{x}$,
  a symmetric and positive definite mass/norm matrix $M$ satisfying
  $\vec{1}^T M \vec{1} = \xmax - \xmin$, and
  a consistent derivative operator $D$ such that
  \begin{equation}
  \label{eq:SBP}
    M D + D^T M = \tR \tR^T - \tL \tL^T,
  \end{equation}
  where $\tR^T = (0, \dots, 0, 1)$ and $\tL = (1, 0, \dots, 0)^T$.
  It is called {diagonal-norm operator} if $M$ is diagonal.
\end{definition}

We often identify an SBP operator with the derivative operator $D$ and
assume that the remaining parts are clear from the context.
Since the boundary nodes are included, \eqref{eq:SBP} guarantees that the
discrete operators mimic integration-by-parts as
\begin{equation}
\begin{array}{*3{>{\displaystyle}c}}
  \underbrace{
    \vec{u}^T M D \vec{v}
    + \vec{u}^T D^T M \vec{v}
  }
  & = &
  \underbrace{
    \vec{u}^T \tR \tR^T \vec{v} - \vec{u}^T \tL \tL^T \vec{v},
  }
  \\
  \rotatebox{90}{$\!\approx\;$}
  &&
  \rotatebox{90}{$\!\!\approx\;$}
  \\
  \overbrace{
    \int_{\xmin}^{\xmax} u \, (\partial_x v)
    + \int_{\xmin}^{\xmax} (\partial_x u) \, v
  }
  & = &
  \overbrace{
    u(\xmax) v(\xmax) - u(\xmin) v(\xmin)
  }.
\end{array}
\end{equation}

Upwind SBP operators were introduced by Mattsson~\cite{mattsson2017diagonal}. The
basic idea is to introduce two derivative operators $D_\pm$ that mimic
integration-by-parts together and are compatible in the sense that their difference
is negative semidefinite, which allows to introduce artificial dissipation.
\begin{definition}
  An \emph{upwind SBP operator} on the domain $[\xmin, \xmax]$ consists of
  a grid $\vec{x}$,
  a symmetric and positive definite mass/norm matrix $M$ satisfying
  $\vec{1}^T M \vec{1} = \xmax - \xmin$, and
  two consistent derivative operators $D_\pm$ such that
  \begin{equation}
  \label{eq:upwind-SBP}
    M D_+ + D_-^T M = \tR \tR^T - \tL \tL^T,
    \qquad
    M (D_+ - D_-) \text{ is negative semidefinite},
  \end{equation}
  where again $\tR^T = (0, \dots, 0, 1)$ and $\tL = (1, 0, \dots, 0)^T$.
  It is called {diagonal-norm operator} if $M$ is diagonal.
\end{definition}

For convenience, we also identify an upwind SBP operator simply with the derivative matrices $D_\pm$.
In matrix form, the upwind SBP operators derived by Mattsson
\cite{mattsson2017diagonal} are constructed such that $D_+$ is biased toward
the upper-triangular part, i.e., it has more non-zero entries above the diagonal.
Similarly, $D_-$ is biased toward the lower-triangular part.

\begin{example}
\label{ex:upwind-2nd-order}
  The second-order accurate upwind operators of \cite{mattsson2017diagonal} are
  given by
  \begin{equation}
    D_+ = \frac{1}{\dx}
    \begin{pmatrix}
      -3 & 5 & -2 \\
      -1/5 & -1 & 8/5 & -2/5 \\
      && -3/2 & 2 & -1/2 \\
      &&& \ddots & \ddots & \ddots \\
      &&&& 3/2 & 2 & -1/2 \\
      &&&&& -1 & 1 \\
      &&&&& -1 & 1
    \end{pmatrix},
  \end{equation}
  \begin{equation}
    D_- = \frac{1}{\dx}
    \begin{pmatrix}
      -1 & 1 \\
      -1 & 1 \\
      1/2 & -2 & 3/2 \\
      & 1/2 & -2 & 3/2 \\
      && \ddots & \ddots & \ddots \\
      &&& 1/2 & -2 & 3/2 \\
      &&&& 2/5 & -8/5 & 1 & 1/5 \\
      &&&&& 2 & -5 & 3
    \end{pmatrix},
  \end{equation}
  and $M = \dx \diag(1/4, 5/4, 1, \dots, 1, 5/4, 1/4)$.
\end{example}

Upwind SBP operators are constructed to create provably stable semidiscretizations
of linear transport problems as already described in \cite{mattsson2017diagonal}. For completeness and as an example,
consider the linear advection equation
\begin{equation}
\begin{aligned}
  \partial_t u(t, x) + \partial_x u(t, x) &= 0, && t \in (0, T), x \in (\xmin, \xmax),
  \\
  u(0, x) &= u^0(x), && x \in [\xmin, \xmax],
  \\
  u(t, \xmin) &= g_L(t), && t \in [0, T].
\end{aligned}
\end{equation}
Since the transport happens from left to right, we choose the left-biased
upwind operator $D_-$ such that solution information from the correct
characteristic direction is used and obtain a stable semidiscretization
\begin{equation}
\label{eq:linear-advection-upwind-SBP}
  \partial_t \vec{u} + D_- \vec{u} = M^{-1} \tL (g_L - \tL^T \vec{u}).
\end{equation}
This semidiscretization is globally conservative, since
\begin{equation}
\begin{aligned}
  \partial_t (\vec{1}^T M \vec{u})
  &=
  \vec{1}^T M \partial_t \vec{u}
  =
  - \vec{1}^T M D_- \vec{u}
  + \vec{1}^T \tL (g_L - \tL^T \vec{u})
  \\
  &=
  \vec{1}^T D_+^T M \vec{u}
  - \vec{1}^T (\tR \tR^T - \tL \tL^T) \vec{u}
  + \vec{1}^T \tL (g_L - \tL^T \vec{u})
  =
  g_L - \tR^T \vec{u},
\end{aligned}
\end{equation}
where we have used the upwind SBP property \eqref{eq:upwind-SBP} and consistency
of the derivative operator. Furthermore, the semidiscretization
\eqref{eq:linear-advection-upwind-SBP} is energy-stable since
\begin{equation}
\begin{aligned}
  \partial_t \| \vec{u} \|_M^2
  &=
  2 \vec{u}^T M \partial_t \vec{u}
  =
  - 2 \vec{u}^T M D_- \vec{u}
  + 2 \vec{u}^T \tL (g_L - \tL^T \vec{u})
  \\
  &=
  - \vec{u}^T M D_- \vec{u}
  + \vec{u}^T D_+^T M \vec{u}
  - \vec{u}^T (\tR \tR^T - \tL \tL^T) \vec{u}
  + 2 \vec{u}^T \tL (g_L - \tL^T \vec{u})
  \\
  &\le
  2 (\tL^T \vec{u}) g_L
  - (\tL^T \vec{u})^2
  - (\tR^T \vec{u})^2
  =
  g_L^2 - (\tR^T \vec{u})^2 - (g_L - \tL^T \vec{u})^2,
\end{aligned}
\end{equation}
mimicking the estimate
\begin{equation}
  \partial_t \|u(t)\|_{L^2}^2
  =
  g_L(t)^2 - u(t, \xmax)^2
\end{equation}
up to additional artificial dissipation due to the upwind operators and the
weak imposition of boundary data.

\subsection{Some useful properties of upwind SBP operators}

As described in \cite{mattsson2017diagonal}, upwind SBP operators can be
interpreted as classical SBP operators plus artificial dissipation in the
context of the linear advection equation. Indeed,
\begin{equation}
\label{eq:upwindDs}
\begin{aligned}
  D_+ &= \frac{1}{2} (D_- + D_+) + \frac{1}{2} (D_+ - D_-),
  \\
  D_- &= \frac{1}{2} (D_- + D_+) - \frac{1}{2} (D_+ - D_-).
\end{aligned}
\end{equation}
The average of the upwind operators is a classical SBP operator since
\cite{mattsson2017diagonal}
\begin{equation}
  M (D_- + D_+) + (D_- + D_+)^T M
  =
  (M D_+ + D_-^T M) + (M D_- + D_+^T M)
  =
  2 (\tR \tR^T - \tL \tL^T).
\end{equation}
The difference of the upwind SBP operators introduces artificial dissipation
for the linear advection equation
when multiplied by the mass matrix $M$ by construction. Thus, the upwind SBP
discretization \eqref{eq:linear-advection-upwind-SBP} can be written as
\cite{mattsson2017diagonal}
\begin{equation}
  \partial_t \vec{u}
  \underbrace{+ \frac{D_- + D_+}{2} \vec{u}}_{\text{central}}
  \underbrace{- \frac{D_+ - D_-}{2} \vec{u}}_{\text{dissipation}}
  = M^{-1} \tL (g_L - \tL^T \vec{u}).
\end{equation}
This is the form of a central SBP discretization plus artificial dissipation
for linear advection. For general nonlinear problems, we can still use the
negative semidefiniteness of the difference of the operators to introduce
artificial dissipation, but we need proper upwinding as discussed in
Section~\ref{sec:formulations}.

\subsection{Upwind SBP operators in periodic domains}

In periodic domains, we require boundary terms to vanish, resulting in the
following definitions \cite{ranocha2021broad}.

\begin{definition}
  A \emph{periodic SBP operator} on the domain $[\xmin, \xmax]$ consists of
  a grid $\vec{x}$,
  a symmetric and positive definite mass/norm matrix $M$ satisfying
  $\vec{1}^T M \vec{1} = \xmax - \xmin$, and
  a consistent derivative operator $D$ such that
  \begin{equation}
  \label{eq:SBP-periodic}
    M D + D^T M = 0.
  \end{equation}
  It is called {diagonal-norm operator} if $M$ is diagonal.
\end{definition}

\begin{definition}
  A \emph{periodic upwind SBP operator} on the domain $[\xmin, \xmax]$ consists of
  a grid $\vec{x}$,
  a symmetric and positive definite mass/norm matrix $M$ satisfying
  $\vec{1}^T M \vec{1} = \xmax - \xmin$, and
  two consistent derivative operators $D_\pm$ such that
  \begin{equation}
  \label{eq:upwind-SBP-periodic}
    M D_+ + D_-^T M = 0,
    \qquad
    M (D_+ - D_-) \text{ is negative semidefinite}.
  \end{equation}
  It is called {diagonal-norm operator} if $M$ is diagonal.
\end{definition}

An upwind SBP discretization of the linear advection equation
\begin{equation}
\begin{aligned}
  \partial_t u(t, x) + \partial_x u(t, x) &= 0, && t \in (0, T), x \in (\xmin, \xmax),
  \\
  u(0, x) &= u^0(x), && x \in [\xmin, \xmax],
\end{aligned}
\end{equation}
with periodic boundary conditions is
\begin{equation}
  \partial_t \vec{u} + D_- \vec{u} = \vec{0}.
\end{equation}
Following the same steps as in the case of a bounded domain, we see that it
is conservative, i.e.,
\begin{equation}
  \partial_t (\vec{1}^T M \vec{u}) = 0,
\end{equation}
and energy-stable, i.e.,
\begin{equation}
  \partial_t \| \vec{u} \|_M^2 \le 0.
\end{equation}

\begin{example}
  The interior stencils of the second-order accurate upwind operators of
  \cite{mattsson2017diagonal} shown in Example~\ref{ex:upwind-2nd-order}
  yield periodic upwind operators. Specifically, we have
  \begin{equation}
    D_+ = \frac{1}{\dx}
    \begin{pmatrix}
      -3/2 & 2 & -1/2 \\
      & -3/2 & 2 & -1/2 \\
      && -3/2 & 2 & -1/2 \\
      &&& \ddots & \ddots & \ddots \\
      &&&& 3/2 & 2 & -1/2 \\
      -1/2 &&&&& -3/2 & 2 \\
      2 & -1/2 &&&& & -3/2
    \end{pmatrix},
  \end{equation}
  \begin{equation}
    D_- = \frac{1}{\dx}
    \begin{pmatrix}
      3/2 &&&&&& 1/2 & -2 \\
      -2 & 3/2 &&&&&& 1/2 \\
      1/2 & -2 & 3/2 \\
      & 1/2 & -2 & 3/2 \\
      && \ddots & \ddots & \ddots \\
      &&& 1/2 & -2 & 3/2 \\
      &&&& 1/2 & -2 & 3/2 \\
      &&&&& 1/2 & -2 & 3/2
    \end{pmatrix},
  \end{equation}
  and $M = \dx \diag(1, \dots, 1)$.
\end{example}

\section{Formulation of upwind SBP methods for nonlinear problems}
\label{sec:formulations}

Following earlier work on upwind SBP operators \cite{mattsson2017diagonal},
we first apply the flux vector splitting \eqref{eq:flux_vec_split} and
rewrite the hyperbolic conservation law in one space dimension \eqref{eq:HCL-1D} as
\begin{equation}
  \partial_t u + \partial_x f^-(u) + \partial_x f^+(u) = 0.
\end{equation}
Next, we discretize the conservation law in space by using
upwind SBP operators as
\begin{equation}
\label{eq:HCL-1D-upwind-global-periodic}
  \partial_t \vec{u} + D_+ \vec{f^-} + D_- \vec{f^+} = \vec{0}.
\end{equation}
These formulations are well-known in the literature, e.g.,
\cite{mattsson2017diagonal,stiernstrom2021residual,lundgren2020efficient,mattsson2007high,svard2005steady}.
To couple multiple blocks of upwind SBP operators, we introduce
interface terms as in discontinuous Galerkin methods in the following.
On each element, we will use the semidiscretzation
\eqref{eq:HCL-1D-upwind-global-periodic} as a baseline and add additional
terms to couple the elements weakly at the interfaces.
Such a construction has been used for central-type SBP operators in
several works, e.g., \cite{gassner2013skew}.

After describing and contextualizing the method in one space dimension, we
describe and analyze the method in two-dimensional curvilinear coordinates in
Section~\ref{sec:curvilinear}.

\begin{remark}
  The indices $\pm$ of the upwind operators and the fluxes do not match.
  This is due to historical reasons since we want to keep backwards
  compatibility with both the notation of flux vector splitting methods
  \cite[Chapter~8]{toro2009riemann}
  and upwind SBP operators as introduced in \cite{mattsson2017diagonal}.
\end{remark}

\subsection{Local upwind SBP formulation with SATs and numerical fluxes}
\label{sec:local-upwind}

On each element, we consider a discretization of the form
\begin{equation}
\label{eq:HCL-1D-upwind-SAT}
  \partial_t \vec{u} + D_+ \vec{f^-} + D_- \vec{f^+} = \SAT
\end{equation}
with a simultaneous approximation term $\;\SAT$. To motivate the construction
of the SAT, we consider the upwind SBP discretization \eqref{eq:HCL-1D-upwind-SAT}
with the global Lax-Friedrichs flux vector splitting, resulting in
\begin{equation}
  \partial_t \vec{u}
  + \frac{1}{2} D_+ (\vec{f} + \lambda \vec{u})
  + \frac{1}{2} D_- (\vec{f} - \lambda \vec{u}) = \SAT.
\end{equation}
We can formulate this in the central SBP plus dissipation form as
\begin{equation}
  \partial_t \vec{u}
  + \frac{1}{2} (D_- + D_+) \vec{f}
  + \frac{\lambda}{2} (D_+ - D_-) \vec{u}
  = \SAT.
\end{equation}
The second term is the central SBP discretization and the third term is the
artificial dissipation term built into the upwind operators. Thus, we select
the standard SAT for a central SBP discretization, i.e.,
\begin{equation}
  \SAT
  =
  - M^{-1} \tR (\fnum_R - \tR^T \vec{f})
  + M^{-1} \tL (\fnum_L - \tL^T \vec{f}),
\end{equation}
where $\fnum_{L/R}$ is the numerical flux at the left/right interface of the
element. To clarify this notation, we use an upper index to denote the element.
Thus, $\vec{u}^k$ is the numerical solution in element $k$. Then, the SAT
in element $k$ becomes
\begin{equation}
  \SAT^k
  =
  - M^{-1} \tR \left( \fnum(\vec{u}^k_R, \vec{u}^{k+1}_L) - \vec{f}^k_R \right)
  + M^{-1} \tL \left( \fnum(\vec{u}^{k-1}_R, \vec{u}^k_L) - \vec{f}^k_L \right),
\end{equation}
where we have abbreviated the left/right interface value as
$\vec{u}_{L/R} = \vec{t}_{L/R}^T \vec{u}$.

There are many classical numerical fluxes that we can use for $\fnum$.
Next, we use the flux vector splitting to design the numerical fluxes.
We demonstrate this procedure first for the right interface.
Using the same splitting for the numerical flux and the physical flux yields
\begin{equation}
\begin{aligned}
  \label{eq:fnumR}
  \fnum(\vec{u}^k_R, \vec{u}^{k+1}_L) - \vec{f}^k_R
  &=
  \left( f^+(\vec{u}^k_R) + f^-(\vec{u}^{k+1}_L) \right)
  - \left( f^+(\vec{u}^k_R) + f^-(\vec{u}^k_R) \right)
  \\
  &=
  f^-(\vec{u}^{k+1}_L) - f^-(\vec{u}^k_R).
\end{aligned}
\end{equation}
Similarly, we get
\begin{equation}
  \label{eq:fnumL}
\begin{aligned}
  \fnum(\vec{u}^{k-1}_R, \vec{u}^k_L) - \vec{f}^k_L
  &=
  \left( f^+(\vec{u}^{k-1}_R) + f^-(\vec{u}^k_L) \right)
  - \left( f^+(\vec{u}^k_L) + f^-(\vec{u}^k_L) \right)
  \\
  &=
  f^+(\vec{u}^{k-1}_R) - f^+(\vec{u}^k_L).
\end{aligned}
\end{equation}
for the left interface. Thus, the SAT becomes
\begin{equation}
  \SAT^k
  =
  - M^{-1} \tR \left( f^-(\vec{u}^{k+1}_L) - f^-(\vec{u}^k_R) \right)
  + M^{-1} \tL \left( f^+(\vec{u}^{k-1}_R) - f^+(\vec{u}^k_L) \right).
\end{equation}
To sum up, we arrive at the upwind SBP discretization
\begin{equation}
  \partial_t \vec{u} + D_+ \vec{f^-} + D_- \vec{f^+} = \SAT^k,
\end{equation}
where the simultaneous approximation term can be expressed using
(general) numerical fluxes as
\begin{equation}
\label{eq:SAT-fnum}
  \SAT^k
  =
  - M^{-1} \tR \left( \fnum(\vec{u}^k_R, \vec{u}^{k+1}_L) - \vec{f}^k_R \right)
  + M^{-1} \tL \left( \fnum(\vec{u}^{k-1}_R, \vec{u}^k_L) - \vec{f}^k_L \right)
\end{equation}
or specifically using the upwind fluxes as
\begin{equation}
\label{eq:SAT-upwind}
  \SAT^k
  =
  - M^{-1} \tR \left( f^-(\vec{u}^{k+1}_L) - f^-(\vec{u}^k_R) \right)
  + M^{-1} \tL \left( f^+(\vec{u}^{k-1}_R) - f^+(\vec{u}^k_L) \right).
\end{equation}

Finally, we can integrate the semidiscretization in time using any suitable
time integration scheme, e.g., Runge-Kutta methods.

\begin{remark}
  If the corresponding upwind flux is used at interfaces as in
  \eqref{eq:SAT-upwind},
  the final discretization of the hyperbolic conservation law \eqref{eq:HCL-1D}
  is actually agnostic to the flux and does not require the solution to a
  Riemann problem. All the physics is contained in the particular flux
  splitting one considers.
\end{remark}

\subsection{Global upwind SBP formulation}
\label{sec:global-upwind}

There is another formulation of the method above that is useful as an
interpretation. Instead of first introducing an upwind SBP discretization
on each element and coupling terms in a second step, we can directly
couple the element-local upwind operators to obtain a global upwind operator
as described in \cite{ranocha2021broad}. Here, we just concentrate on a
coupling as in DG methods.

\begin{theorem}[Theorem~2.2 of \cite{ranocha2021broad}]
\label{thm:upwind-DG}
  Consider two upwind SBP operators $D_{\pm, l/r}$
  on the grids
  $\vec{x}_{l/r}$ with $\vec{x}_{N_l,l} = \vec{x}_{1,r}$. Then,
  \begin{equation}
  \label{eq:D1-upwind-DG}
  \begin{gathered}
    D_{+} =
    \begin{pmatrix}
      D_{+,l} - M_{l}^{-1} \vec{t}_{R,l} \vec{t}_{R,l}^T &
      M_{l}^{-1} \vec{t}_{R,l} \vec{t}_{L,r}^T \\
      0 &
      D_{+,r}
    \end{pmatrix},
    \quad
    D_{-} =
    \begin{pmatrix}
      D_{-,l} &
      0 \\
      - M_{r}^{-1} \vec{t}_{L,r} \vec{t}_{R,l}^T &
      D_{-,r} + M_{r}^{-1} \vec{t}_{L,r} \vec{t}_{L,r}^T
    \end{pmatrix},
    \\
    M =
    \begin{pmatrix}
      M_{l} & 0 \\
      0 & M_{r}
    \end{pmatrix},
  \end{gathered}
  \end{equation}
  yield upwind SBP operators on the joint grid $\vec{x} =
  (\vec{x}_{1,l}, \dots, \vec{x}_{N_l,l}, \vec{x}_{1,r}, \dots, \vec{x}_{N_r,r})^T$
  with $N = N_l + N_r$ nodes.
  These global operators have the same order of accuracy as the less accurate one
  of the given local operators.
\end{theorem}

The global upwind SBP operators described in Theorem~\ref{thm:upwind-DG} are
obtained by taking upwind numerical fluxes in a DG-type discretization.
Indeed, consider the discretization of two elements and their shared interface
written using the coupled upwind operators of Theorem~\ref{thm:upwind-DG}.
We have
\begin{equation}
  \partial_t \begin{pmatrix} \vec{u}_{l} \\ \vec{u}_{r} \end{pmatrix}
  + D_+
    \begin{pmatrix} \vec{f}^-_{l} \\ \vec{f}^-_{r} \end{pmatrix}
  + D_-
    \begin{pmatrix} \vec{f}^+_{l} \\ \vec{f}^+_{r} \end{pmatrix}
  = \BTS,
\end{equation}
where $\BTS$ collects the surface terms of their non-shared interfaces.
For the left element, we get
\begin{equation}
  \partial_t \vec{u}_{l}
  + (D_{+,l} - M_{l}^{-1} \vec{t}_{R,l} \vec{t}_{R,l}^T) \vec{f}^-_{l}
  + M_{l}^{-1} \vec{t}_{R,l} \vec{t}_{L,r}^T \vec{f}^-_{r}
  + D_{-,l} \vec{f}^+_{l}
  =
  \BTS_{l}.
\end{equation}
Replacing the index $l$ by the element number $k$ leads to
\begin{equation}
  \partial_t \vec{u}^k
  + D_{+} \vec{f^-}
  + D_{-} \vec{f^+}
  =
  - M^{-1} \tR \left( f^-(\vec{u}^{k+1}_L) - f^-(\vec{u}^k_R) \right)
  + \BTS^{k}.
\end{equation}
Thus, the interface term is identical to the SAT \eqref{eq:SAT-upwind}
using the upwind numerical flux coming from the flux vector splitting.

In particular, the discontinuous Galerkin spectral element method (DGSEM)
with upwind flux for the linear advection equation yields
an upwind SBP operator. Indeed, the local operators used on each element with
Gauss-Lobatto-Legendre nodes are classical SBP operators \cite{gassner2013skew}
and the upwind (Godunov) flux yields exactly the interface coupling described
in Theorem~\ref{thm:upwind-DG}.

\begin{example}
  Consider the DGSEM with polynomials of degree $p = 2$ and two elements
  in the domain $[0, 2]$. The corresponding nodes are
  \begin{equation}
    \vec{x}_{l} = (0, 1/2, 1)^T,
    \qquad
    \vec{x}_{r} = (1, 3/2, 2)^T.
  \end{equation}
  The polynomial derivative matrix $D$ and the mass matrix $M$ on each element
  with length unity are given by
  \begin{equation}
    D =
    \begin{pmatrix}
      -3 & 4  & -1 \\
      -1 & 0 & 1 \\
      1 & -4 & 3
    \end{pmatrix},
    \qquad
    M =
    \begin{pmatrix}
      1/6 \\
      & 2/3 \\
      && 1/6
    \end{pmatrix}.
  \end{equation}
  These matrices satisfy $M D + D^T M = \operatorname{diag}(-1, 0, 1)$, i.e., the
  SBP property \eqref{eq:SBP}.
  The construction in Theorem~\ref{thm:upwind-DG} yields the global operators
  \begin{equation}
    D_+ =
    \begin{pmatrix}
      -3 & 4 & -1 \\
      -1 & 0 & 1 \\
      1 & -4 & -3 & 6 \\
      &&& -3 & 4 & -1 \\
      &&& -1 & 0 & 1 \\
      &&& 1 & -4 & 3
    \end{pmatrix},
    \qquad
    D_- =
    \begin{pmatrix}
      -3 & 4 & -1 \\
      -1 & 0 & 1 \\
      1 & -4 & 3 \\
      && -6 & 3 & 4 & -1 \\
      &&& -1 & 0 & 1 \\
      &&& 1 & -4 & 3
    \end{pmatrix},
  \end{equation}
  and $M = \operatorname{diag}(1/6, 2/3, 1/6, 1/6, 2/3, 1/6)$.
  These operators satisfy
  \begin{equation}
    M D_+ + D_-^T M = \operatorname{diag}(-1, 0, 0, 0, 0, 1),
    \quad
    M (D_+ - D_-) =
    \begin{pmatrix}
      0 & 0 & 0 & 0 & 0 & 0 \\
      0 & 0 & 0 & 0 & 0 & 0 \\
      0 & 0 & -1 & 1 & 0 & 0 \\
      0 & 0 & 1 & -1 & 0 & 0 \\
      0 & 0 & 0 & 0 & 0 & 0 \\
      0 & 0 & 0 & 0 & 0 & 0
    \end{pmatrix}.
  \end{equation}
  The eigenvalues of the symmetric matrix $M (D_+ - D_-)$ are
  zero (with multiplicity five) and $-2$ (with multiplicity one).
  Thus, it is symmetric and negative semidefinite.
  Hence, the defining property \eqref{eq:upwind-SBP} of upwind SBP operators
  is satisfied.
\end{example}

\begin{remark}
  This also holds on periodic domains. Indeed, coupling upwind SBP operators
  as in Theorem~\ref{thm:upwind-DG} on all interfaces results in periodic
  upwind SBP operators.
\end{remark}

\subsection{Classical flux vector splitting using upwind SBP operators}

There are no first-order accurate upwind SBP operators in
\cite{mattsson2017diagonal}. However, we can construct such finite
difference operators as
\begin{equation}
\begin{gathered}
  D_- =
  \frac{1}{\Delta x}
  \begin{pmatrix}
    0  & 0 \\
    -1 &  1 \\
       & \ddots & \ddots \\
       &        & -1     & 1 \\
       &        &        & -2 & 2
  \end{pmatrix},
  \quad
  D_+ =
  \frac{1}{\Delta x}
  \begin{pmatrix}
    -2 &  2 \\
       & -1 & 1 \\
       &    & \ddots & \ddots \\
       &    &        &   -1   & 1 \\
       &    &        &    0   & 0
  \end{pmatrix},
  \\
  M = \Delta x \diag(1/2, 1, \dots, 1, 1/2).
\end{gathered}
\end{equation}
Indeed,
\begin{equation}
    M D_+ + D_-^T M = \diag(-1, 0, \dots, 0, 1)
\end{equation}
and
\begin{equation}
   \label{eq:diss_matrix}
    M (D_+ - D_-) =
    \begin{pmatrix}
    -1 &    1 \\
     1 &   -2   &    1 \\
       & \ddots & \ddots & \ddots \\
       &        &    1   &   -2   &  1 \\
       &        &        &    1   & -1
    \end{pmatrix}
\end{equation}
is negative semidefinite. In fact, \eqref{eq:diss_matrix} is the classical finite difference
discretization of the Laplacian with homogeneous Neumann boundary
conditions.
Note that the order of accuracy of $D_{\pm}$ at one of the boundaries is
reduced to zero, in accordance with the general order reduction of SBP
operators.

An upwind SBP semidiscretization of a conservation law
$\partial_t u + \partial_x f(u) = 0$
is the scheme
\begin{equation}
  \partial_t \vec{u} + D_+ \vec{f}^- + D_- \vec{f}^+ = \vec{\mathrm{SATs}},
\end{equation}
where $\vec{\mathrm{SATs}}$ are boundary terms used to impose the boundary conditions.
Applying the upwind SBP operators shown above in such a discretization
results in the classical first-order flux vector splitting
\begin{equation}
\begin{aligned}
  \partial_t u_i
  &=
  -\frac{1}{\Delta x} \left( f_{i+\frac{1}{2}}^* - f_{i-\frac{1}{2}}^* \right)
  \\
  &=
  -\frac{1}{\Delta x} \left( (f_{i}^+ + f_{i+1}^-) - (f_{i-1}^+ + f_{i}^-) \right)
  \\
  &=
  -\frac{1}{\Delta x} \left( (f_{i}^+ - f_{i-1}^+) + (f_{i+1}^- - f_{i}^-) \right)
\end{aligned}
\end{equation}
in the interior. Thus, high-order upwind SBP methods can be seen as
extensions of the classical first-order flux vector splitting methods.

\subsection{Formulation in two-dimensional curvilinear coordinates}
\label{sec:curvilinear}
The generic conservation law in two dimensions takes the form
\begin{equation}
\label{eq:HCL-2D}
  \partial_t u(t, x,y) + \partial_x f_1\bigl( u(t, x,y) \bigr) + \partial_y f_2\bigl( u(t, x,y) \bigr) = 0,
  \quad t \in (0, T), \quad (x,y) \in \Omega \subset \R^2,
\end{equation}
with conserved variable $u$ and fluxes $f_1$, $f_2$ in each coordinate direction,
equipped with appropriate initial and boundary conditions. We first subdivide
the problem domain $\Omega$ into $K$ non-overlapping
quadrilateral elements $E_k$, $k = 1,\ldots, K$. In the following, we consider
the conservation law \eqref{eq:HCL-2D} on an individual element and suppress
the index $k$.

Next, we create a transformation on each element $E_k$ between the
computational coordinates $(\xi,\eta) \in E_0$ where $E_0 = [-1,1]^2$
is the reference element and the physical coordinates $(x,y)$ as
\begin{equation}
\label{eq:mapping}
x = {X}(\xi,\eta), \quad y = {Y}(\xi,\eta).
\end{equation}
Typically, this mapping is a linear blending transfinite map between
the opposing sides of an element \cite{gordon1973construction,kopriva2006metric}.
When the element sides are straight, the mapping \eqref{eq:mapping}
is linear in each coordinate direction. However, if the sides are curved and
high-order polynomials are used to approximate the element boundaries, then
the mapping \eqref{eq:mapping} is a polynomial in each direction. In that case,
we represent the mapping as a polynomial of degree $\Ngeo$ in each coordinate direction.

Under this transformation, the conservation law in physical coordinates remains
a conservation law in reference coordinates, see, e.g., \cite{rydin2018high,kopriva2006metric}
\begin{equation}
\label{eq:transformed_cons}
J \partial_t u(t, {\xi}, \eta) + \partial_\xi \tilde{f}_1\bigl( u(t, {\xi}, \eta) \bigr) + \partial_\eta \tilde{f}_2\bigl( u(t, {\xi}, \eta) \bigr) = 0,
\end{equation}
where the contravariant fluxes $\tilde{f}$ and Jacobian $J$ for the two dimensional transformation
are
\begin{equation}
\label{eq:contra_vecs}
\tilde{f}_1 = Y_\eta f_1 - X_\eta f_2, \quad \tilde{f}_2 = -Y_\xi f_1 + X_\xi f_2, \quad J = Y_\eta X_\xi - Y_\xi X_\eta.
\end{equation}
For convenience we introduce a compact notation for the flux in the contravariant (or normal)
direction. The normal direction (but not normalized) vectors in reference space are written as
\begin{equation}
\label{eq:normals}
\hat{n}^1 = (Y_\eta, -X_\eta)^T \quad \textrm{and} \quad \hat{n}^2 = (-Y_\xi, X_\xi)^T,
\end{equation}
where $X_\xi, X_\eta, Y_\xi, Y_\eta$ are the metric terms. So, for example, the first contravariant
flux is given by $\tilde{f}_1 = f_1\hat{n}^1_1 + f_2\hat{n}^1_2$.
Additionally, the metric terms satisfy two metric identities
\begin{equation}
\label{eq:metricIDs}
\partial_\xi Y_\eta - \partial_\eta Y_\xi = 0 \quad \textrm{and} \quad -\partial_\xi X_\eta + \partial_\eta X_\xi = 0
\end{equation}
that are crucial to guarantee free-stream preservation (FSP) \cite{vinokur2002extension,kopriva2006metric,visbal1999high}.
That is, given a flux that is constant in space, its divergence vanishes and the (constant) solution of \eqref{eq:HCL-2D}
does not change in time. We will revisit the recovery of FSP on the discrete level later in this section.

From the mapped conservation law \eqref{eq:transformed_cons} the next step is to perform a flux vector splitting.
However, it is important to note that one cannot simply multiply the Cartesian flux vector splittings with the metric
terms to create their curvilinear counterparts. This would lead to inconsistencies with respect to the directionality
of the waves in the considered splitting. Instead, one follows a procedure of rotation into a generalized coordinate's
normal direction, performing the flux vector splitting, and back-rotating the result,
see \cite{anderson1986comparison,bolcs1989} for complete details. This process guarantees
that the flux vector splittings satisfy the following relationship in each contravariant direction
\begin{equation}
\label{eq:generic_fvs}
\tilde{f}_i(u) = f_1(u) \hat{n}^i_1 + f_2(u)\hat{n}^i_2 = {f}^+(u; \hat{n}^i) + {f}^-(u; \hat{n}^i) = \tilde{f}^+_i(u) + \tilde{f}^-_i(u),
 \quad i = 1,2.
\end{equation}
We introduce the notation $\tilde{f}^\pm(u; \hat{n}^i)$ to highlight that the normal direction components
can no longer be factored out of the flux vector splitting and different flux components may depend
on the normal direction in different ways.
To clarify the form of the flux vector splittings in generalized coordinate directions we highlight three examples
for the compressible Euler equations.
\begin{example}
\label{ex:llf_splitting}
  The local Lax-Friedrichs splitting in the contravariant directions uses a local estimate for the
  largest value of $\lambda$ for the possible wave speeds and has the form
  \begin{equation}
    \label{eq:llf_splitting}
    \tilde{f}^\pm_i = {f}^\pm(u; \hat{n}^i) = \frac{1}{2} \left( \tilde{f}_i(u) \pm \lambda u \right),\quad i = 1,2,
  \end{equation}
  where $u = (\rho, \rho v_1, \rho v_2, \rho e)^T$ and $\lambda = \sqrt{v_1^2 + v_2^2} + a$.
  As discussed above, \eqref{eq:llf_splitting} has a linear dependency on the mapping terms; however they
  cannot be factored out to separate the Cartesian splitting from the normal directions as was the case for
  the complete physical flux \eqref{eq:contra_vecs}.
\end{example}
\begin{example}
\label{ex:drikakis}
  We describe an improved variant of the Steger-Warming splitting (Example \ref{ex:Steger-Warming}) for
  generalized coordinates due to Drikakis and Tsangaris \cite{drikakis1993solution}. We, again, use the
  standard notation for the positive/negative part of an eigenvalue $\lambda_i$ with
  \begin{equation}
     \lambda^\pm_i = \frac{\lambda_i + |\lambda_i |}{2}.
  \end{equation}
  The wave speeds in the normal direction $\hat{n}^i$ used by this splitting are
  \begin{equation}
     \tilde{\lambda}_1 = v_1\hat{n}^i_1 + v_2\hat{n}^i_2 - a,
     \qquad
     \tilde{\lambda}_2 = v_1\hat{n}^i_1 + v_2\hat{n}^i_2 + a,
  \end{equation}
  with the sound speed $a = \sqrt{\gamma p/\rho}$. The flux vector splitting of Drikakis and Tsangaris is given by
  \begin{equation}
  \label{eq:drikakis_split}
     \tilde{f}^\pm_i = {f}^\pm(u; \hat{n}^i) =
     \frac{\rho}{2}
     \begin{pmatrix}
     \tilde{\lambda}^\pm_1 + \tilde{\lambda}^\pm_2\\[0.1cm]
     (\tilde{\lambda}^\pm_1 + \tilde{\lambda}^\pm_2)v_1 + \frac{a \hat{n}_1^i}{\gamma}(\tilde{\lambda}^\pm_2 - \tilde{\lambda}^\pm_1) \\[0.1cm]
     (\tilde{\lambda}^\pm_1 + \tilde{\lambda}^\pm_2)v_2 + \frac{a \hat{n}_2^i}{\gamma}(\tilde{\lambda}^\pm_2 - \tilde{\lambda}^\pm_1) \\[0.1cm]
     (\tilde{\lambda}^\pm_1 + \tilde{\lambda}^\pm_2)H
     \end{pmatrix},
  \end{equation}
  where $H = (\rho e + p) / \rho = v^2 / 2 + a^2 / (\gamma - 1)$
  is again the enthalpy.
\end{example}
\begin{example}
\label{ex:van-Leer-Hanel_curved}
  As a last example, we describe the van Leer-Hänel splitting
  \cite{vanleer1982flux,hanel1987accuracy,liou1991high} rotated into
  a contravariant normal direction \cite{anderson1986comparison}.
  We introduce the signed Mach number in the normal direction
  \begin{equation}\label{eq:signedMach}
  \tilde{M} = \frac{v_1\hat{n}^i_1 + v_2 \hat{n}^i_2}{a}, \quad i = 1,2,
  \end{equation}
  and the pressure splitting
  \begin{equation}
     \label{eq:pressure_vLH}
    p^\pm = \frac{1 \pm \gamma \tilde{M}}{2} p.
  \end{equation}
  The flux splittings are then given by
  \begin{equation}
  \label{eq:vL_splitting}
    \tilde{f}^\pm_i
    =
    {f}^\pm(u; \hat{n}^i)
    =
    \pm \frac{\rho a (\tilde{M} \pm 1)^2}{4}
    \begin{pmatrix}
      1 \\
      v_1 \\
      v_2 \\
      H
    \end{pmatrix}
    +
    \begin{pmatrix}
      0 \\
      \hat{n}^i_1 p^\pm \\
      \hat{n}^i_2 p^\pm \\
      0
    \end{pmatrix},
  \end{equation}
  where $H$ is the enthalpy.
\end{example}
\begin{remark}
\label{rem:metrics}
Notice, as the flux vector splittings become more sophisticated their dependency on the normal direction
(and in turn the metric terms) increase in complexity as well. That is, the local Lax-Friedrichs splitting
\eqref{eq:llf_splitting} is linear in the metric terms, the Drikakis-Tsangaris splitting \eqref{eq:drikakis_split}
is linear in the advective components and quadratic in the metric terms for the pressure splitting, and
the van Leer-Hänel splitting \eqref{eq:vL_splitting} is quadratic in the metric terms in all components.
\end{remark}

With appropriate flux vector splittings that satisfy \eqref{eq:generic_fvs} in hand, we split the contravariant
fluxes in mapped conservation law on each element \eqref{eq:transformed_cons} to have
\begin{equation}
\label{eq:transformed_split}
J \partial_t u + \partial_\xi \tilde{f}^+_1 + \partial_\xi \tilde{f}^-_1
+ \partial_\eta \tilde{f}^+_2 + \partial_\eta \tilde{f}^-_2 = 0.
\end{equation}
Just as in Section~\ref{sec:local-upwind}, we discretize the mapped conservation law in space
with upwind SBP operators and couple the element to its neighbor elements with
appropriate $\vec{\mathrm{SATs}}$. These $\vec{\mathrm{SATs}}$ have the same form
as given in \eqref{eq:SAT-fnum} along each of the four element interfaces. As before, if the same splitting
in the normal direction is used for the numerical flux as the physical flux, then we recover
analogous statements to \eqref{eq:fnumR} and \eqref{eq:fnumL}. We then have a generic statement of a
$\SAT$ in the normal direction on an interface in element $k$
\begin{equation}
\label{eq:transformed_SATs}
\widetilde{\SAT}^k
=
  - M^{-1} \tR \left( \fnum(\vec{u}^k_R, \vec{u}^{k+1}_L; \hat{n}^i) - \tilde{\vec{f}}^k_R \right)
  + M^{-1} \tL \left( \fnum(\vec{u}^{k-1}_R, \vec{u}^k_L; \hat{n}^i) - \tilde{\vec{f}}^k_L \right).
\end{equation}
The resulting upwind SBP discretization on element $k$ takes the form
\begin{equation}
\label{eq:transformed_split}
\vec{J} \partial_t \vec{u} + D_- \vec{\tilde{f}}^+_1 + D_+ \vec{\tilde{f}}^-_1
+ \vec{\tilde{f}}^+_2 D_-^T + \vec{\tilde{f}}^-_2 D_+^T = \widetilde{\vec{\mathrm{SATs}}}^k.
\end{equation}
Compared to the one-dimensional case described in Section~\ref{sec:local-upwind},
we use a tensor product structure for the mapped conservation law
\eqref{eq:transformed_cons} on each quadrilateral element. Thus, we can think of the
discrete solution $\vec{u}$ as a two-dimensional array of size
$N_\xi \times N_\eta$, where $N_\xi$ and $N_\eta$ are the number of
grid points in the $\xi$ and $\eta$ directions, respectively. At each
grid point with index $(i, j)$, $\vec{u}_{i,j}$ is the vector of conserved variables
at this point. With this data layout, multiplying the mapped fluxes
$\vec{\tilde{f}}^{\pm}_1$ from the left by the upwind SBP operator $D_{\mp}$
approximates the derivative in the $\xi$-direction. Similarly, multiplying
the mapped fluxes $\vec{\tilde{f}}^{\pm}_2$ from the right by the upwind SBP
operator $D_{\mp}^T$ approximates the derivative in the $\eta$-direction.
Due to the tensor product structure for the mapped conservation law
\eqref{eq:transformed_cons}, the $\widetilde{\vec{\mathrm{SAT}}}$ \eqref{eq:transformed_SATs} in the
normal directions are computed in a similar way as in the one-dimensional
case.

The final component to fully describe the upwind SBP method on curvilinear domains \eqref{eq:transformed_split}
is to discuss how the metric terms are approximated. By design, from \cite{mattsson2017diagonal},
the upwind SBP operators $D_{\pm}$, as well as the central SBP operator
$(D_+ + D_-) / 2$
they generate, have $p^{\textrm{th}}$ order accurate interior stencils
and $p/2$ order accurate boundary stencils. This boundary closure means that any of the
three available differencing operators can differentiate polynomials up to degree $p/2$ exactly.
For instance, one available upwind SBP operator is the fourth-order interior, second-order boundary
closure, denoted 4-2, operator where $D_{\pm}$ or $D=(D_+ + D_-)/2$ can differentiate
up to quadratic polynomials exactly.

Because all available upwind SBP operators of a given order have the same boundary closure
accuracy, we use the central operator $D=(D_+ + D_-)/2$ to compute the metric terms by
directly differentiating the mapping $X(\xi,\eta)$ from \eqref{eq:mapping}, i.e.,
\begin{equation}
\label{eq:discrete_metrics}
X_\xi \approx D \vec{X} = \vec{X_\xi}, \quad X_\eta \approx \vec{X} D^T = \vec{X_\eta}, \quad
Y_\xi \approx D \vec{Y} = \vec{Y_\xi}, \quad Y_\eta \approx \vec{Y} D^T = \vec{Y_\eta}.
\end{equation}
We note, depending on the strategy used to compute the discrete metric terms, one may or
may not recover a discrete equivalent of the metric identities
\eqref{eq:metricIDs}. That is, it is possible to lose discrete FSP
\cite{kopriva2006metric,vinokur2002extension,visbal1999high}.
Applying the approximation strategy from \eqref{eq:discrete_metrics}, we examine the discrete
version of the metric identities \eqref{eq:metricIDs} to find
 \begin{equation}
 \label{eq:discrete_metricIDs}
D \vec{Y_\eta} - \vec{Y_\xi}D^T = D \vec{Y} D^T - D \vec{Y} D^T = \vec{0}
\quad \textrm{and} \quad
-D \vec{X_\eta} + \vec{X_\xi}D^T = -D \vec{X} D^T + D \vec{X} D^T = \vec{0}.
 \end{equation}
Thus, the metric identities hold discretely as has been shown previously for
finite difference methods in two-dimensional curvilinear coordinates, e.g.,
\cite{vinokur2002extension,visbal1999high,aalund2019encapsulated}.
Moreover, the result \eqref{eq:discrete_metricIDs}
actually holds independently of the boundary closure accuracy or the
polynomial degree $\Ngeo$ of the mapping \eqref{eq:mapping}.
The result that a central SBP finite difference method is free-stream preserving, via the discrete
metric identities \eqref{eq:discrete_metricIDs}, is directly related to the fact that the contravariant
fluxes \eqref{eq:contra_vecs} have a linear dependency on the metric terms. However, as discussed
in Remark~\ref{rem:metrics}, this is not always the case for a splitting in curvilinear
coordinates.

For more sophisticated splittings, like that of van Leer-Hänel in
Example~\ref{ex:van-Leer-Hanel_curved}, the issue of FSP becomes more subtle.
There is a delicate interplay between the dependency
a given splitting has with respect to the metric terms, the boundary closure order of the upwind SBP
operator, and the polynomial degree $\Ngeo$ of the mapping.
We collect the implications of this interplay into the following theorem.
\begin{theorem}[FSP for the curvilinear upwind SBP method]
\label{thm:curvi_fsp}
Consider a flux vector splitting that has a maximum dependency on the metric terms of degree $m$,
a set of upwind SBP operators with $p^{\textrm{th}}$ order interior stencils, and mappings $X(\xi,\eta)$
and $Y(\xi,\eta)$ with polynomial degree of $\Ngeo$ in each coordinate direction. The curvilinear upwind
SBP method \eqref{eq:transformed_split} is free-stream preserving when either
\begin{enumerate}
   \item $m=1$, i.e, there is a linear dependence on the metric terms, or
   \item $m>1$ and the polynomial degree of the mapping satisfies
   \begin{equation}
      \label{eq:boundary_constraint}
      \Ngeo \leq \frac{p}{2m}.
   \end{equation}
   That is, the boundary closure can exactly differentiate polynomials up to degree $m \Ngeo \le p / 2$.
\end{enumerate}
\end{theorem}
\begin{proof}
Assume we have a constant solution $u_\infty$ for the conservation law \eqref{eq:HCL-2D}.
The physical variable terms in the flux vector are then also constants.

\underline{Part 1 ($m=1$)}: By construction, the metric terms in the approximation satisfy the discrete metric
identities \eqref{eq:discrete_metricIDs}. Therefore, the flux splitting terms with a linear dependence
on the metric terms vanish from the same reasoning that the standard, non-split curvilinear flux formulation vanishes.

\underline{Part 2 ($m>1$)}: To guarantee that the curvilinear divergence of the upwind SBP scheme
\eqref{eq:transformed_split} vanishes, it is sufficient if the following terms individually vanish
\begin{equation}
\label{eq:vanishing_terms}
D_-\vec{Y_\eta}^m - \vec{Y_\xi}^m D_-^T,
\quad
D_+\vec{Y_\eta}^m - \vec{Y_\xi}^m D_+^T,
\quad
-D_-\vec{X_\eta}^m + \vec{X_\xi}^m D_-^T,
\quad
-D_+\vec{X_\eta}^m + \vec{X_\xi}^m D_+^T.
\end{equation}
The four terms above are similar to the discrete metric identities, but the metric terms are now
polynomials of higher degree.
To clarify their appearance, consider the van~Leer-Hänel splitting
\eqref{eq:vL_splitting} of Example~\ref{ex:van-Leer-Hanel_curved}.
The signed Mach number $\tilde M$ from \eqref{eq:signedMach} depends linearly on the
rotated normal direction $\hat n$, i.e., linearly on the metric terms.
Thus, the split pressure $p^\pm$ from \eqref{eq:pressure_vLH} depends linearly on the metric terms
as well. Hence, the van-Leer-H\"{a}nel splitting \eqref{eq:vL_splitting} has a quadratic dependency
on the metric terms because the advective components are scaled with $(\tilde M \pm 1)^2$
and the pressure components have the form $\hat n p^{\pm}$.
As such, for a constant solution $u_\infty$, the fluxes for the van-Leer-H\"{a}nel splitting are
constant physical variables multiplied with expressions that depend quadratically on
the (non-constant) metric terms, i.e., $m = 2$.

Consider the first term in \eqref{eq:vanishing_terms}, i.e., $D_-\vec{Y_\eta}^m - \vec{Y_\xi}^m D_-^T$.
From the constraint on the polynomial degree of the mapping \eqref{eq:boundary_constraint}
we know that the boundary closure order of the upwind SBP operators can exactly differentiate
polynomials up to degree $m \Ngeo \le p / 2$. The discrete metric term $\vec{Y_\eta} = \vec{Y}D^T$
is a polynomial of degree $\Ngeo$ in the $\xi$-direction and degree $\Ngeo-1$ in the $\eta$-direction
and the discrete metric term $\vec{Y_\xi} = D \vec{Y}$ is a polynomial of degree $\Ngeo - 1$ in
the $\xi$-direction and degree $\Ngeo$ in the $\eta$-direction. Taking these metric terms to the power
$m$ means that $\vec{Y_\eta}^m$ is a polynomial of degree $m\Ngeo$ in the $\xi$-direction and degree
$m(\Ngeo-1)$ in the $\eta$-direction and $\vec{Y_\xi}^m$ is a polynomial of degree $m(\Ngeo-1)$
in the $\xi$-direction and degree $m\Ngeo$ in the $\eta$-direction.
By design, the upwind SBP derivative operators $D_{\pm}$ and $D = (D_+ + D_-)/2$ all have the
same boundary order closure and can differentiate polynomials up to degree $p/2$ exactly.
Therefore, under the constraint \eqref{eq:boundary_constraint} the $D_-$ operator can exactly
differentiate the term $\vec{Y_\eta}^m$ in the $\xi$- direction as it is a polynomial of degree $m\Ngeo$.
Similarly, the $D_-$ operator can exactly differentiate the term $\vec{Y_\xi}^m$ in the $\eta$- direction
as it is also a polynomial of degree $m\Ngeo$. Thus,
\begin{equation}
D_-\vec{Y_\eta}^m - \vec{Y_\xi}^m D_-^T = \vec{0},
\end{equation}
due to the exactness of polynomial differentiation of the boundary closure. The remaining three
terms from \eqref{eq:vanishing_terms} individually vanish from a similar argument.
Therefore, the curvilinear upwind SBP method is FSP.
\end{proof}

The result of Theorem~\ref{thm:curvi_fsp} is two-fold.
If the dependency of the curvilinear flux vector splitting on the metric terms
remains linear, i.e. $m=1$, then the curvilinear upwind SBP method
\eqref{eq:transformed_split} retains discrete FSP regardless of the boundary closure
order and polynomial degree of the mapping $\Ngeo$. This is the case for the
local Lax-Friedrichs splitting in Example~\ref{ex:llf_splitting} and, as shown in
Section~\ref{sec:fsp_numerics}, this splitting retains discrete FSP for all considered
upwind SBP operators and meshes. These results agree with the previous work
of Rydin et al.~\cite{rydin2018high}, where the global Lax-Friedrichs splitting was used
and there were no reported spurious wave artifacts due to curved elements.
However, if the curvilinear flux vector splitting has a higher degree polynomial dependence
on the metric terms it places a cap on the polynomial degree of curvilinear elements.
This is the case for the Drikakis-Tsangaris and van Leer-Hänel splittings of the compressible
Euler equations given in Examples~\ref{ex:drikakis}
and \ref{ex:van-Leer-Hanel_curved}, respectively. For both curvilinear splittings, the maximum
dependency on the metric terms is quadratic, so $m=2$. Thus, upwind SBP operators with a boundary
closure order of two are restricted to unstructured bi-linear element meshes and operators with a boundary
closure order of four are restricted to at most $\Ngeo=2$ or quadratic polynomial boundaries.
If curvilinear meshes are constructed with boundaries beyond these values for $\Ngeo$, then
the method is not FSP. We numerically examine the FSP properties of the high-order, curvilinear upwind SBP
method \eqref{eq:transformed_split} with different splittings, boundary closures, and mesh polynomial degrees
in Section~\ref{sec:fsp_numerics}.
\begin{remark}
As a word of caution, one must take care when adapting a flux vector splitting into the curvilinear high-order
upwind SBP context. For instance, the pressure for the van Leer-Hänel splitting \eqref{eq:pressure_vLH}
considered herein is linear with respect to the signed Mach number in the normal direction.
This means that the pressure term is quadratic in the
metric terms as the pressure splitting is multiplied with the components of the normal direction vector $\hat{n}^i$.
Other pressure splittings are proposed by Liou and Steffen~\cite{liou1991high}, e.g.,
\begin{equation}
p^\pm = \frac{1}{4}(\tilde{M} \pm 1)^2(2 \mp \tilde{M})p,
\end{equation}
which is cubic in the signed Mach number and, in turn, cubic with respect to the metric terms. This means that
overall, the van Leer-Hänel splitting with the above pressure splitting has $m=4$,
as the pressure components are further multiplied with $\hat n$.
Thus, bi-linear element
meshes with $\Ngeo=1$ require at least fourth order boundary closures to guarantee FSP of the
upwind SBP method.
\end{remark}
\begin{remark}
The constraint on the boundary polynomial degree \eqref{eq:boundary_constraint} is similar to the constraint
found by Kopriva~\cite[Theorem 4]{kopriva2006metric} for three-dimensional cross-product discrete metric terms.
\end{remark}

\begin{remark}
\label{rem:curvilinear-sufficient}
  The conditions given in Theorem~\ref{thm:curvi_fsp} are sufficient
  to guarantee discrete FSP. They may not be necessary in all cases. However,
  the numerical results in Section~\ref{sec:fsp_numerics} suggest that
  the result is sharp for the considered flux vector splittings, nodal polynomial approximation of curved boundaries,
  and tensor product extension to curvilinear coordinates.
\end{remark}

\section{Analysis of local linear/energy stability}
\label{sec:stability}

From the review and discussion of upwind SBP methods for nonlinear
conservation laws, we now turn to one goal of this article:
the analysis of local linear/energy stability properties.
We follow \cite{gassner2022stability} and consider local linear/energy
stability for Burgers' equation
\begin{equation}
  \partial_t u(t, x) + \partial_x \frac{u(t, x)^2}{2} = 0
\end{equation}
with periodic boundary conditions. Thus, we linearize the equation around
a baseflow $\widetilde{u}$, write $u = \widetilde{u} + v$, and get
\begin{equation}
  \partial_t v + \partial_x (\widetilde{u} v) = 0.
\end{equation}
This is a linear advection equation for the perturbation $v$ with variable
coefficient $\widetilde{u}$. For a positive baseflow $\widetilde{u} > 0$,
the spatial operator has an imaginary spectrum since it is skew-symmetric
with respect to the weighted $L^2$ inner product
$(v, w) \mapsto \int \widetilde{u} v w$
\cite{gassner2022stability,manzanero2017insights}.

However, \citet{gassner2022stability} observed that high-order
semidiscretizations conserving/dissipating the $L^2$ entropy $\int u^2$
lead to a linearized operator having some eigenvalues with a significantly
positive real part for a non-constant baseflow $\widetilde{u}$.
They discussed this in the context of local linear/energy
stability and issues of discretizations for under-resolved flows, see also
\cite{ranocha2021preventing}.

From the results and discussion in \cite{gassner2022stability}, it appears to be desirable
that a semidiscretization mimics the property of a linearization
having eigenvalues with non-positive real part for all positive baseflows
$\widetilde{u}$. Before studying the upwind SBP method specifically for Burgers'
equation, we concentrate on constant baseflows in a general setting.

Consider a scalar conservation law $\partial_t u + \partial_x f(u) = 0$
with periodic boundary conditions. We first consider entropy-conservative
flux differencing schemes of the form
\begin{equation}
\label{eq:flux-differencing}
  \partial_t \vec{u}_i + \sum_{j} 2 D_{ij} f^\mathrm{vol}(\vec{u}_i, \vec{u}_j) = 0,
\end{equation}
where $D$ is a periodic SBP operator and $f^\mathrm{vol}$ is an
entropy-conservative numerical flux in the sense of Tadmor
\cite{tadmor1987numerical,tadmor2003entropy}, i.e., it satisfies
\begin{equation}
  \forall u_l, u_r\colon
  \bigl( w(u_r) - w(u_l) \bigr) \cdot f^\mathrm{vol}(u_l, u_r)
  =
  \psi(u_r) - \psi(u_l),
\end{equation}
where $w(u) = U'(u)$ are the entropy variables and $\psi$ is the flux
potential associated to a convex entropy $U$. These methods have been
introduced in \cite{lefloch2002fully,fisher2013high}; see also
\cite{chen2020review,ranocha2023efficient}.

\begin{theorem}
  Entropy-conservative semidiscretizations using flux differencing in
  periodic domains are linearly/energy stable around constant states;
  in particular, their Jacobian has a purely imaginary spectrum.
\end{theorem}
\begin{proof}
  Here, we use the notation of \cite[Theorem~2.1]{chan2022efficient}, i.e.,
  $Q = M D$ and $F_{ij} = f^\mathrm{vol}(\vec{u}_i, \vec{u}_j)$.
  The Jacobian (multiplied by the negative mass matrix) is
  \begin{equation}
    J = 2 (Q \circ F_y) - \diag( 1^T (2 Q \circ F_y) ),
  \end{equation}
  where $Q$ is skew-symmetric (due to the periodic boundary conditions)
  and $\circ$ denotes the Hadamard (pointwise)
  product of two matrices. The matrix $F_y$ is given by the entries
  $\partial_2 f^\mathrm{vol}(\vec{u}_i, \vec{u}_j)$, i.e., the
  derivatives of the numerical flux $f^\mathrm{vol}$ with respect to the
  second argument evaluated at the states $\vec{u}_i, \vec{u}_j$.
  Since the derivative $F_y$ is evaluated at a constant state
  $\widetilde{u}$, all of its components are the same. In particular,
  $\forall i,j\colon (F_y)_{ij} = \frac{1}{2} f'(u)$
  \cite[Lemma~3.1]{ranocha2018comparison}. Thus,
  \begin{equation}
    J
    =
    2 (Q \circ F_y) - \diag( 1^T (2 Q \circ F_y) )
    =
    f'(u) \left( Q - \diag( 1^T Q ) \right)
    =
    f'(u) Q,
  \end{equation}
  where we used the SBP property. Hence, $J$ is skew-symmetric and has a
  purely imaginary spectrum.
\end{proof}

Next, we consider a central SBP discretization of the form
\begin{equation}
  \partial_t \vec{u} + D \vec{f} = \vec{0},
\end{equation}
where $D$ is a periodic SBP operator. As observed numerically in
\cite{gassner2022stability}, this leads to a purely imaginary spectrum
of the linearization.
\begin{theorem}
  Central nodal diagonal-norm SBP semidiscretizations of conservation laws in periodic
  domains are linearly stable around states with positive speed
  $f'(\widetilde{u}) > 0$;
  in particular, the Jacobian has a purely imaginary spectrum.
\end{theorem}
\begin{proof}
  The Jacobian (multiplied by the negative mass matrix) is
  \begin{equation}
    J = Q \diag(\vec{f}').
  \end{equation}
  Thus, it is skew-symmetric w.r.t. the inner product weighted by
  $\vec{f}' > 0$.
\end{proof}

Next, we consider fully upwind SBP methods.
\begin{theorem}
\label{thm:fully-upwind}
  Consider a possibly spatially varying baseflow $\widetilde{u}$ with
  positive speed $f'(\widetilde{u}) > 0$ everywhere.
  Upwind nodal diagonal-norm SBP semidiscretizations of the form
  $\partial_t \vec{u} + D_- \vec{f} = \vec{0}$ in periodic
  domains are linearly stable;
  in particular, the Jacobian has a spectrum in the left half of
  the complex plane.
\end{theorem}
\begin{proof}
  The Jacobian of the semidiscretization $-D_- \vec{f}$ is
  \begin{equation}
    J = -D_- \diag(\vec{f}').
  \end{equation}
  The spectrum of this operator must be in the left half of the complex
  plane, since for each (possibly complex-valued) vector $\vec{v}$
  \begin{equation}
  \begin{aligned}
    2 \Re \langle \vec{v}, J \vec{v} \rangle_{\diag(\vec{f}') M}
    &=
      \langle \vec{v}, J \vec{v} \rangle_{\diag(\vec{f}') M}
    + \langle J \vec{v}, \vec{v} \rangle_{\diag(\vec{f}') M}
    \\
    &=
    - \vec{v}^* \diag(\vec{f}') M D_- \diag(\vec{f}') \vec{v}
    - \vec{v}^* \diag(\vec{f}') D_-^T M \diag(\vec{f}') \vec{v}
    \\
    &=
    \vec{v}^* \diag(\vec{f}') (-D_-^T M - M D_-) \diag(\vec{f}') \vec{v}
    \\
    &=
    \vec{v}^* \diag(\vec{f}') (M D_+ - M D_-) \diag(\vec{f}') \vec{v}
    \leq 0,
  \end{aligned}
  \end{equation}
  where we used the SBP property and negative semidefiniteness for
  periodic upwind operators \eqref{eq:upwind-SBP-periodic} in the last two steps.
\end{proof}

\begin{remark}
  The assumption of a positive speed $f'(\widetilde{u}) > 0$ is equivalent
  to a positive baseflow $\widetilde{u}$ for Burgers' equation. While this
  is a strong assumption, it is exactly the situation investigated in
  \cite{gassner2022stability} where local linear/energy stability fails
  for entropy-stable methods.
\end{remark}

We get similar results for splittings such as the Lax-Friedrichs splitting,
at least for Burgers' equation.
\begin{theorem}
\label{thm:burgers-LLF-splitting}
  Upwind nodal diagonal-norm SBP semidiscretizations of Burgers' equation with local
  Lax-Friedrichs flux splitting in periodic domains are linearly stable
  around positive states. In particular,
  the Jacobian has a spectrum in the left half of the complex plane.
\end{theorem}
\begin{proof}
  The flux splitting is
  \begin{equation}
      \frac{u^2}{2}
      =
      \frac{1}{2} \left( \frac{u^2}{2} + |u| u \right)
      + \frac{1}{2} \left( \frac{u^2}{2} - |u| u \right).
  \end{equation}
  For positive $\vec{u}$, the semidiscretization is
  \begin{equation}
      \partial_t \vec{u}
      =
      - \frac{3}{4} D_- \vec{u}^2
      + \frac{1}{4} D_+ \vec{u}^2.
  \end{equation}
  The Jacobian of the right-hand side is
  \begin{equation}
    J
    =
    - \frac{3}{4} D_- \diag((\vec{u}^2)')
    + \frac{1}{4} D_+ \diag((\vec{u}^2)').
  \end{equation}
  As in Theorem~\ref{thm:fully-upwind}, we can show that this Jacobian
  has a spectrum in the left half of the complex plane, since for all
  (complex) grid vectors $\vec{v}$
  \begin{equation}
  \begin{aligned}
    2 \Re \langle \vec{v}, J \vec{v} \rangle_{\diag((\vec{u}^2)') M}
    &=
    - \frac{3}{2} \vec{v}^* \diag((\vec{u}^2)') M D_- \diag((\vec{u}^2)') \vec{v}
    + \frac{1}{2} \vec{v}^* \diag((\vec{u}^2)') M D_+ \diag((\vec{u}^2)') \vec{v}
    \\
    &\quad
    - \frac{3}{2} \vec{v}^* \diag((\vec{u}^2)') D_-^T M \diag((\vec{u}^2)') \vec{v}
    + \frac{1}{2} \vec{v}^* \diag((\vec{u}^2)') D_+^T M \diag((\vec{u}^2)') \vec{v}
    \\
    &=
    \frac{3}{2} \vec{v}^* \diag((\vec{u}^2)') \left(
        -M D_- - D_-^T M
    \right) \diag((\vec{u}^2)') \vec{v}
    \\
    &\quad
    + \frac{1}{2} \vec{v}^* \diag((\vec{u}^2)') \left(
        M D_+ + D_+^T M
    \right) \diag((\vec{u}^2)') \vec{v}
  \end{aligned}
  \end{equation}
  Both matrices in brackets are negative semidefinite, since the upwind
  SBP properties guarantee that
  \begin{equation}
    -M D_- - D_-^T M
    =
    M D_+ - M D_-
    =
    M D_+ + D_+^T M
  \end{equation}
  is negative semidefinite.
\end{proof}

\begin{remark}
  The proof of Theorem~\ref{thm:burgers-LLF-splitting} holds for
  scalar conservation laws with homogeneous flux $f(s u) = s^\alpha f(u)$.
  In this case, the flux splitting with positive speeds is
  \begin{equation}
    f_+ = \frac{1}{2} \left( f + f_u u \right)
        = \frac{1 + \alpha}{2} f,
    \quad
    f_- = \frac{1}{2} \left( f - f_u u \right)
        = \frac{1 - \alpha}{2} f,
  \end{equation}
  due to Eulers' theorem.
\end{remark}

\subsection{A special choice of entropy}

We relate standard central schemes
$\partial_t \vec{u} + D \vec{f} = \vec{0}$
to entropy-conservative schemes with a special choice
of entropy function. Consider a scalar conservation law
$\partial_t u + \partial_x f(u) = 0$ with positive wave speeds
$f'(u) > 0$. In this case, the primitive $U(u) = \int^u f(y) \dif y$
of the flux is a convex entropy with entropy flux $F = f^2 / 2$,
cf.~\cite{tadmor1987entropy}.
The associated entropy variable $w = U' = f$ is the flux itself and
a smooth solution yields
\begin{equation}
  \partial_t U + \partial_x F
  =
  \partial_t U + \partial_x \frac{f^2}{2}
  =
  f \cdot (\partial_t u + \partial_x f)
  =
  0.
\end{equation}
Thus, the entropy-conservative numerical flux of Tadmor is given by
the central flux since
\begin{equation}
  (f_r - f_l) \frac{f_r + f_l}{2} = \frac{f_r^2}{2} - \frac{f_l^2}{2},
\end{equation}
where we have used the entropy flux potential
\begin{equation}
  \psi = w \cdot f - F = f^2 - \frac{f^2}{2} = \frac{f^2}{2}.
\end{equation}
Finally, flux differencing methods using the central numerical flux
are equivalent to the central discretization
$\partial_t \vec{u} + D \vec{f} = \vec{0}$.
Thus, nonlinear entropy stability and local linear/energy stability
can be combined in this very special situation.

\begin{remark}
  The dissipation introduced by upwind SBP operators is compatible with the
  structure of the local linear/energy stability estimate. In particular,
  the dissipation introduced compared to a central scheme dissipates the
  entropy $U(u) = \int^u f(y) \dif y$, since
  \begin{equation}
    \vec{f}^T M (-D_- \vec{f})
    =
    - \vec{f}^T M \frac{D_+ + D_-}{2} \vec{f}
    + \vec{f}^T M \frac{D_+ - D_-}{2} \vec{f}
    \leq
    -\vec{f}^T M \frac{D_+ + D_-}{2} \vec{f}.
  \end{equation}
  The central scheme with operator $(D_+ + D_-) / 2$ conserves this entropy,
  so the total upwind scheme is entropy-dissipative.
\end{remark}

\subsection{Discussion}

Following \cite{gassner2022stability}, three desirable properties of
numerical methods are
i) nonlinear entropy stability,
ii) local linear/energy stability, and
iii) high-order accuracy.
There have been substantial discussions in the high-order community about
these properties. Clearly, central-type schemes can be high-order accurate
and just satisfy local linear/energy stability without any dissipation.
We have shown that upwind SBP methods can have the same properties while
coming with some built-in dissipation. However, it is unclear whether they
have some nonlinear/entropy stability properties besides the special,
academic choice of entropy in the previous subsection. Clearly, first-order
methods such as Godunov's method can have both nonlinear entropy and
local linear/energy stability properties. It is an open question whether
higher-order methods can have both of these stability properties as well.

\section{Numerical experiments}
\label{sec:experiments}

We use the Julia programming language \cite{bezanson2017julia} for the numerical experiments.
Time integration is performed using Runge-Kutta methods implemented in
OrdinaryDiffEq.jl \cite{rackauckas2017differentialequations}
(specific choices of the Runge-Kutta methods are stated below).
The spatial discretizations are available in Trixi.jl
\cite{ranocha2022adaptive,schlottkelakemper2021purely}.
All numerical experiments presented in this section use diagonal-norm
upwind SBP operators of \cite{mattsson2017diagonal} available from
SummationByPartsOperators.jl \cite{ranocha2021sbp} (unless stated
otherwise).
Some of the unstructured curvilinear quadrilateral meshes were constructed
with HOHQMesh.jl\footnote{https://github.com/trixi-framework/HOHQMesh.jl}.
We use Plots.jl \cite{christ2023plots} and ParaView \cite{ahrens2005paraview}
to visualize the results.
All source code required to reproduce the numerical experiments is
available online in our reproducibility repository
\cite{ranocha2023highRepro}.

\subsection{Convergence experiments with linear advection}
\label{sec:convergence_advection}

First, we consider the linear advection equation
\begin{equation}
\begin{aligned}
  \partial_t u(t, x) + \partial_x u(t, x) &= 0, && t \in (0, 5), x \in (-1, 1),
  \\
  u(0, x) &= \sin(\pi x), && x \in [-1, 1],
\end{aligned}
\end{equation}
with periodic boundary conditions. We use the classical Lax-Friedrichs flux
vector splitting with $\lambda = 1$, i.e.,
\begin{equation}
  f^-(u) = 0, \quad f^+(u) = u.
\end{equation}

\begin{table}[htbp]
\centering
  \caption{Convergence results using upwind SBP discretizations
           of the linear advection equation with Lax-Friedrichs splitting,
           $K$ elements, $N$ nodes per element, and an interior order
           of accuracy 2.}
  \label{tab:convergence_advection_2}
  \begin{tabular}{rrrr}
    \toprule
    $K$ & $N$ & $L^2$ error & EOC \\
    \midrule
      1 &  20 & \num{3.46e-01} &  \\
      2 &  20 & \num{9.24e-02} & 1.91 \\
      4 &  20 & \num{2.33e-02} & 1.99 \\
      8 &  20 & \num{5.83e-03} & 2.00 \\
    16 &  20 & \num{1.46e-03} & 2.00 \\
    32 &  20 & \num{3.64e-04} & 2.00 \\
    64 &  20 & \num{9.11e-05} & 2.00 \\
    128 &  20 & \num{2.28e-05} & 2.00 \\
    \bottomrule
  \end{tabular}
  \hspace{1cm}
  \begin{tabular}{rrrr}
    \toprule
    $K$ & $N$ & $L^2$ error & EOC \\
    \midrule
    4 &  10 & \num{9.32e-02} &  \\
    4 &  20 & \num{2.33e-02} & 2.00 \\
    4 &  40 & \num{5.77e-03} & 2.01 \\
    4 &  80 & \num{1.44e-03} & 2.01 \\
    4 & 160 & \num{3.58e-04} & 2.00 \\
    4 & 320 & \num{8.94e-05} & 2.00 \\
    4 & 640 & \num{2.23e-05} & 2.00 \\
    4 & 1280 & \num{5.58e-06} & 2.00 \\
    \bottomrule
  \end{tabular}
\end{table}

\begin{table}[htbp]
\centering
  \caption{Convergence results using upwind SBP discretizations
           of the linear advection equation with Lax-Friedrichs splitting,
           $K$ elements, $N$ nodes per element, and an interior order
           of accuracy 3.}
  \label{tab:convergence_advection_3}
  \begin{tabular}{rrrr}
    \toprule
    $K$ & $N$ & $L^2$ error & EOC \\
    \midrule
      1 &  20 & \num{3.40e-02} &  \\
      2 &  20 & \num{4.93e-03} & 2.78 \\
      4 &  20 & \num{8.73e-04} & 2.50 \\
      8 &  20 & \num{1.87e-04} & 2.22 \\
    16 &  20 & \num{4.47e-05} & 2.07 \\
    32 &  20 & \num{1.11e-05} & 2.01 \\
    64 &  20 & \num{2.76e-06} & 2.00 \\
    128 &  20 & \num{6.90e-07} & 2.00 \\
    \bottomrule
  \end{tabular}
  \hspace{1cm}
  \begin{tabular}{rrrr}
    \toprule
    $K$ & $N$ & $L^2$ error & EOC \\
    \midrule
    4 &  10 & \num{7.98e-03} &  \\
    4 &  20 & \num{8.73e-04} & 3.19 \\
    4 &  40 & \num{1.05e-04} & 3.06 \\
    4 &  80 & \num{1.34e-05} & 2.96 \\
    4 & 160 & \num{1.83e-06} & 2.87 \\
    4 & 320 & \num{2.68e-07} & 2.77 \\
    4 & 640 & \num{4.20e-08} & 2.68 \\
    4 & 1280 & \num{6.91e-09} & 2.60 \\
    \bottomrule
  \end{tabular}
\end{table}

\begin{table}[htbp]
\centering
  \caption{Convergence results using upwind SBP discretizations
           of the linear advection equation with Lax-Friedrichs splitting,
           $K$ elements, $N$ nodes per element, and an interior order
           of accuracy 4.}
  \label{tab:convergence_advection_4}
  \begin{tabular}{rrrr}
    \toprule
    $K$ & $N$ & $L^2$ error & EOC \\
    \midrule
      1 &  20 & \num{5.03e-03} &  \\
      2 &  20 & \num{3.96e-04} & 3.67 \\
      4 &  20 & \num{3.19e-05} & 3.63 \\
      8 &  20 & \num{3.66e-06} & 3.12 \\
    16 &  20 & \num{4.51e-07} & 3.02 \\
    32 &  20 & \num{5.57e-08} & 3.02 \\
    64 &  20 & \num{6.96e-09} & 3.00 \\
    128 &  20 & \num{8.73e-10} & 3.00 \\
    \bottomrule
  \end{tabular}
  \hspace{1cm}
  \begin{tabular}{rrrr}
    \toprule
    $K$ & $N$ & $L^2$ error & EOC \\
    \midrule
    4 &  10 & \num{3.86e-04} &  \\
    4 &  20 & \num{3.19e-05} & 3.60 \\
    4 &  40 & \num{2.57e-06} & 3.63 \\
    4 &  80 & \num{2.10e-07} & 3.61 \\
    4 & 160 & \num{1.77e-08} & 3.57 \\
    4 & 320 & \num{1.51e-09} & 3.55 \\
    4 & 640 & \num{1.30e-10} & 3.54 \\
    4 & 1280 & \num{1.17e-11} & 3.48 \\
    \bottomrule
  \end{tabular}
\end{table}

\begin{table}[htbp]
\centering
  \caption{Convergence results using upwind SBP discretizations
           of the linear advection equation with Lax-Friedrichs splitting,
           $K$ elements, $N$ nodes per element, and an interior order
           of accuracy 5.}
  \label{tab:convergence_advection_5}
  \begin{tabular}{rrrr}
    \toprule
    $K$ & $N$ & $L^2$ error & EOC \\
    \midrule
      1 &  20 & \num{3.49e-03} &  \\
      2 &  20 & \num{3.77e-04} & 3.21 \\
      4 &  20 & \num{3.31e-05} & 3.51 \\
      8 &  20 & \num{4.09e-06} & 3.01 \\
    16 &  20 & \num{5.13e-07} & 3.00 \\
    32 &  20 & \num{6.43e-08} & 3.00 \\
    64 &  20 & \num{8.02e-09} & 3.00 \\
    128 &  20 & \num{1.01e-09} & 2.99 \\
    \bottomrule
  \end{tabular}
  \hspace{1cm}
  \begin{tabular}{rrrr}
    \toprule
    $K$ & $N$ & $L^2$ error & EOC \\
    \midrule
    4 &  10 & \num{5.17e-04} &  \\
    4 &  20 & \num{3.31e-05} & 3.97 \\
    4 &  40 & \num{2.64e-06} & 3.65 \\
    4 &  80 & \num{2.23e-07} & 3.56 \\
    4 & 160 & \num{1.93e-08} & 3.53 \\
    4 & 320 & \num{1.69e-09} & 3.51 \\
    4 & 640 & \num{1.51e-10} & 3.48 \\
    4 & 1280 & \num{1.34e-11} & 3.49 \\
    \bottomrule
  \end{tabular}
\end{table}

We use the fourth-order accurate Runge-Kutta method of
\cite{ranocha2021optimized} with error-based step size control and a
sufficiently small tolerance to integrate the semidiscretizations in time.
We measure the discrete $L^2$ error using the quadrature rule induced by
the mass matrix $M$. Results of these convergence experiments, including the
experimental order of convergence (EOC), are shown in
Tables~\ref{tab:convergence_advection_2}--\ref{tab:convergence_advection_5}.

When used in DG refinement mode, i.e., increasing the number of elements
while keeping the number of nodes per element constant,
the methods with an interior order of accuracy $p$ converge with an EOC of
$\lfloor p / 2 + 1 \rfloor$.
When used in FD refinement mode, i.e., increasing the number of nodes per
element while keeping the number of elements constant,
the methods with an interior order of accuracy $p$ converge with an EOC of
roughly $\max(p, \lfloor p / 2 + 1 \rfloor + 1/2)$.

\subsection{Convergence experiments with the compressible Euler equations}
\label{sec:convergence_euler}

Next, we investigate the experimental order of convergence for the upwind SBP
framework with different flux vector splittings in one and two space dimensions.
The one-dimensional results are presented in Section~\ref{sec:1d_conv} and
the two-dimensional results on unstructured curvilinear meshes are given in
Section~\ref{sec:2d_conv}.

\subsubsection{One spatial dimension}
\label{sec:1d_conv}

Consider the 1D compressible Euler equations
\begin{equation}
  \partial_t \begin{pmatrix} \rho \\ \rho v \\ \rho e \end{pmatrix}
  + \partial_x \begin{pmatrix} \rho v \\ \rho v^2 + p \\ (\rho e + p) v \end{pmatrix}
  = 0
\end{equation}
of an ideal gas with density $\rho$, velocity $v$, total energy density
$\rho e$, and pressure
\begin{equation}
  p = (\gamma - 1) \left( \rho e - \frac{1}{2} \rho v^2 \right),
\end{equation}
where the ratio of specific heats is chosen as $\gamma = 1.4$. We add a source
term to create the manufactured solution
\begin{equation}
  \rho(t, x) = h(t, x), \quad v(t, x) = 1, \quad \rho e(t, x) = h(t, x)^2,
\end{equation}
with
\begin{equation}
  h(t, x) = 2 + 0.1 \sin\bigl( \pi (x - t) \bigr)
\end{equation}
for $t \in [0, 2]$ and $x \in [0, 2]$. We use the flux vector splittings introduced in
Examples~\ref{ex:Steger-Warming} and \ref{ex:van-Leer-Hanel}.

\begin{table}[htbp]
\centering
  \caption{Convergence results using upwind SBP discretizations
           of the compressible Euler equations with
           $K$ elements, $N$ nodes per element, and an interior order
           of accuracy 2.}
  \label{tab:convergence_euler_2}
  \begin{subtable}{0.32\textwidth}
  \centering
    \caption{van Leer-Hänel spl. \cite{vanleer1982flux,hanel1987accuracy,liou1991high}.}
    \begin{tabular}{rrrr}
      \toprule
      $K$ & $N$ & $L^2$ error & EOC \\
      \midrule
        1 &  20 & \num{1.01e-02} &  \\
        2 &  20 & \num{2.94e-03} & 1.78 \\
        4 &  20 & \num{7.55e-04} & 1.96 \\
        8 &  20 & \num{1.91e-04} & 1.98 \\
      16 &  20 & \num{4.79e-05} & 2.00 \\
      32 &  20 & \num{1.19e-05} & 2.01 \\
      64 &  20 & \num{2.98e-06} & 2.00 \\
      128 &  20 & \num{7.45e-07} & 2.00 \\
      \bottomrule
    \end{tabular}
  \end{subtable}%
  \hspace{\fill}
  \begin{subtable}{0.32\textwidth}
  \centering
    \caption{Steger-Warming splitting \cite{steger1979flux}.}
    \begin{tabular}{rrrr}
      \toprule
      $K$ & $N$ & $L^2$ error & EOC \\
      \midrule
        1 &  20 & \num{1.02e-02} &  \\
        2 &  20 & \num{2.95e-03} & 1.79 \\
        4 &  20 & \num{7.59e-04} & 1.96 \\
        8 &  20 & \num{1.92e-04} & 1.99 \\
      16 &  20 & \num{4.79e-05} & 2.00 \\
      32 &  20 & \num{1.19e-05} & 2.01 \\
      64 &  20 & \num{2.98e-06} & 2.00 \\
      128 &  20 & \num{7.46e-07} & 2.00 \\
      \bottomrule
    \end{tabular}
  \end{subtable}%
  \hspace{\fill}
  \begin{subtable}{0.32\textwidth}
  \centering
    \caption{Steger-Warming splitting \cite{steger1979flux}.}
    \begin{tabular}{rrrr}
      \toprule
      $K$ & $N$ & $L^2$ error & EOC \\
      \midrule
      4 &  10 & \num{2.97e-03} &  \\
      4 &  20 & \num{7.59e-04} & 1.97 \\
      4 &  40 & \num{1.89e-04} & 2.01 \\
      4 &  80 & \num{4.69e-05} & 2.01 \\
      4 & 160 & \num{1.17e-05} & 2.01 \\
      4 & 320 & \num{2.92e-06} & 2.00 \\
      4 & 640 & \num{7.29e-07} & 2.00 \\
      4 & 1280 & \num{1.82e-07} & 2.00 \\
      \bottomrule
    \end{tabular}
  \end{subtable}%
\end{table}

\begin{table}[htbp]
\centering
  \caption{Convergence results using upwind SBP discretizations
           of the compressible Euler equations with
           $K$ elements, $N$ nodes per element, and an interior order
           of accuracy 3.}
  \label{tab:convergence_euler_3}
  \begin{subtable}{0.32\textwidth}
  \centering
    \caption{van Leer-Hänel spl. \cite{vanleer1982flux,hanel1987accuracy,liou1991high}.}
    \begin{tabular}{rrrr}
      \toprule
      $K$ & $N$ & $L^2$ error & EOC \\
      \midrule
        1 &  20 & \num{1.14e-03} &  \\
        2 &  20 & \num{1.75e-04} & 2.70 \\
        4 &  20 & \num{4.41e-05} & 1.99 \\
        8 &  20 & \num{1.21e-05} & 1.86 \\
      16 &  20 & \num{3.16e-06} & 1.94 \\
      32 &  20 & \num{5.96e-07} & 2.41 \\
      64 &  20 & \num{1.43e-07} & 2.06 \\
      128 &  20 & \num{3.43e-08} & 2.06 \\
      \bottomrule
    \end{tabular}
  \end{subtable}%
  \hspace{\fill}
  \begin{subtable}{0.32\textwidth}
  \centering
    \caption{Steger-Warming splitting \cite{steger1979flux}.}
    \begin{tabular}{rrrr}
      \toprule
      $K$ & $N$ & $L^2$ error & EOC \\
      \midrule
        1 &  20 & \num{1.18e-03} &  \\
        2 &  20 & \num{1.83e-04} & 2.69 \\
        4 &  20 & \num{4.51e-05} & 2.02 \\
        8 &  20 & \num{1.21e-05} & 1.90 \\
      16 &  20 & \num{3.05e-06} & 1.98 \\
      32 &  20 & \num{5.59e-07} & 2.45 \\
      64 &  20 & \num{1.35e-07} & 2.05 \\
      128 &  20 & \num{3.22e-08} & 2.07 \\
      \bottomrule
    \end{tabular}
  \end{subtable}%
  \hspace{\fill}
  \begin{subtable}{0.32\textwidth}
  \centering
    \caption{Steger-Warming splitting \cite{steger1979flux}.}
    \begin{tabular}{rrrr}
      \toprule
      $K$ & $N$ & $L^2$ error & EOC \\
      \midrule
      4 &  10 & \num{2.97e-04} &  \\
      4 &  20 & \num{4.51e-05} & 2.72 \\
      4 &  40 & \num{7.56e-06} & 2.58 \\
      4 &  80 & \num{1.16e-06} & 2.70 \\
      4 & 160 & \num{1.78e-07} & 2.70 \\
      4 & 320 & \num{2.81e-08} & 2.66 \\
      4 & 640 & \num{4.59e-09} & 2.62 \\
      4 & 1280 & \num{7.69e-10} & 2.58 \\
      \bottomrule
    \end{tabular}
  \end{subtable}%
\end{table}

\begin{table}[htbp]
\centering
  \caption{Convergence results using upwind SBP discretizations
           of the compressible Euler equations with
           $K$ elements, $N$ nodes per element, and an interior order
           of accuracy 4.}
  \label{tab:convergence_euler_4}
  \begin{subtable}{0.32\textwidth}
  \centering
    \caption{van Leer-Hänel spl. \cite{vanleer1982flux,hanel1987accuracy,liou1991high}.}
    \begin{tabular}{rrrr}
      \toprule
      $K$ & $N$ & $L^2$ error & EOC \\
      \midrule
        1 &  20 & \num{3.01e-04} &  \\
        2 &  20 & \num{3.51e-05} & 3.10 \\
        4 &  20 & \num{4.14e-06} & 3.08 \\
        8 &  20 & \num{5.84e-07} & 2.83 \\
      16 &  20 & \num{6.60e-08} & 3.14 \\
      32 &  20 & \num{6.68e-09} & 3.31 \\
      64 &  20 & \num{8.24e-10} & 3.02 \\
      128 &  20 & \num{9.63e-11} & 3.10 \\
      \bottomrule
    \end{tabular}
  \end{subtable}%
  \hspace{\fill}
  \begin{subtable}{0.32\textwidth}
  \centering
    \caption{Steger-Warming splitting \cite{steger1979flux}.}
    \begin{tabular}{rrrr}
      \toprule
      $K$ & $N$ & $L^2$ error & EOC \\
      \midrule
        1 &  20 & \num{2.30e-04} &  \\
        2 &  20 & \num{2.71e-05} & 3.09 \\
        4 &  20 & \num{4.21e-06} & 2.69 \\
        8 &  20 & \num{5.88e-07} & 2.84 \\
      16 &  20 & \num{6.50e-08} & 3.18 \\
      32 &  20 & \num{6.53e-09} & 3.32 \\
      64 &  20 & \num{7.94e-10} & 3.04 \\
      128 &  20 & \num{9.17e-11} & 3.11 \\
      \bottomrule
    \end{tabular}
  \end{subtable}%
  \hspace{\fill}
  \begin{subtable}{0.32\textwidth}
  \centering
    \caption{Steger-Warming splitting \cite{steger1979flux}.}
    \begin{tabular}{rrrr}
      \toprule
      $K$ & $N$ & $L^2$ error & EOC \\
      \midrule
      4 &  10 & \num{3.72e-05} &  \\
      4 &  20 & \num{4.21e-06} & 3.14 \\
      4 &  40 & \num{3.66e-07} & 3.52 \\
      4 &  80 & \num{2.66e-08} & 3.78 \\
      4 & 160 & \num{1.95e-09} & 3.77 \\
      4 & 320 & \num{1.60e-10} & 3.61 \\
      4 & 640 & \num{1.37e-11} & 3.54 \\
      4 & 1280 & \num{2.05e-12} & 2.74 \\
      \bottomrule
    \end{tabular}
  \end{subtable}%
\end{table}

\begin{table}[htbp]
\centering
  \caption{Convergence results using upwind SBP discretizations
           of the compressible Euler equations with
           $K$ elements, $N$ nodes per element, and an interior order
           of accuracy 5.}
  \label{tab:convergence_euler_5}
  \begin{subtable}{0.32\textwidth}
  \centering
    \caption{van Leer-Hänel spl. \cite{vanleer1982flux,hanel1987accuracy,liou1991high}.}
    \begin{tabular}{rrrr}
      \toprule
      $K$ & $N$ & $L^2$ error & EOC \\
      \midrule
        1 &  20 & \num{2.05e-04} &  \\
        2 &  20 & \num{3.29e-05} & 2.64 \\
        4 &  20 & \num{3.74e-06} & 3.14 \\
        8 &  20 & \num{5.15e-07} & 2.86 \\
      16 &  20 & \num{5.99e-08} & 3.10 \\
      32 &  20 & \num{6.97e-09} & 3.10 \\
      64 &  20 & \num{8.67e-10} & 3.01 \\
      128 &  20 & \num{1.05e-10} & 3.04 \\
      \bottomrule
    \end{tabular}
  \end{subtable}%
  \hspace{\fill}
  \begin{subtable}{0.32\textwidth}
  \centering
    \caption{Steger-Warming splitting \cite{steger1979flux}.}
    \begin{tabular}{rrrr}
      \toprule
      $K$ & $N$ & $L^2$ error & EOC \\
      \midrule
        1 &  20 & \num{1.17e-04} &  \\
        2 &  20 & \num{2.44e-05} & 2.26 \\
        4 &  20 & \num{4.13e-06} & 2.56 \\
        8 &  20 & \num{5.20e-07} & 2.99 \\
      16 &  20 & \num{5.85e-08} & 3.15 \\
      32 &  20 & \num{6.78e-09} & 3.11 \\
      64 &  20 & \num{8.38e-10} & 3.02 \\
      128 &  20 & \num{1.02e-10} & 3.04 \\
      \bottomrule
    \end{tabular}
  \end{subtable}%
  \hspace{\fill}
  \begin{subtable}{0.32\textwidth}
  \centering
    \caption{Steger-Warming splitting \cite{steger1979flux}.}
    \begin{tabular}{rrrr}
      \toprule
      $K$ & $N$ & $L^2$ error & EOC \\
      \midrule
      4 &  10 & \num{4.20e-05} &  \\
      4 &  20 & \num{4.13e-06} & 3.35 \\
      4 &  40 & \num{3.28e-07} & 3.66 \\
      4 &  80 & \num{2.53e-08} & 3.69 \\
      4 & 160 & \num{2.05e-09} & 3.63 \\
      4 & 320 & \num{1.75e-10} & 3.55 \\
      4 & 640 & \num{1.52e-11} & 3.52 \\
      4 & 1280 & \num{1.63e-12} & 3.23 \\
      \bottomrule
    \end{tabular}
  \end{subtable}%
\end{table}

The convergence results for the compressible Euler equations shown in
Tables~\ref{tab:convergence_euler_2}--\ref{tab:convergence_euler_5}
confirm that the behavior of the experimental order of convergence
observed earlier for the linear advection equation remains the same for
a nonlinear hyperbolic system.

\subsubsection{Two spatial dimensions}
\label{sec:2d_conv}

Next, consider the 2D compressible Euler equations
\begin{equation}
  \label{eq:2d_comp_euler}
  \partial_t \begin{pmatrix} \rho \\ \rho v_1 \\ \rho v_2 \\ \rho e \end{pmatrix}
  + \partial_x \begin{pmatrix} \rho v_1 \\ \rho v_1^2 + p \\ \rho v_1 v_2 \\ (\rho e + p) v_1 \end{pmatrix}
  + \partial_y \begin{pmatrix} \rho v_2 \\ \rho v_1 v_2 \\ \rho v_2^2 + p \\ (\rho e + p) v_2 \end{pmatrix}
  = 0
\end{equation}
of an ideal gas with density $\rho$, velociteis $v_1$, $v_2$, total energy density
$\rho e$, and pressure
\begin{equation}
  p = (\gamma - 1) \left( \rho e - \frac{1}{2} \rho (v_1^2 + v_2^2) \right),
\end{equation}
where the ratio of specific heats is chosen as $\gamma = 1.4$. We add a source
term to create the manufactured solution
\begin{equation}
  \rho(t, x) = h(t, x), \quad v_1(t, x) = v_2(t, x) = 1, \quad \rho e(t, x) = h(t, x)^2,
\end{equation}
with
\begin{equation}
  h(t, x) = 2 + 0.1 \sin\bigl( \sqrt{2}\pi (x - t) \bigr)
\end{equation}
for $t \in [0, 2]$, $x \in [0, \sqrt{2}]^2$, and periodic boundary conditions.
The full expressions of the source terms and all code required to reproduce
these experiments is available in our reproducibility repository
\cite{ranocha2023highRepro}.

We subdivide the domain $[0,\sqrt{2}]^2$ with 16 non-overlapping quadrilateral elements.
For these convergence tests we consider two unstructured meshes, one with only bi-linear elements
and the other containing internal element boundaries approximated with quadratic polynomials.
Moreover, we design these meshes such that several neighboring elements have flipped local
coordinate systems, as is possible in unstructured mesh computations. Even so, the domain
discretized with either mesh remains periodic as required by the manufactured solution setup.
The two meshes used for the convergence testing are given in Figure~\ref{fig:conv_meshes}.

\begin{figure}[htbp]
\centering
  \begin{subfigure}{0.44\textwidth}
  \centering
    \includegraphics[width=\textwidth]{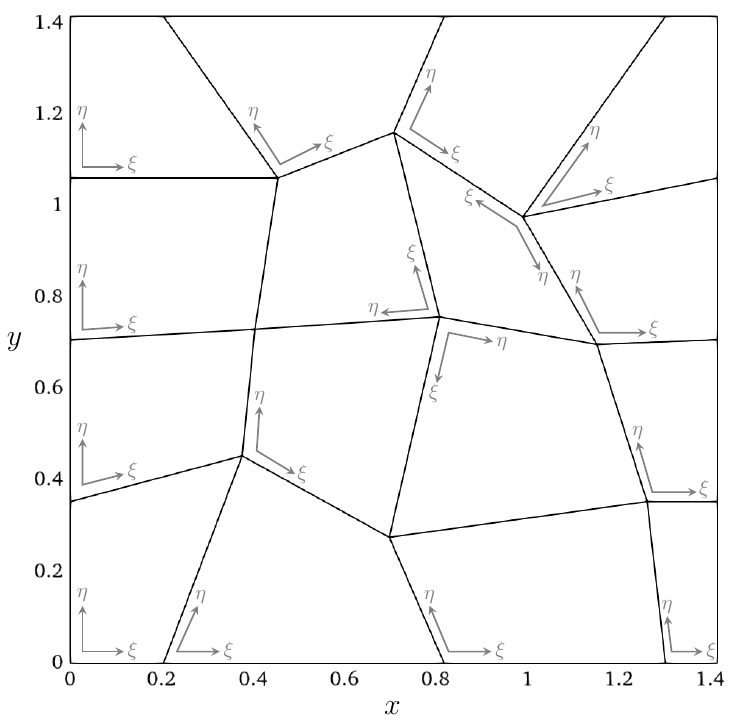}
    \caption{Mesh with bi-linear elements.}
  \end{subfigure}%
  \hspace*{\fill}
  \begin{subfigure}{0.44\textwidth}
  \centering
    \includegraphics[width=\textwidth]{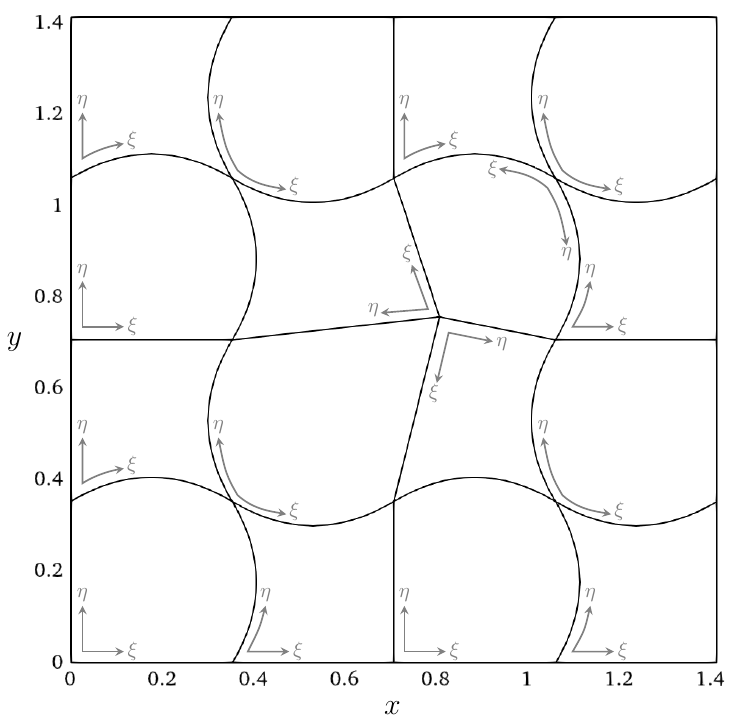}
    \caption{Mesh with (possibly) quadratic elements.}
  \end{subfigure}%
  \caption{Non-overlapping quadrilateral meshes used for the convergence testing on unstructured meshes.
  The local coordinate axes on each element denoted with $\xi$ and $\eta$ demonstrate that several elements
  have flipped local coordinate systems with respect to their neighbor elements.}
  \label{fig:conv_meshes}
\end{figure}

We use the manufactured solution described above to compute the experimental order of convergence
for the Lax-Friedrichs, Drikakis-Tsangaris, and van Leer-Hänel splittings on both meshes given
in Figure~\ref{fig:conv_meshes}. In particular, we use the bi-linear unstructured mesh from Figure~\ref{fig:conv_meshes}(a)
to test convergence of the 4-2 and 6-3 upwind SBP operators and the quadratic unstructured mesh shown
in Figure~\ref{fig:conv_meshes}(b) to test the convergence of the 8-4 upwind SBP operator.

\begin{table}[htbp]
\centering
  \caption{Convergence results using curvilinear upwind SBP discretizations
           for the compressible Euler equations with
           $K$ elements, $N$ nodes per element, and an interior order
           of accuracy 4 on the unstructured bi-linear mesh
           shown in Figure~\ref{fig:conv_meshes}(a).}
  \label{tab:convergence_euler_curved1}
  \begin{subtable}{0.32\textwidth}
  \centering
    \caption{local Lax-Friedrichs spl. Ex.~\ref{ex:llf_splitting}.}
    \begin{tabular}{rrrr}
      \toprule
      $K$ & $N$ & $L^2$ error & EOC \\
      \midrule
   16 &  17 & \num{1.92e-04} &  \\
   16 &  34 & \num{1.60e-05} & 3.58 \\
   16 &  68 & \num{1.46e-06} & 3.45 \\
   16 & 136 & \num{1.53e-07} & 3.26 \\
   16 & 272 & \num{1.66e-08} & 3.20 \\
    \bottomrule
    \end{tabular}
  \end{subtable}%
  \hspace{\fill}
  \begin{subtable}{0.32\textwidth}
  \centering
    \caption{Drikakis-Tsangaris splitting \cite{drikakis1993solution}.}
    \begin{tabular}{rrrr}
      \toprule
      $K$ & $N$ & $L^2$ error & EOC \\
      \midrule
   16 &  17 & \num{1.35e-04} & \\
   16 &  34 & \num{1.18e-05} & 3.52 \\
   16 &  68 & \num{1.14e-06} & 3.36 \\
   16 & 136 & \num{1.18e-07} & 3.27 \\
   16 & 272 & \num{1.30e-08} & 3.18 \\
    \bottomrule
    \end{tabular}
  \end{subtable}%
  \hspace{\fill}
  \begin{subtable}{0.32\textwidth}
  \centering
    \caption{van Leer-Hänel splitting \cite{anderson1986comparison}.}
    \begin{tabular}{rrrr}
      \toprule
      $K$ & $N$ & $L^2$ error & EOC \\
      \midrule
   16 &  17 & \num{9.21e-05} & \\
   16 &  34 & \num{8.07e-06} & 3.51 \\
   16 &  68 & \num{8.15e-07} & 3.31 \\
   16 & 136 & \num{8.19e-08} & 3.31 \\
   16 & 272 & \num{8.83e-09} & 3.21 \\
   \bottomrule
    \end{tabular}
  \end{subtable}%
\end{table}

\begin{table}[htbp]
\centering
  \caption{Convergence results using curvilinear upwind SBP discretizations
           for the compressible Euler equations with
           $K$ elements, $N$ nodes per element, and an interior order
           of accuracy 6 on the unstructured bi-linear mesh
           shown in Figure~\ref{fig:conv_meshes}(a).}
  \label{tab:convergence_euler_curved2}
  \begin{subtable}{0.32\textwidth}
  \centering
    \caption{local Lax-Friedrichs spl. Ex.~\ref{ex:llf_splitting}.}
    \begin{tabular}{rrrr}
      \toprule
      $K$ & $N$ & $L^2$ error & EOC \\
      \midrule
   16 &  17 & \num{1.95e-05} & \\
   16 &  34 & \num{9.26e-07} & 4.40 \\
   16 &  68 & \num{4.71e-08} & 4.30 \\
   16 & 136 & \num{2.46e-09} & 4.26 \\
   16 & 272 & \num{1.32e-10} & 4.22 \\
      \bottomrule
    \end{tabular}
  \end{subtable}%
  \hspace{\fill}
  \begin{subtable}{0.32\textwidth}
  \centering
    \caption{Drikakis-Tsangaris splitting \cite{drikakis1993solution}.}
    \begin{tabular}{rrrr}
      \toprule
      $K$ & $N$ & $L^2$ error & EOC \\
      \midrule
   16 &  17 & \num{2.15e-05} & \\
   16 &  34 & \num{1.04e-06} & 4.37 \\
   16 &  68 & \num{5.75e-08} & 4.17 \\
   16 & 136 & \num{3.06e-09} & 4.23 \\
   16 & 272 & \num{1.65e-10} & 4.21 \\
    \bottomrule
    \end{tabular}
  \end{subtable}%
  \hspace{\fill}
  \begin{subtable}{0.32\textwidth}
  \centering
    \caption{van Leer-Hänel splitting \cite{anderson1986comparison}.}
    \begin{tabular}{rrrr}
      \toprule
      $K$ & $N$ & $L^2$ error & EOC \\
      \midrule
   16 &  17 & \num{2.33e-05} & \\
   16 &  34 & \num{1.15e-06} & 4.34 \\
   16 &  68 & \num{6.74e-08} & 4.09 \\
   16 & 136 & \num{3.70e-09} & 4.19 \\
   16 & 272 & \num{2.06e-10} & 4.17 \\
   \bottomrule
    \end{tabular}
  \end{subtable}%
\end{table}

\begin{table}[htbp]
\centering
  \caption{Convergence results using curvilinear upwind SBP discretizations
           for the compressible Euler equations with
           $K$ elements, $N$ nodes per element, and an interior order
           of accuracy 8 on the unstructured quadratic curvilinear mesh
           shown in Figure~\ref{fig:conv_meshes}(b).}
  \label{tab:convergence_euler_curved3}
  \begin{subtable}{0.32\textwidth}
  \centering
    \caption{local Lax-Friedrichs spl. Ex.~\ref{ex:llf_splitting}.}
    \begin{tabular}{rrrr}
      \toprule
      $K$ & $N$ & $L^2$ error & EOC \\
      \midrule
   16 &  17 & \num{1.71e-05} & \\
   16 &  34 & \num{4.16e-07} & 5.36 \\
   16 &  68 & \num{1.17e-08} & 5.15 \\
   16 & 136 & \num{3.64e-10} & 5.01 \\
   16 & 272 & \num{1.15e-11} & 4.99 \\
      \bottomrule
    \end{tabular}
  \end{subtable}%
  \hspace{\fill}
  \begin{subtable}{0.32\textwidth}
  \centering
    \caption{Drikakis-Tsangaris splitting \cite{drikakis1993solution}.}
    \begin{tabular}{rrrr}
      \toprule
      $K$ & $N$ & $L^2$ error & EOC \\
      \midrule
   16 &  17 & \num{1.71e-05} & \\
   16 &  34 & \num{4.40e-07} & 5.28 \\
   16 &  68 & \num{1.32e-08} & 5.06 \\
   16 & 136 & \num{4.33e-10} & 4.93 \\
   16 & 272 & \num{1.33e-11} & 5.02 \\
      \bottomrule
    \end{tabular}
  \end{subtable}%
  \hspace{\fill}
  \begin{subtable}{0.32\textwidth}
  \centering
    \caption{van Leer-Hänel splitting \cite{anderson1986comparison}.}
    \begin{tabular}{rrrr}
      \toprule
      $K$ & $N$ & $L^2$ error & EOC \\
      \midrule
   16 &  17 & \num{1.66e-05} & \\
   16 &  34 & \num{4.51e-07} & 5.20 \\
   16 &  68 & \num{1.42e-08} & 4.99 \\
   16 & 136 & \num{4.84e-10} & 4.88 \\
   16 & 272 & \num{1.52e-11} & 4.99 \\
      \bottomrule
    \end{tabular}
  \end{subtable}%
\end{table}

The convergence results for the compressible Euler equations shown in
Tables~\ref{tab:convergence_euler_curved1}--\ref{tab:convergence_euler_curved3}
confirm the behavior of the experimental order of convergence
observed for the earlier one-dimensional convergence tests.

\subsection{Spectral analysis}

We consider the same linear advection setup with periodic boundary conditions
as in Section~\ref{sec:convergence_advection} and compute the spectra of the
semidiscretizations. The results visualized in Figure~\ref{fig:spectra}
demonstrate the linear stability of the upwind discretizations,
since the spectra are contained in the left half of the complex plane and
the maximum real part is zero up to machine precision.
Furthermore, they indicate that the stiffness of the methods (as measured by
the largest eigenvalue by magnitude) increases when
reducing the number of elements and increasing the number of nodes per element
such that the total number of degrees of freedom (DOFs) is constant.

\begin{figure}[htbp]
  \centering
    \begin{subfigure}{0.49\textwidth}
    \centering
      \includegraphics[width=\textwidth]{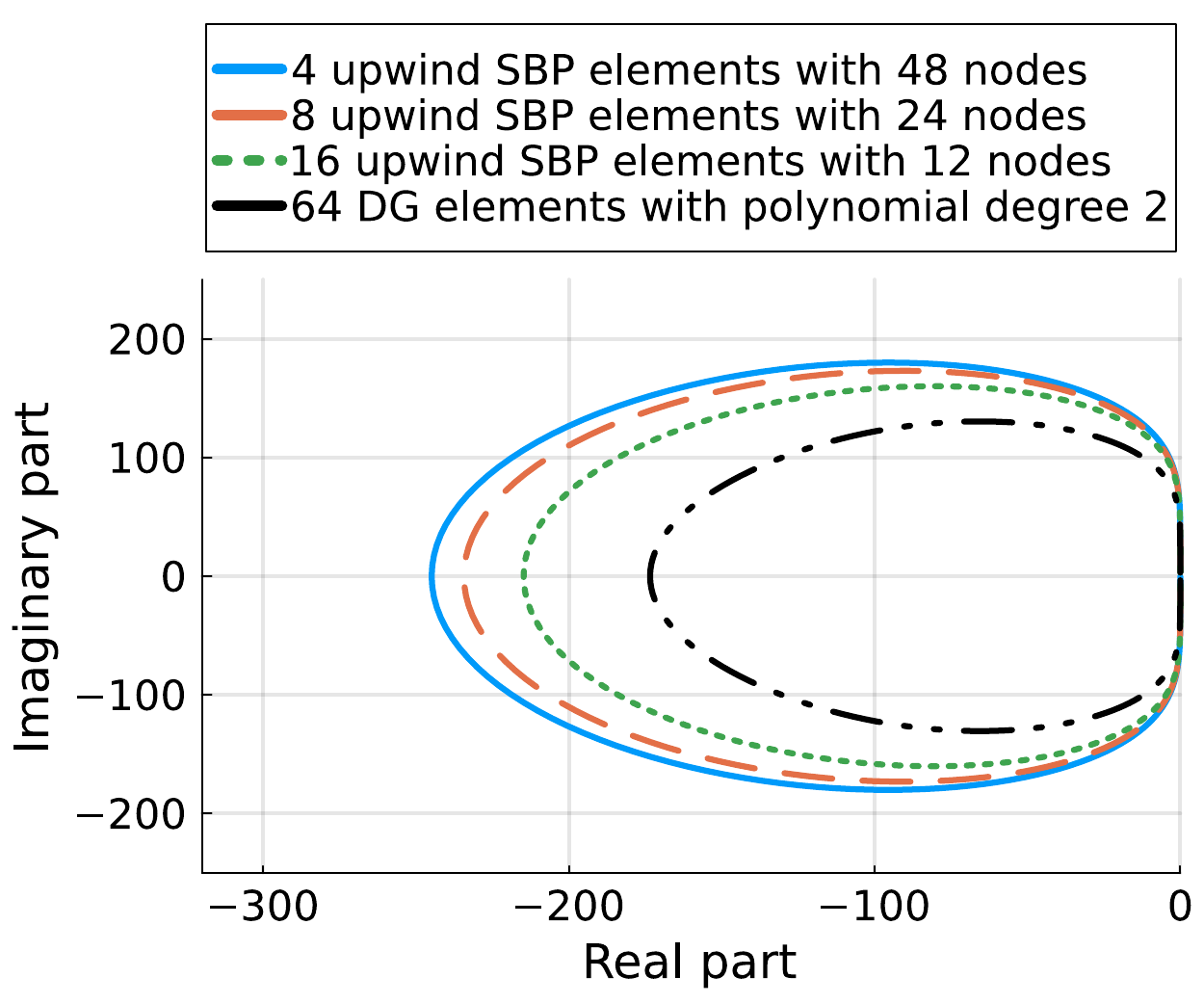}
      \caption{Upwind SBP methods with interior order of accuracy 4 and
               classical DGSEM with polynomial degree $p = 2$.}
    \end{subfigure}%
    \vspace{\fill}
    \begin{subfigure}{0.49\textwidth}
    \centering
      \includegraphics[width=\textwidth]{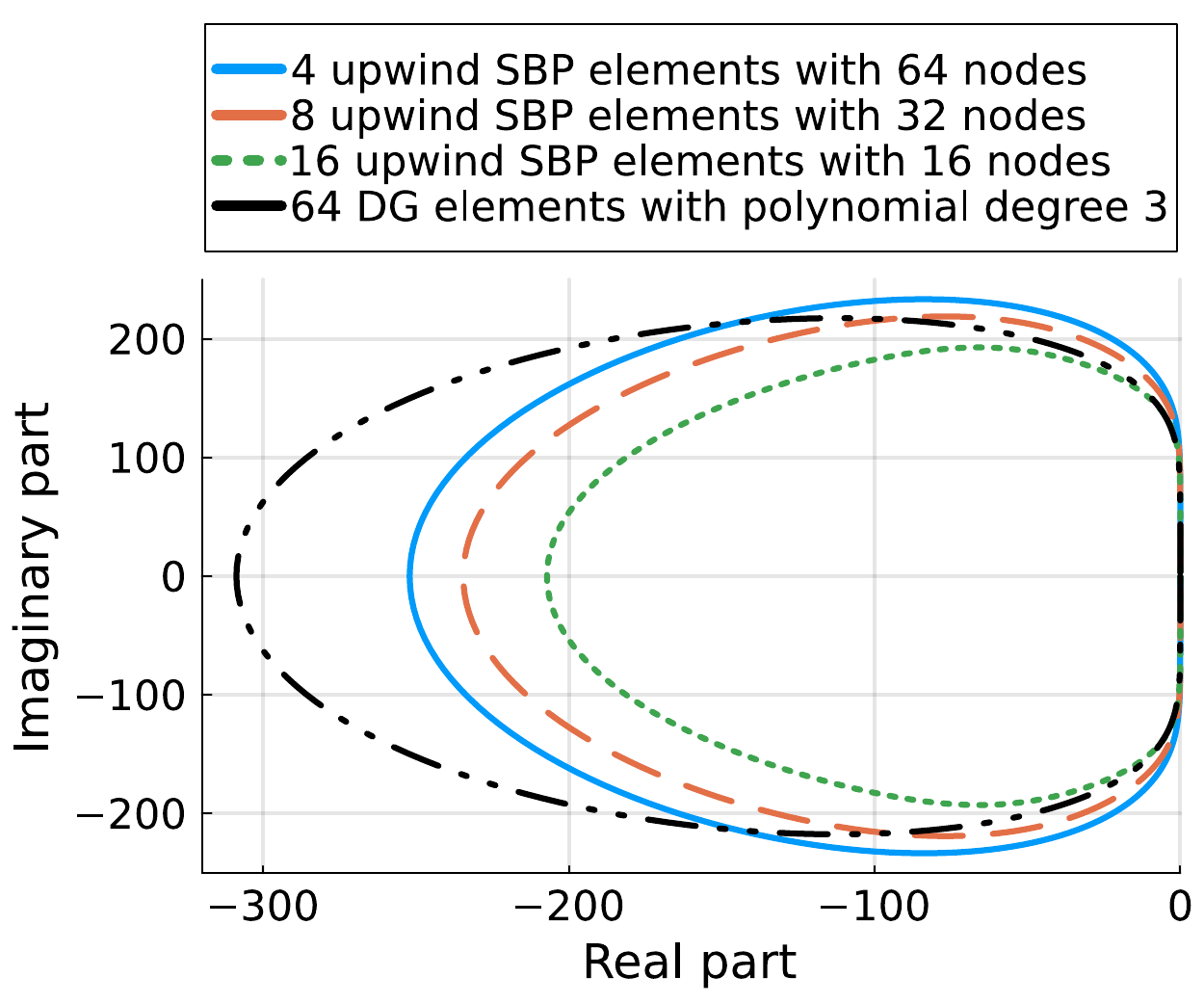}
      \caption{Upwind SBP methods with interior order of accuracy 6 and
               classical DGSEM with polynomial degree $p = 3$.}
    \end{subfigure}%
    \caption{Spectra of semidiscretizations of the 1D linear scalar advection
             equation with periodic boundary conditions. The maximum real part
             of all eigenvalues is around machine precision.}
    \label{fig:spectra}
  \end{figure}

In general, the spectra are comparable to the spectra obtained by the
classical nodal DGSEM method on Gauss-Lobatto-Legendre nodes. The spectra
shown in Figure~\ref{fig:spectra} suggest that the upwind SBP methods with
an interior order of accuracy 4 are stiffer than DGSEM with a polynomial
degree of 2; the situation is reversed for upwind SBP methods with an interior
of accuracy 6 and DGSEM with a polynomial degree of 3.

\subsection{Local linear/energy stability}

Next, we verify the local linear/energy stability properties discussed in
Section~\ref{sec:stability} numerically. For this, we discretize Burgers'
equation in the domain $(-1, 1)$ with periodic boundary conditions using
upwind SBP methods with a fully upwind discretization using only $D_-$.
To stress the methods, we consider a completely under-resolved case by
computing the Jacobian at a random non-negative state using automatic/algorithmic
differentiation via ForwardDiff.jl \cite{revels2016forward}.

\begin{table}[htbp]
\centering
  \caption{Maximal real part of the spectrum of upwind SBP discretizations
           of Burgers' equation with different interior order of accuracy $p$,
           $K$ elements, $N$ nodes per element, and a purely upwind discretization
           using only $D_-$.}
  \label{tab:spectrum_burgers}
  \begin{subtable}{0.32\textwidth}
  \centering
    \begin{tabular}{crrr}
      \toprule
      $p$ & $K$ & $N$ & \multicolumn{1}{c}{$\max \Re \sigma$} \\
      \midrule
      2 &  1 & 13 & \num{ 2.93e-16} \\
      2 &  1 & 14 & \num{-5.58e-16} \\
      2 &  2 & 13 & \num{ 2.62e-15} \\
      2 &  2 & 14 & \num{ 1.89e-15} \\
      3 &  1 & 13 & \num{ 3.08e-17} \\
      3 &  1 & 14 & \num{ 4.76e-16} \\
      3 &  2 & 13 & \num{ 5.51e-16} \\
      3 &  2 & 14 & \num{-2.78e-16} \\
      \bottomrule
    \end{tabular}
  \end{subtable}%
  \hspace{\fill}
  \begin{subtable}{0.32\textwidth}
  \centering
    \begin{tabular}{crrr}
      \toprule
      $p$ & $K$ & $N$ & \multicolumn{1}{c}{$\max \Re \sigma$} \\
      \midrule
      4 &  1 & 13 & \num{-4.67e-16} \\
      4 &  1 & 14 & \num{ 2.29e-16} \\
      4 &  2 & 13 & \num{-6.84e-17} \\
      4 &  2 & 14 & \num{ 3.39e-16} \\
      5 &  1 & 13 & \num{-1.67e-16} \\
      5 &  1 & 14 & \num{ 1.77e-17} \\
      5 &  2 & 13 & \num{ 2.11e-16} \\
      5 &  2 & 14 & \num{ 4.03e-16} \\
      \bottomrule
    \end{tabular}
  \end{subtable}%
  \hspace{\fill}
  \begin{subtable}{0.32\textwidth}
  \centering
    \begin{tabular}{crrr}
      \toprule
      $p$ & $K$ & $N$ & \multicolumn{1}{c}{$\max \Re \sigma$} \\
      \midrule
      6 &  1 & 13 & \num{-2.08e-16} \\
      6 &  1 & 14 & \num{-2.67e-16} \\
      6 &  2 & 13 & \num{-2.40e-16} \\
      6 &  2 & 14 & \num{-3.08e-16} \\
      7 &  1 & 13 & \num{ 1.65e-16} \\
      7 &  1 & 14 & \num{ 2.72e-16} \\
      7 &  2 & 13 & \num{ 4.33e-16} \\
      7 &  2 & 14 & \num{ 1.23e-16} \\
      \bottomrule
    \end{tabular}
  \end{subtable}%
\end{table}

The results are
shown in Table~\ref{tab:spectrum_burgers}. Clearly, the maximal positive real
part of the spectrum is around machine precision for 64 bit floating point
numbers in all cases.

\subsection{Free-stream preservation on unstructured meshes}
\label{sec:fsp_numerics}
In this section, we present numerical evidence for the proof of free-stream preservation
for the upwind SBP framework presented in Theorem~\ref{thm:curvi_fsp}. This theorem
found that more complicated splittings, like the Drikakis-Tsangaris, for the upwind method
in curvilinear coordinates are only guaranteed to be FSP
under a particular interplay between the boundary (or interface) polynomial degree of an
unstructured curvilinear mesh, the particular flux vector splitting,
and the boundary closure order of the upwind SBP operator.
The analysis in Section~\ref{sec:curvilinear} also demonstrated that FSP is easily obtained
provided the splitting technique remains linear as a function of the metrics terms, as was
the case for the local Lax-Friedrichs splitting.

We reiterate that on a Cartesian box mesh, where all metric terms are constants proportional
to fixed values of $\Delta x$ or $\Delta y$, there is no issue with FSP.
It is only when we move the approximation into generalized curvilinear coordinates that one
must take care of the mapping, the metric terms, and their approximation strategy.
The importance, and subtleties, of the discrete approximation of the metric terms
has been known for decades across different computational fluid dynamics
communities, e.g., \cite{vinokur2002extension,kopriva2006metric,visbal1999high}.

Setting up an FSP test is straightforward and, at a glance, fairly innocuous.
A constant solution should remain constant (up to machine precision) for all time as indicated
by the governing equations \eqref{eq:HCL-2D} with appropriate boundary
conditions. For the test herein we consider the compressible Euler
equations in two space dimensions \eqref{eq:2d_comp_euler}.
Given the free-stream solution state of the conservative variables
\begin{equation}
\label{eq:fsp_sol}
u_\infty
=
\begin{pmatrix}
\rho_\infty\\
(\rho v_1)_\infty\\
(\rho v_2)_\infty\\
(\rho e)_\infty
\end{pmatrix}
=
\begin{pmatrix}
1.0\\
0.1\\
-0.2\\
10.0
\end{pmatrix},
\end{equation}
the fluxes are all constant and their divergence vanishes on the continuous level.
However, in the discrete setting this is (potentially) not always the case.

For the tests in this section we consider a domain with a circular outer boundary
and an interior boundary composed of two straight lines and a semicircle.
This domain is then divided into 204 non-overlapping quadrilateral elements.
We create two versions of the mesh presented in Figure~\ref{fig:fsp_meshes}:
one composed only of bi-linear elements and the other with bi-linear elements in
the interior and boundaries approximated with quadratic polynomials.

\begin{figure}[htbp]
\centering
  \begin{subfigure}{0.44\textwidth}
  \centering
    \includegraphics[width=\textwidth]{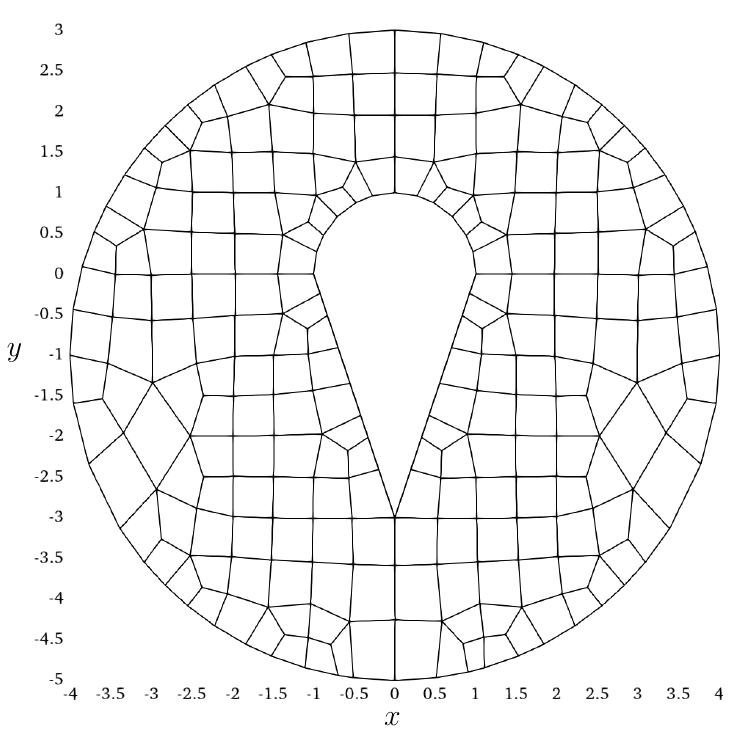}
    \caption{Mesh with linear boundary polynomials.}
  \end{subfigure}%
  \hspace*{\fill}
  \begin{subfigure}{0.44\textwidth}
  \centering
    \includegraphics[width=\textwidth]{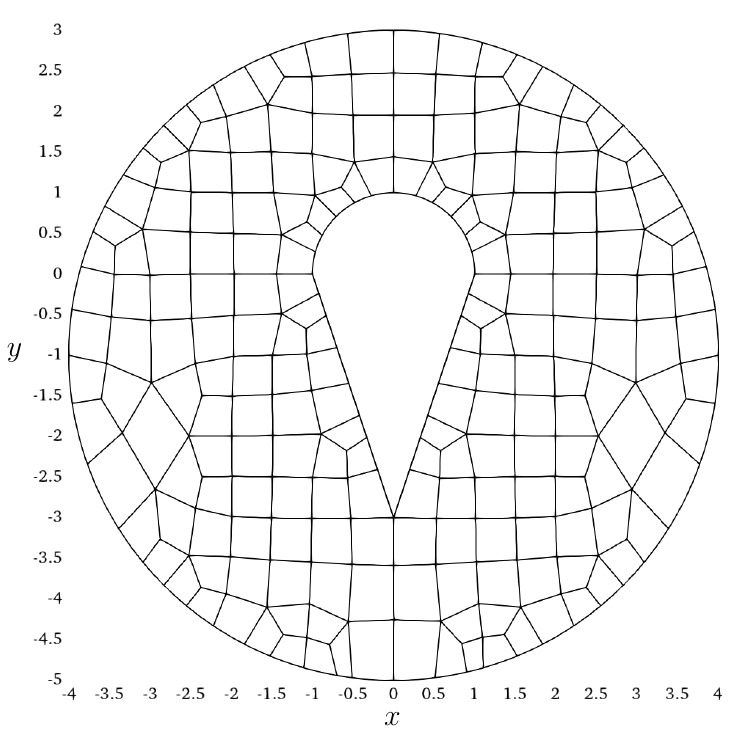}
    \caption{Mesh with quadratic boundary polynomials.}
  \end{subfigure}%
  \caption{Non-overlapping quadrilateral meshes used for the free-stream preservation testing.}
  \label{fig:fsp_meshes}
\end{figure}

Whether or not a particular upwind SBP operator is FSP depends upon the splitting,
the boundary closure accuracy, and the polynomial degree of the curvilinear
boundary approximations in the mesh.
We use the two meshes shown in Figure~\ref{fig:fsp_meshes} to examine the theoretical
finding of Theorem~\ref{thm:curvi_fsp} for different combinations of the upwind SBP operator,
the flux vector splitting in generalized coordinates, and the polynomial degree of the mesh.
For the FSP testing we fix the spatial resolution to be $17$ nodes in each spatial direction
and integrate the constant solution initial condition \eqref{eq:fsp_sol} up to a final time of 10.
We present results of the FSP test in Tables~\ref{tab:fsp_llf}--\ref{tab:fsp_vLH} where we vary
the curvilinear splitting and consider the upwind SBP operators provided by
Mattsson~\cite{mattsson2017diagonal} from interior order 2 up to interior order 9.

\begin{table}[htbp]
\centering
  \caption{Free-stream preservation error at final time $10$ for the local Lax-Friedrichs splitting
  on two mesh types with $17$ nodes in each spatial direction. As expected from the results in
  Section~\ref{sec:curvilinear}, the local Lax-Friedrichs splitting is FSP across all configurations.}
  \label{tab:fsp_llf}
   \begin{adjustbox}{max width=\textwidth}
    \begin{tabular}{ccccccccc}
      \toprule
      interior order & 2 & 3 & 4 & 5 & 6 & 7 & 8 & 9 \\
      \midrule
      bi-linear mesh & \num{7.55e-14} & \num{1.16e-13} & \num{6.28e-14} & \num{6.99e-15} & \num{1.52e-13} & \num{1.16e-13} & \num{4.35e-14} & \num{3.89e-14} \\
      quadratic mesh & \num{2.38e-14} & \num{4.72e-14} & \num{8.41e-14} & \num{5.48e-14} & \num{1.22e-13} & \num{1.47e-14} & \num{3.83e-14} & \num{4.18e-14} \\
      \bottomrule
    \end{tabular}
    \end{adjustbox}
\end{table}

\begin{table}[htbp]
\centering
  \caption{Free-stream preservation error at final time $10$ for the Drikakis-Tsangaris splitting
  on two mesh types with $17$ nodes in each spatial direction. As shown from the result of
  Theorem~\ref{thm:curvi_fsp}, this splitting is FSP provided the boundary closure is accurate
  enough to exactly differentiate polynomials of degree $2\Ngeo$.}
  \label{tab:fsp_drikakis}
    \begin{adjustbox}{max width=\textwidth}
    \begin{tabular}{ccccccccc}
      \toprule
      interior order & 2 & 3 & 4 & 5 & 6 & 7 & 8 & 9 \\
      \midrule
      bi-linear mesh & \num{2.04e-6} & \num{9.13e-7} & \num{3.34e-14} & \num{1.77e-14}& \num{3.34e-14} & \num{1.24e-14} & \num{8.20e-14} & \num{2.05e-14}\\
      quadratic mesh & \num{2.08e-6} & \num{9.32e-7} & \num{5.74e-9} & \num{2.75e-9} & \num{7.09e-11} & \num{3.33e-11} & \num{2.18e-14} & \num{1.55e-14} \\
      \bottomrule
    \end{tabular}
    \end{adjustbox}
\end{table}

\begin{table}[htbp]
\centering
  \caption{Free-stream preservation error at final time $10$ for the van Leer-Hänel splitting
  on two mesh types with $17$ nodes in each spatial direction. As shown from the result of
  Theorem~\ref{thm:curvi_fsp}, this splitting is FSP provided the boundary closure is accurate
  enough to exactly differentiate polynomials of degree $2\Ngeo$.}
  \label{tab:fsp_vLH}
   \begin{adjustbox}{max width=\textwidth}
    \begin{tabular}{ccccccccc}
      \toprule
      interior order & 2 & 3 & 4 & 5 & 6 & 7 & 8 & 9 \\
      \midrule
      bi-linear mesh & \num{3.32e-6} & \num{1.46e-6} & \num{4.78e-14} & \num{5.26e-15} & \num{3.36e-14} & \num{1.02e-14} & \num{2.66e-14} & \num{1.95e-14} \\
      quadratic mesh & \num{3.40e-6} & \num{1.49e-6} & \num{8.78e-9} & \num{4.20e-9} & \num{1.06e-10} & \num{4.98e-11} & \num{1.62e-14} & \num{5.43e-14}\\
      \bottomrule
    \end{tabular}
    \end{adjustbox}
\end{table}

As anticipated from the discussion in Section~\ref{sec:curvilinear} the local Lax-Friedrichs splitting
maintains FSP for both meshes and all operator orders. The results also support the conclusion
of Theorem~\ref{thm:curvi_fsp} for the more complicated Drikakis-Tsangaris and van Leer-Hänel
splittings. Both of these splittings have a maximum quadratic dependence on the metric terms.
We see that the operators with interior accuracy 2 and 3 (and boundary closure accuracy 1)
are not FSP for either the Drikakis-Tsangaris nor van Lerr-Hänel splittings due to the lack of accuracy in the boundary closures.
The numerical results show that FSP is maintained on the bi-linear test mesh for all operator orders
above interior order 4 whereas FSP is only maintained on the quadratic test mesh for the 8-4 and 9-4 operators.
All upwind SBP operator with interior order less than eight lack the required boundary closure accuracy
to guarantee FSP on a quadratic mesh.

\subsection{Isentropic vortex}

We consider the classical isentropic vortex test case of
\cite{shu1997essentially} with initial conditions
\begin{equation}
\label{eq:isen_vort}
\begin{gathered}
  T = T_0 - \frac{(\gamma-1) \epsilon^2}{8 \gamma \pi^2} \exp\bigl(1-r^2\bigr),
  \quad
  \rho = \rho_0 (T / T_0)^{1 / (\gamma - 1)},
  \\
  v = v_0 + \frac{\varepsilon}{2 \pi} \exp\bigl((1-r^2) / 2\bigr) (-x_2, x_1)^T,
\end{gathered}
\end{equation}
where $\epsilon = 10$ is the vortex strength,
$r$ is the distance from the origin,
$T = p / \rho$ the temperature, $\rho_0 = 1$ the background density,
$v_0 = (1, 1)^T$ the background velocity,
$p_0 = 10$ the background pressure, $\gamma = 1.4$, and
$T_0 = p_0 / \rho_0$ the background temperature.
The domain $[-5, 5]^2$ is equipped with periodic boundary conditions.

\begin{figure}[htbp]
\centering
  \begin{subfigure}{0.49\textwidth}
  \centering
    \includegraphics[width=\textwidth]{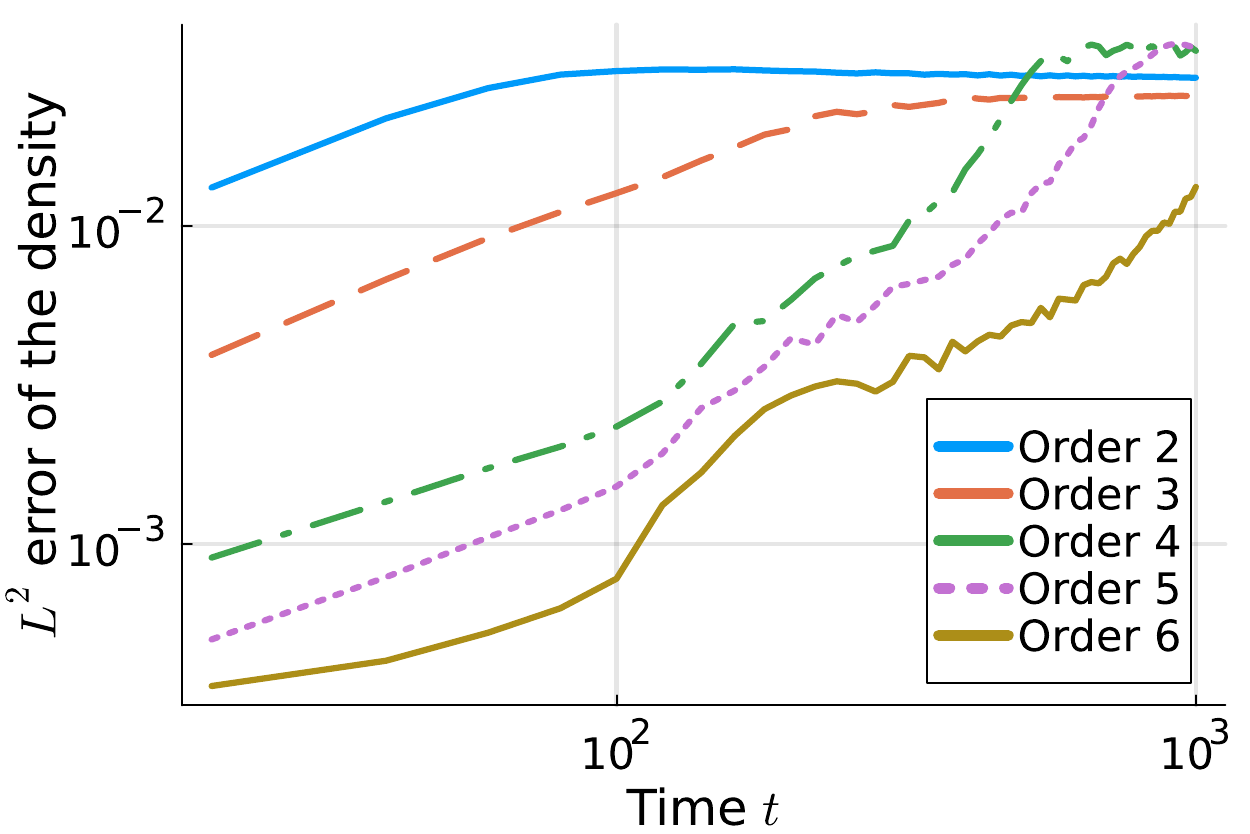}
    \caption{$4^2$ blocks with $16^2$ nodes each using the operators of
             \cite{mattsson2017diagonal} with different interior order of
             accuracy.}
  \end{subfigure}%
  \hspace*{\fill}
  \begin{subfigure}{0.49\textwidth}
  \centering
    \includegraphics[width=\textwidth]{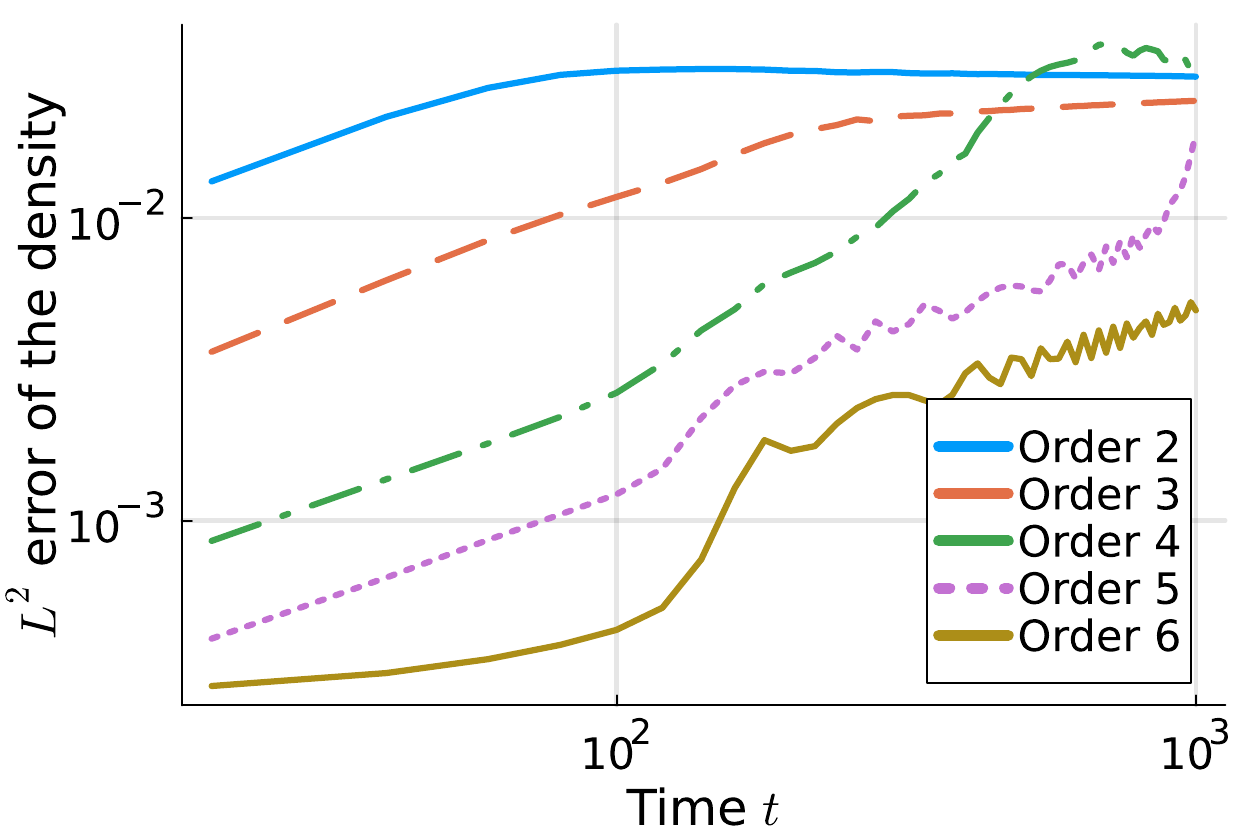}
    \caption{A single block with $64^2$ nodes and fully periodic operators using
             only the interior stencils of \cite{mattsson2017diagonal}.}
  \end{subfigure}%
  \caption{Discrete $L^2$ of the density for long-time simulations of the
           isentropic vortex for the 2D compressible Euler equations.}
  \label{fig:isentropic_vortex_error}
\end{figure}

Following \cite{sjogreen2018high}, we use this setup to demonstrate the
long-time stability of the methods. We use the same time integration
method and approach to compute the discrete $L^2$ error of the density
as in Section~\ref{sec:convergence_euler}.
As shown in Figure~\ref{fig:isentropic_vortex_error}, the upwind methods
remain stable and are able to run the simulations successfully for long times.

To demonstrate the robustness of the upwind methods on curvilinear meshes
we, again, consider the isentropic vortex \eqref{eq:isen_vort} with
$\varepsilon=5$ and $p_0 = 25$. We take
$\Omega = [-10,10]^2$ and subdivide the domain with eight elements in each spatial direction
for a total of 64 elements. The Cartesian domain of $\Omega = [-10,10]^2$ is then heavily
warped with a strategy adapted from \cite{hennemann2021provably,chan2019efficient}
where
\begin{equation}
\label{eq:warping}
\begin{aligned}
y &= \eta + \frac{L_y}{8}\cos\!\left(\frac{3\pi}{2}\left(\frac{2\xi - L_x}{L_x}\right)\right) \cos\!\left(\frac{\pi}{2}\left(\frac{2\eta - L_y}{L_y}\right)\right),\\[0.1cm]
x &= \xi + \frac{L_x}{8}\cos\!\left(\frac{\pi}{2}\left(\frac{2\xi - L_x}{L_x}\right)\right) \cos\!\left(2\pi\left(\frac{2y - L_y}{L_y}\right)\right),
\end{aligned}
\end{equation}
with $L_x = L_y = 10$.
All domain boundaries remain periodic under this mapping.
The resulting mesh, given as the overlay of curvilinear quadrilaterals in Figure~\ref{fig:warped_vortex},
is extremely distorted with many element that approach degeneracy. That is, several elements in the mesh
are close to having an internal angle near 180 degrees that renders the transfinite interpolation procedure
to create the element mapping ill-conditioned. This extreme warping to a ``poor'' quality mesh was
purposely done to demonstrate that the upwind methods remain robust for the isentropic vortex test case
even in this extreme problem setup.

\begin{figure}[htbp]
\centering
  \begin{subfigure}{0.325\textwidth}
  \centering
    \includegraphics[width=\textwidth]{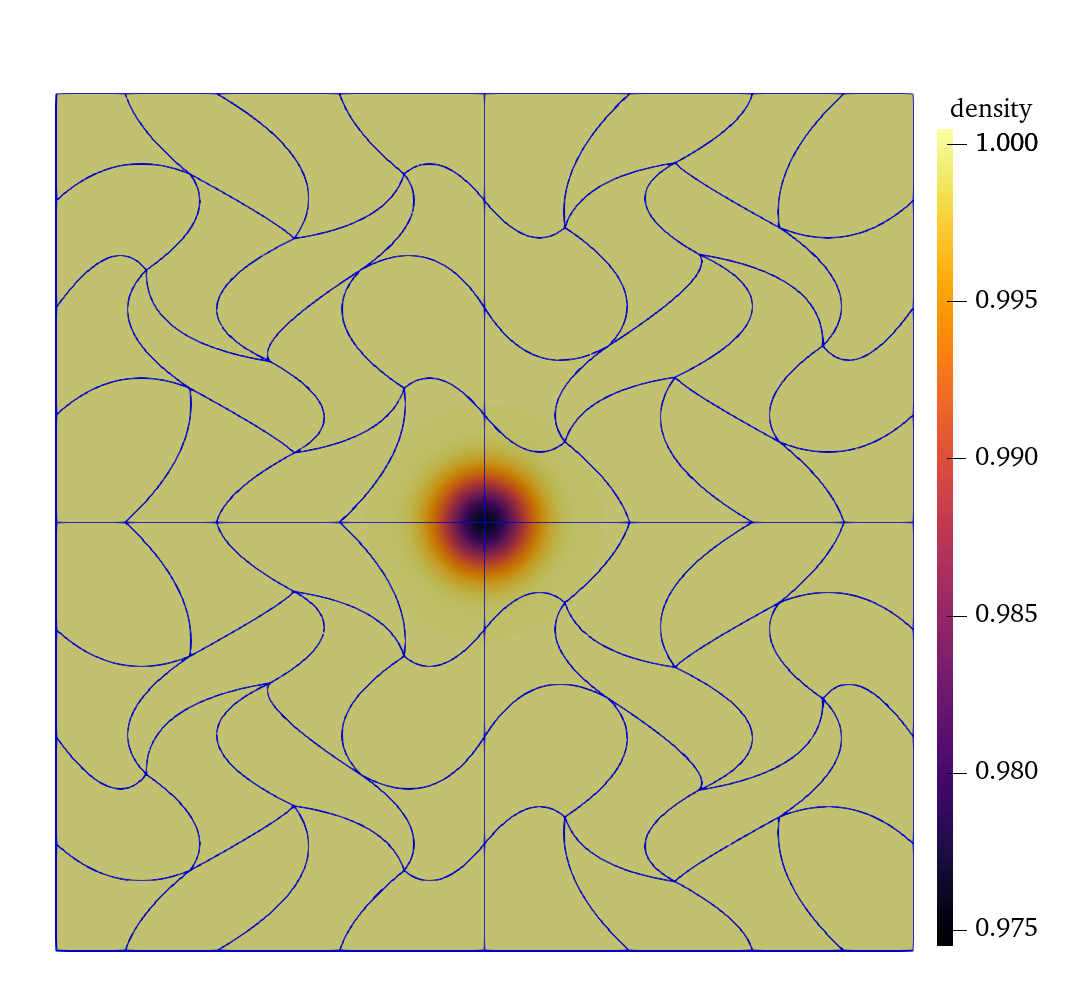}
    \caption{$t=0$}
  \end{subfigure}%
  \hspace{\fill}
  \begin{subfigure}{0.325\textwidth}
  \centering
    \includegraphics[width=\textwidth]{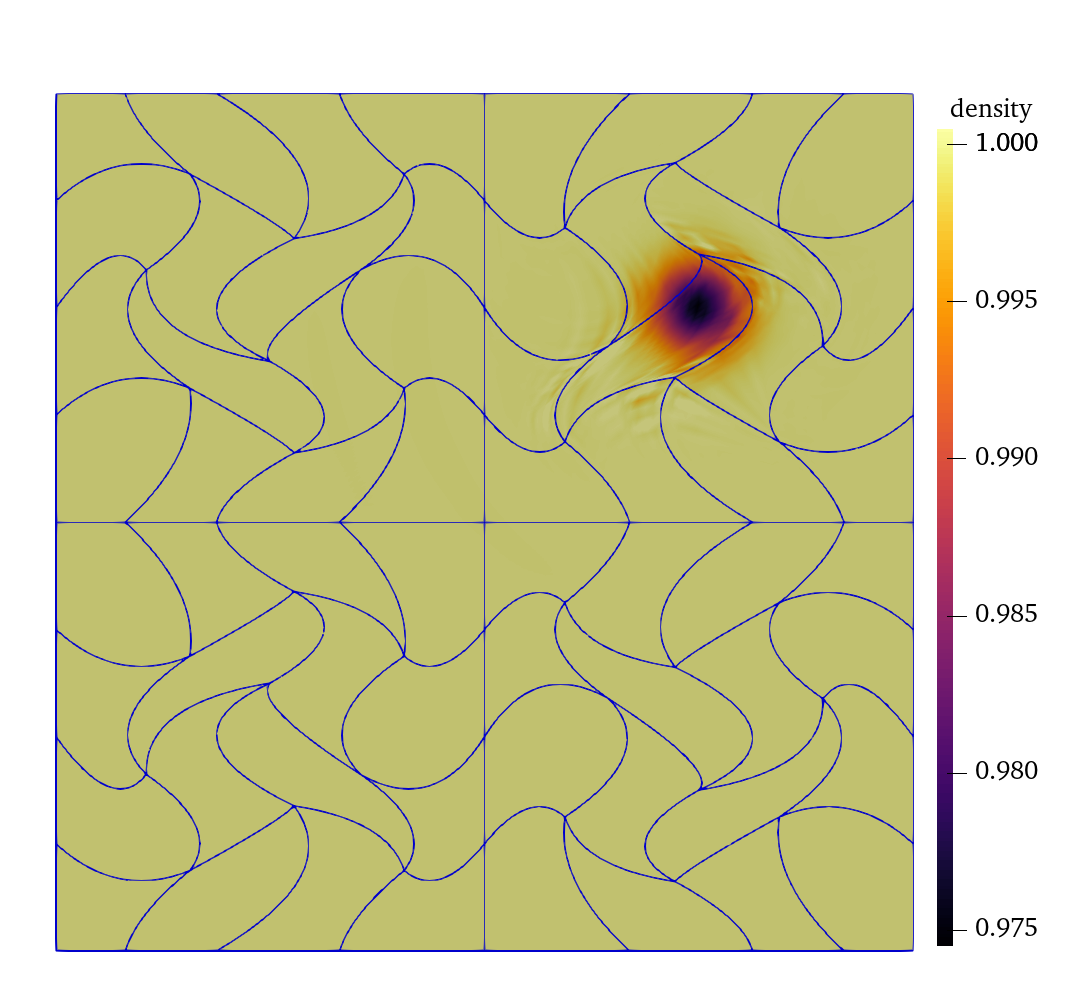}
    \caption{$t=5$}
  \end{subfigure}%
  \hspace{\fill}
  \begin{subfigure}{0.325\textwidth}
  \centering
    \includegraphics[width=\textwidth]{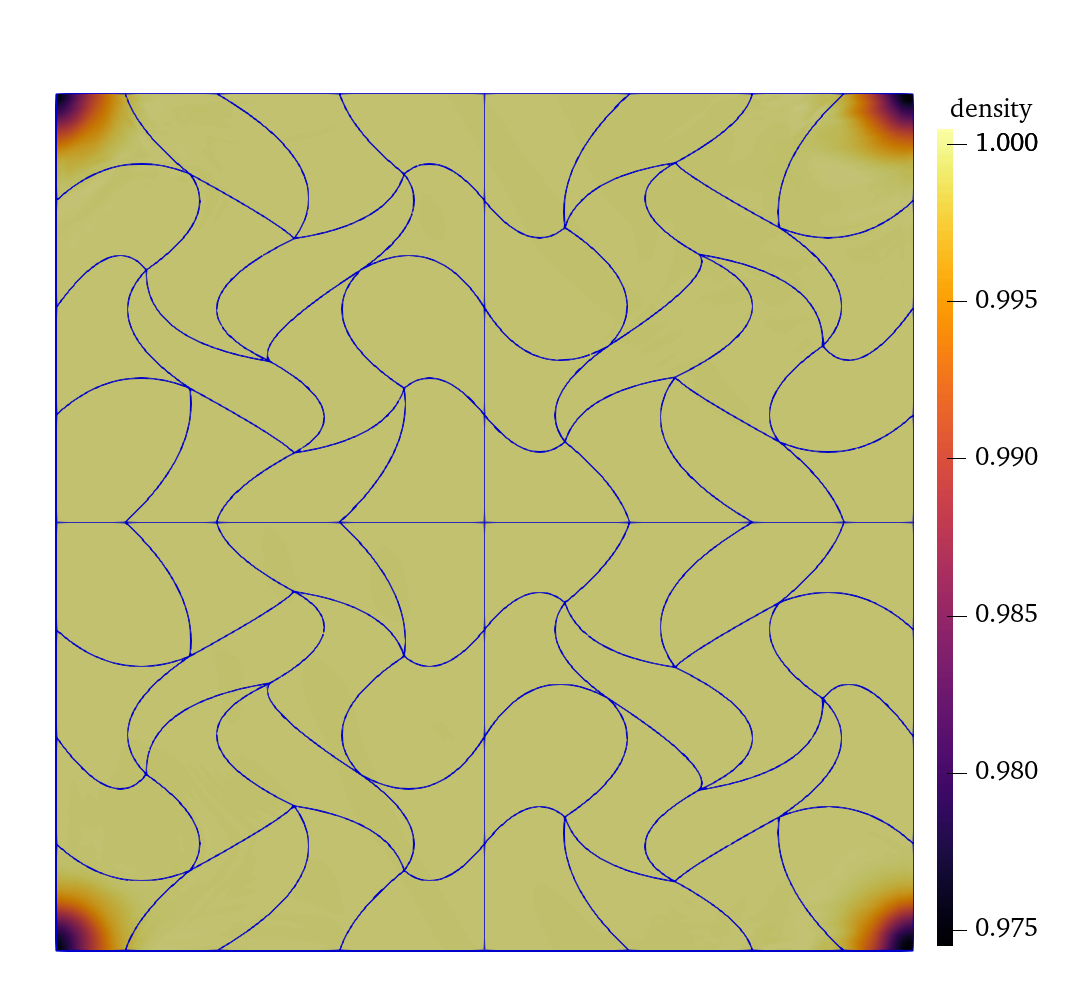}
    \caption{$t=10$}
  \end{subfigure}%
  \\
  \begin{subfigure}{0.325\textwidth}
  \centering
    \includegraphics[width=\textwidth]{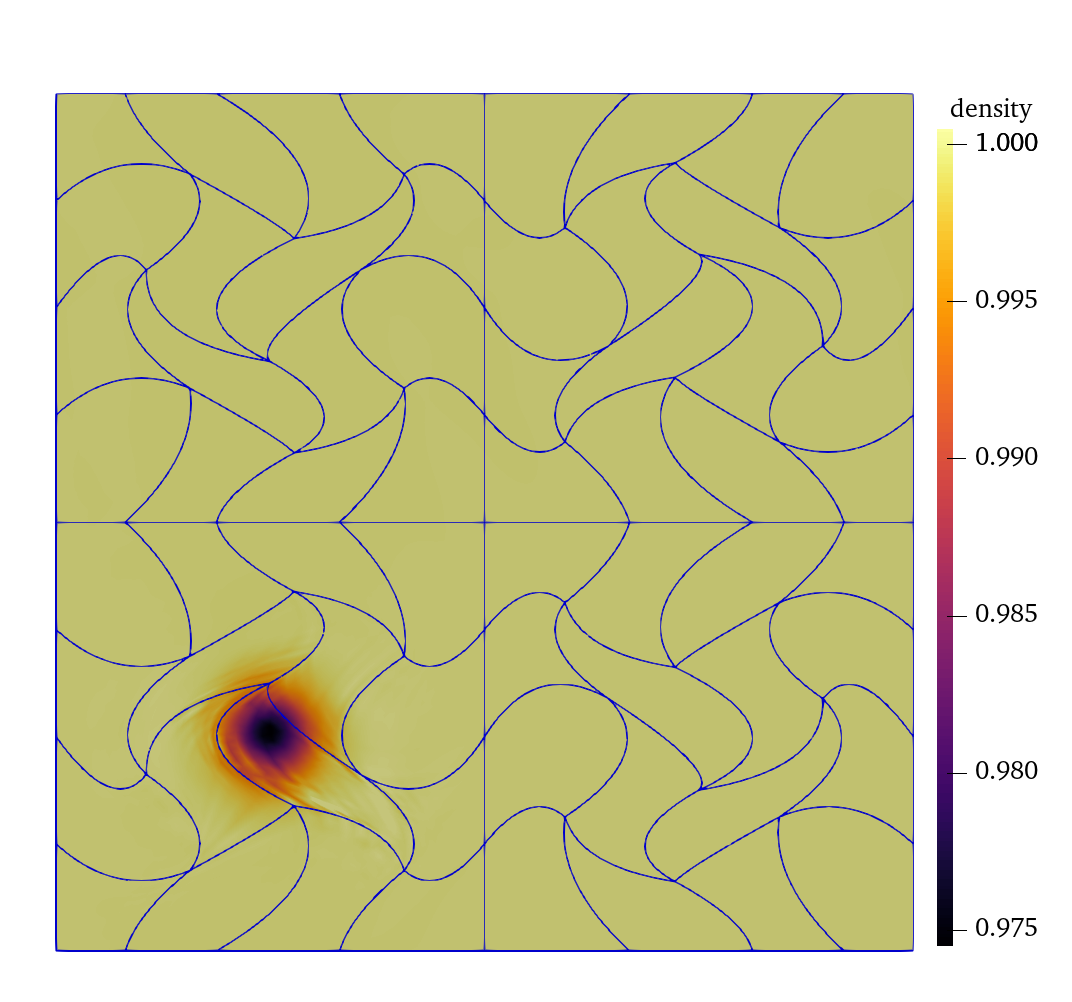}
    \caption{$t=15$}
  \end{subfigure}%
   \hspace{0.15cm}
  \begin{subfigure}{0.325\textwidth}
  \centering
    \includegraphics[width=\textwidth]{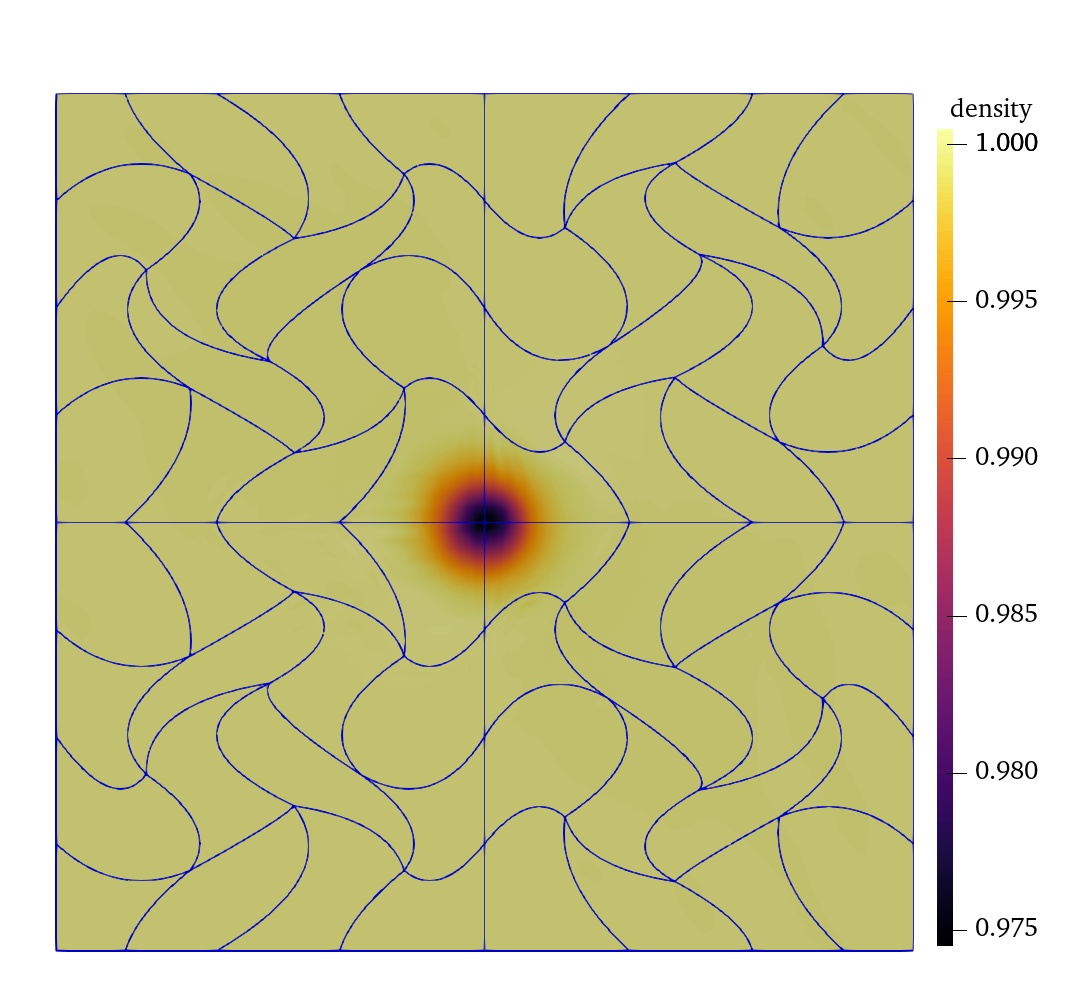}
    \caption{$t=20$}
  \end{subfigure}%
  \caption{Isentropic vortex evolution up to final time $t=20$ on a heavily distorted quadrilateral mesh
  of 64 elements. All curvilinear interior boundaries of the warping \eqref{eq:warping} are approximated with quadratic
  polynomials. This result used the curvilinear van Leer-Hänel splitting, Example \ref{ex:van-Leer-Hanel_curved}, with 17 nodes in each spatial direction
  and the 8th order interior, 4th order boundary accurate upwind SBP operators.}
  \label{fig:warped_vortex}
\end{figure}

We approximate the internal curved boundaries with quadratic polynomials. On each element we take
17 nodes in each spatial direction and use the upwind SBP operator that is 8th order in the interior
with 4th order boundary closures. We present in Figure~\ref{fig:warped_vortex} the results at different
times between $t=0$ and $t=20$ using the curvilinear van Leer-Hänel splitting from Example~\ref{ex:van-Leer-Hanel_curved}.
Although there are some grid artifacts
as the vortex passes through extremely distorted elements, the method maintains the shape of the vortex well.
In Figure~\ref{fig:warped_vortex_long_time} we show the $L^2$ density error for a long time simulation
for the curvilinear local Lax-Friedrichs, Drikakis-Tsangaris, and van Leer-Hänel splittings.
All three splittings of this heavily distorted curvilinear mesh remain stable. Because the
test case configuration is well-resolved, any differences in density errors between the three
splitting techniques are unnoticeable in the eyeball norm.
\begin{figure}[htbp]
  \centering
    \includegraphics[width=0.5\textwidth]{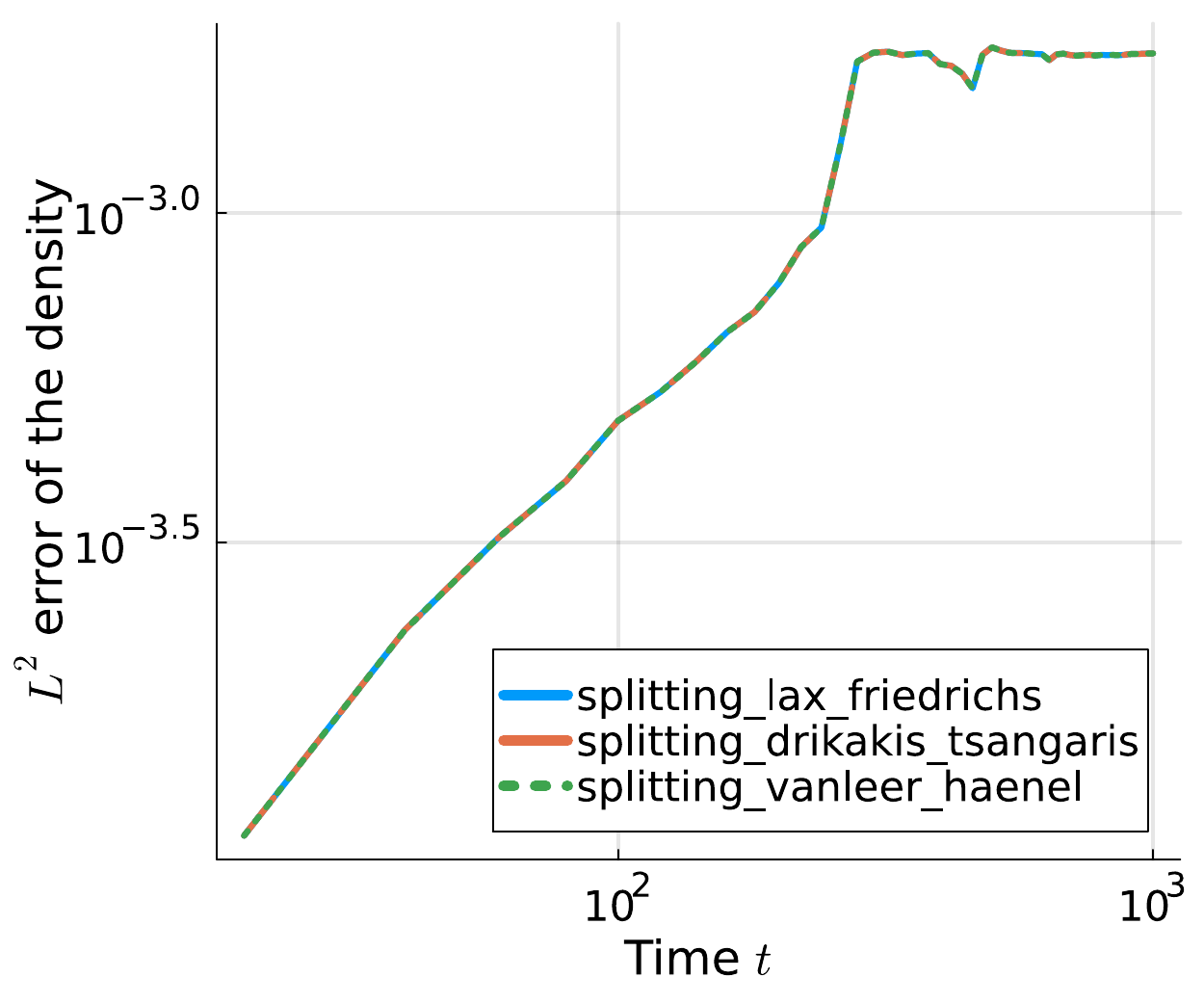}
    \caption{Discrete $L^2$ of the density for long-time simulation of three curvilinear splittings
    applied to the isentropic vortex for the 2D compressible Euler equations. Each run used the
    same heavily distorted quadrilateral mesh.}
  \label{fig:warped_vortex_long_time}
\end{figure}

\subsection{Kelvin-Helmholtz instability}
\label{sec:khi}

Next, we
move beyond well-resolved test configurations and use a Kelvin-Helmholtz instability
setup for the 2D compressible Euler equations of an ideal fluid to further test the robustness of the methods and their different splittings
in various under-resolved regimes in detail.
Specifically, we use the same setup as in \cite{chan2022entropy}, i.e., the
initial condition
\begin{equation}
  \rho = \frac{1}{2} + \frac{3}{4} B(x,y),
  \quad
  p = 1,
  \quad
  v_1 = \frac{1}{2} \bigl( B(x,y) - 1 \bigr),
  \quad
  v_2 = \frac{1}{10}\sin(2 \pi x),
\end{equation}
where $B(x,y)$ is the smoothed approximation
\begin{equation}
  B(x, y) = \tanh(15 y + 7.5) - \tanh(15 y - 7.5)
\end{equation}
to a discontinuous step function. The domain is $[-1,1]^2$ with time interval
$[0, 15]$. We integrate the semidiscretizations in time with the third-order,
four-stage SSP method of \cite{kraaijevanger1991contractivity} with embedded
method of \cite{conde2022embedded} and error-based step size controller
developed in \cite{ranocha2021optimized} with absolute and relative tolerances
chosen as $10^{-6}$.

\begin{table}[htbp]
\centering
  \caption{Final times of numerical simulations of the Kelvin-Helmholtz
           instability with $K$ elements using 16 nodes per coordinate
           direction for the upwind SBP methods. Final times less than 15
           indicate that the simulation crashed.}
  \label{tab:kelvin_helmholtz}
  \begin{subtable}{0.49\textwidth}
  \centering
    \caption{Upwind SBP, van Leer-Hänel splitting \cite{vanleer1982flux,hanel1987accuracy,liou1991high}.}
    \begin{tabular}{r rrrrrr}
      \toprule
      \multicolumn{1}{c}{$K$} & \multicolumn{6}{c}{interior order of accuracy} \\
          & \multicolumn{1}{c}{2}
          & \multicolumn{1}{c}{3}
          & \multicolumn{1}{c}{4}
          & \multicolumn{1}{c}{5}
          & \multicolumn{1}{c}{6}
          & \multicolumn{1}{c}{7} \\
      \midrule
        1 & 15.0 & 15.0 & 15.0 & 15.0 & 15.0 & 15.0 \\
        4 & 15.0 & 15.0 & 15.0 & 15.0 & 15.0 & 15.0 \\
      16 & 15.0 & 15.0 & 15.0 & 15.0 & 4.72 & 3.97 \\
      64 & 15.0 & 15.0 & 4.53 & 3.86 & 4.17 & 3.36 \\
      256 & 15.0 & 15.0 & 5.80 & 3.70 & 3.66 & 3.68 \\
      \bottomrule
    \end{tabular}
  \end{subtable}%
  \hspace{\fill}
  \begin{subtable}{0.49\textwidth}
  \centering
    \caption{Upwind SBP, Steger-Warming splitting \cite{steger1979flux}.}
    \begin{tabular}{r rrrrrr}
      \toprule
      \multicolumn{1}{c}{$K$} & \multicolumn{6}{c}{interior order of accuracy} \\
          & \multicolumn{1}{c}{2}
          & \multicolumn{1}{c}{3}
          & \multicolumn{1}{c}{4}
          & \multicolumn{1}{c}{5}
          & \multicolumn{1}{c}{6}
          & \multicolumn{1}{c}{7} \\
      \midrule
        1 & 15.0 & 15.0 & 15.0 & 15.0 & 15.0 & 15.0 \\
        4 & 15.0 & 15.0 & 15.0 & 15.0 & 15.0 & 15.0 \\
      16 & 15.0 & 15.0 & 15.0 & 15.0 & 4.87 & 3.88 \\
      64 & 15.0 & 15.0 & 4.55 & 3.85 & 4.13 & 4.07 \\
      256 & 15.0 & 15.0 & 5.80 & 3.69 & 3.66 & 3.67 \\
      \bottomrule
    \end{tabular}
  \end{subtable}%
  \\
  \begin{subtable}{0.49\textwidth}
  \centering
    \caption{Flux differencing DGSEM, flux of Ranocha
             \cite{ranocha2018comparison,ranocha2020entropy,ranocha2021preventing}.}
    \begin{tabular}{r rrrrrr}
      \toprule
      \multicolumn{1}{c}{$K$} & \multicolumn{6}{c}{polynomial degree} \\
          & \multicolumn{1}{c}{2}
          & \multicolumn{1}{c}{3}
          & \multicolumn{1}{c}{4}
          & \multicolumn{1}{c}{5}
          & \multicolumn{1}{c}{6}
          & \multicolumn{1}{c}{7} \\
      \midrule
      16 & 15.0 & 4.46 & 2.47 & 3.01 & 2.80 & 3.59 \\
      64 & 4.68 & 1.53 & 4.04 & 3.70 & 4.10 & 3.56 \\
      256 & 4.81 & 3.77 & 4.44 & 3.74 & 3.37 & 3.64 \\
      1024 & 4.12 & 3.66 & 4.27 & 3.54 & 3.66 & 3.56 \\
      \bottomrule
    \end{tabular}
  \end{subtable}%
  \hspace{\fill}
  \begin{subtable}{0.49\textwidth}
  \centering
    \caption{Flux differencing DGSEM, flux of Shima et al.
             \cite{shima2021preventing}.}
    \begin{tabular}{r rrrrrr}
      \toprule
      \multicolumn{1}{c}{$K$} & \multicolumn{6}{c}{polynomial degree} \\
          & \multicolumn{1}{c}{2}
          & \multicolumn{1}{c}{3}
          & \multicolumn{1}{c}{4}
          & \multicolumn{1}{c}{5}
          & \multicolumn{1}{c}{6}
          & \multicolumn{1}{c}{7} \\
      \midrule
      16 & 15.0 & 2.73 & 1.81 & 2.42 & 1.86 & 2.27 \\
      64 & 2.92 & 1.38 & 3.05 & 3.07 & 1.82 & 2.02 \\
      256 & 3.25 & 2.82 & 3.29 & 2.82 & 2.84 & 2.96 \\
      1024 & 3.03 & 2.88 & 3.36 & 2.91 & 3.08 & 3.25 \\
      \bottomrule
    \end{tabular}
  \end{subtable}%
\end{table}

We use two types of semidiscretizations: i) the upwind SBP methods described
in this article and ii) flux differencing DGSEM with different volume fluxes
and a local Lax-Friedrichs (Rusanov) surface flux. We refer to
\cite{ranocha2023efficient} for a description of this DGSEM variant and its
efficient implementation and to \cite{lukacova2023convergence} for convergence results.

The final times of the simulations are summarized in
Table~\ref{tab:kelvin_helmholtz}.
First, it is interesting to observe that all upwind SBP methods with few
numbers of elements $K \in \{1, 4\}$ completed the simulation successfully.
The same is true for low interior orders of accuracy $\in \{2, 3\}$.
However, setups with more elements and higher orders of accuracy became
unstable and crashed before $t = 5$.

The same general trend can be observed for flux differencing DGSEM, where nearly
all setups became unstable and crashed. It is particularly interesting that
the simulations with only a few elements remained stable, even if their total
number of DOFs is comparable to DGSEM setups. For example,
the upwind SBP methods with $K = 4$ elements and 16 nodes per coordinate
direction have $K \cdot 16^2 = 1024$ DOFs, the same amount
as the DGSEM with $K = 64$ elements and a polynomial degree $p = 3$.

\begin{table}[htbp]
\centering
  \caption{Final times of numerical simulations of the Kelvin-Helmholtz
           instability with $K$ elements using $256 / \sqrt{K}$ nodes per
           coordinate direction so that the total number of DOFs
           stays fixed at $\num{65536}$. Final times less than 15 indicate that the
           simulation crashed.}
  \label{tab:kelvin_helmholtz_constant_dofs}
  \begin{subtable}{0.49\textwidth}
  \centering
    \caption{Upwind SBP, van Leer-Hänel splitting \cite{vanleer1982flux,hanel1987accuracy,liou1991high}.}
    \begin{tabular}{r rrrrrr}
      \toprule
      \multicolumn{1}{c}{$K$} & \multicolumn{6}{c}{interior order of accuracy} \\
          & \multicolumn{1}{c}{2}
          & \multicolumn{1}{c}{3}
          & \multicolumn{1}{c}{4}
          & \multicolumn{1}{c}{5}
          & \multicolumn{1}{c}{6}
          & \multicolumn{1}{c}{7} \\
      \midrule
        1 & 15.0 & 15.0 & 15.0 & 15.0 & 5.83 & 4.73 \\
        4 & 15.0 & 15.0 & 6.35 & 15.0 & 5.83 & 4.72 \\
      16 & 15.0 & 15.0 & 15.0 & 15.0 & 5.18 & 4.00 \\
      64 & 15.0 & 15.0 & 15.0 & 5.80 & 4.37 & 3.99 \\
      256 & 15.0 & 15.0 & 5.80 & 3.70 & 3.66 & 3.68 \\
      \bottomrule
    \end{tabular}
  \end{subtable}%
  \hspace{\fill}
  \begin{subtable}{0.49\textwidth}
  \centering
    \caption{Upwind SBP, Steger-Warming splitting \cite{steger1979flux}.}
    \begin{tabular}{r rrrrrr}
      \toprule
      \multicolumn{1}{c}{$K$} & \multicolumn{6}{c}{interior order of accuracy} \\
          & \multicolumn{1}{c}{2}
          & \multicolumn{1}{c}{3}
          & \multicolumn{1}{c}{4}
          & \multicolumn{1}{c}{5}
          & \multicolumn{1}{c}{6}
          & \multicolumn{1}{c}{7} \\
      \midrule
        1 & 15.0 & 15.0 & 15.0 & 15.0 & 5.86 & 4.80 \\
        4 & 15.0 & 15.0 & 15.0 & 15.0 & 15.0 & 4.79 \\
      16 & 15.0 & 15.0 & 15.0 & 15.0 & 5.43 & 4.03 \\
      64 & 15.0 & 15.0 & 15.0 & 5.68 & 4.36 & 4.02 \\
      256 & 15.0 & 15.0 & 5.80 & 3.69 & 3.66 & 3.67 \\
      \bottomrule
    \end{tabular}
  \end{subtable}%
\end{table}

To further investigate this behavior, we ran additional simulations with upwind
SBP operators. Here, we choose the number of nodes such that the total number
of DOFs remains constant. The resulting final simulation times
are shown in Table~\ref{tab:kelvin_helmholtz_constant_dofs}. As in the case
of a constant number of nodes per element investigated before, increasing the
number of elements makes the upwind FD methods less robust. The only exceptions are again
the low-order methods with an interior order of accuracy two and three
(both resulting in an experimental order of convergence of two under mesh
refinement by increasing the number of elements).

\begin{figure}[htbp]
\centering
  \begin{subfigure}{0.43\textwidth}
  \centering
    \includegraphics[width=\textwidth]{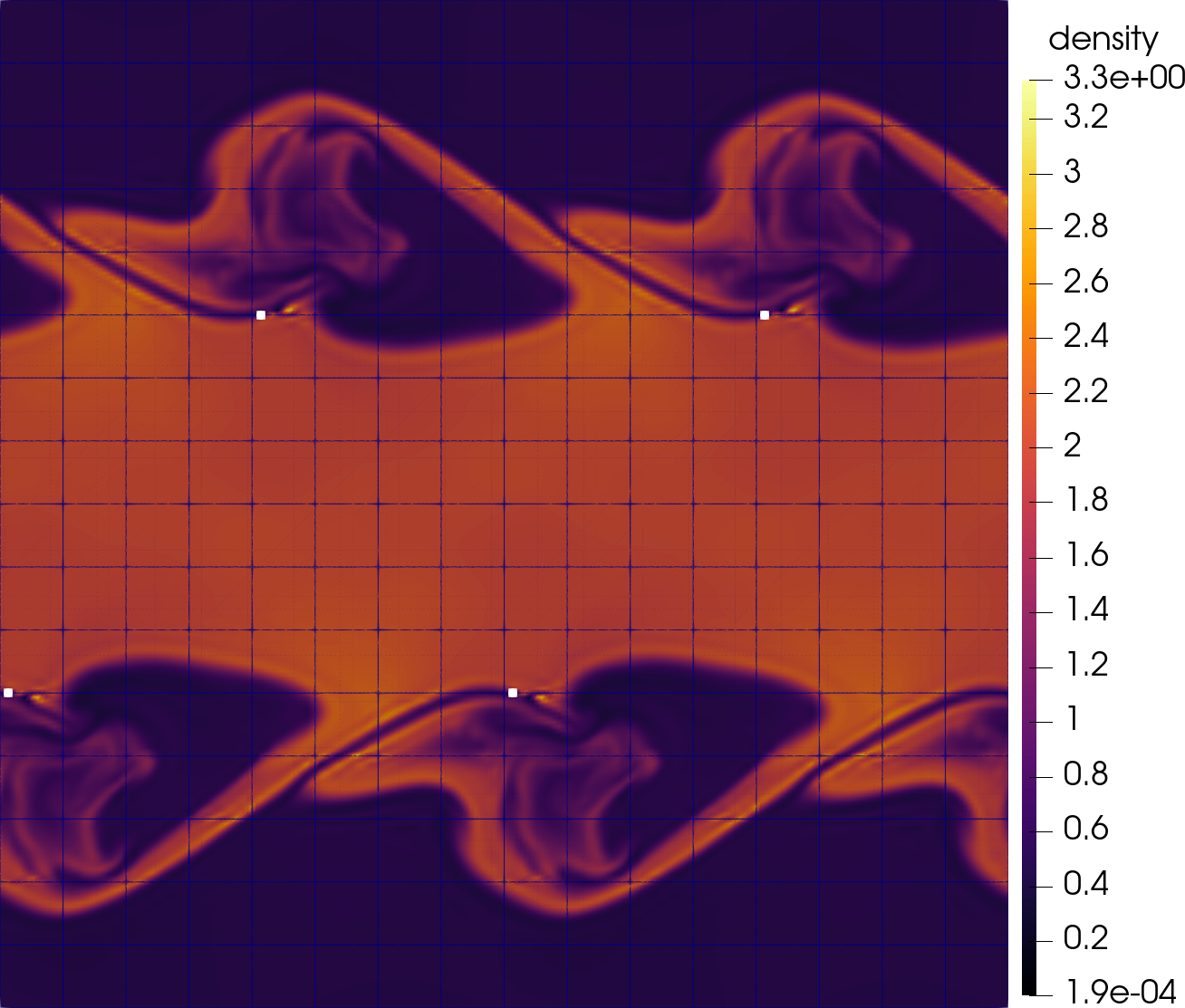}
    \caption{$16^2$ elements with $16^2$ nodes, crash at $t = 3.66$.}
  \end{subfigure}%
  \hspace{\fill}
  \begin{subfigure}{0.43\textwidth}
  \centering
    \includegraphics[width=\textwidth]{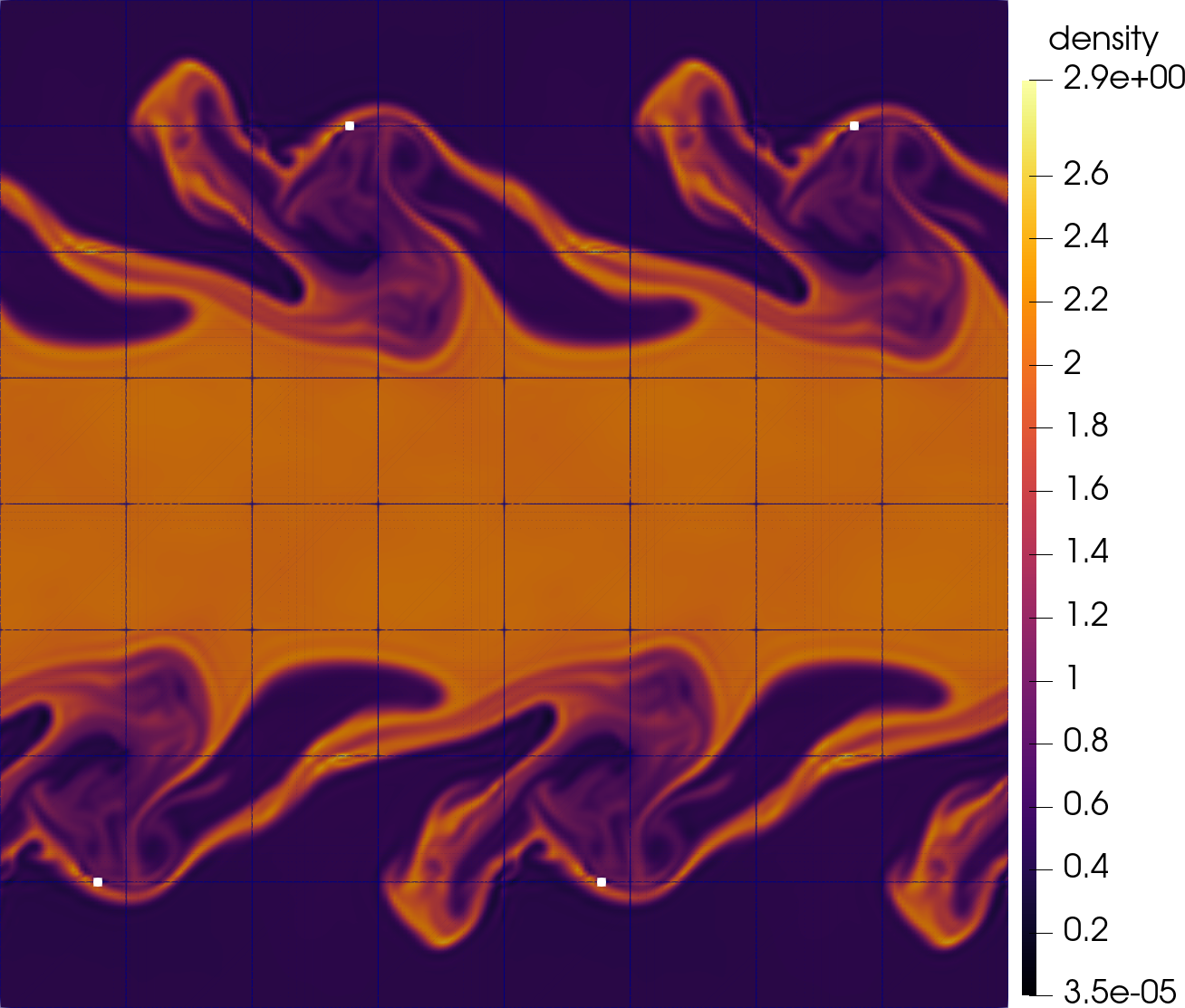}
    \caption{$8^2$ elements with $32^2$ nodes, crash at $t = 4.37$.}
  \end{subfigure}%
  \\
  \begin{subfigure}{0.43\textwidth}
  \centering
    \includegraphics[width=\textwidth]{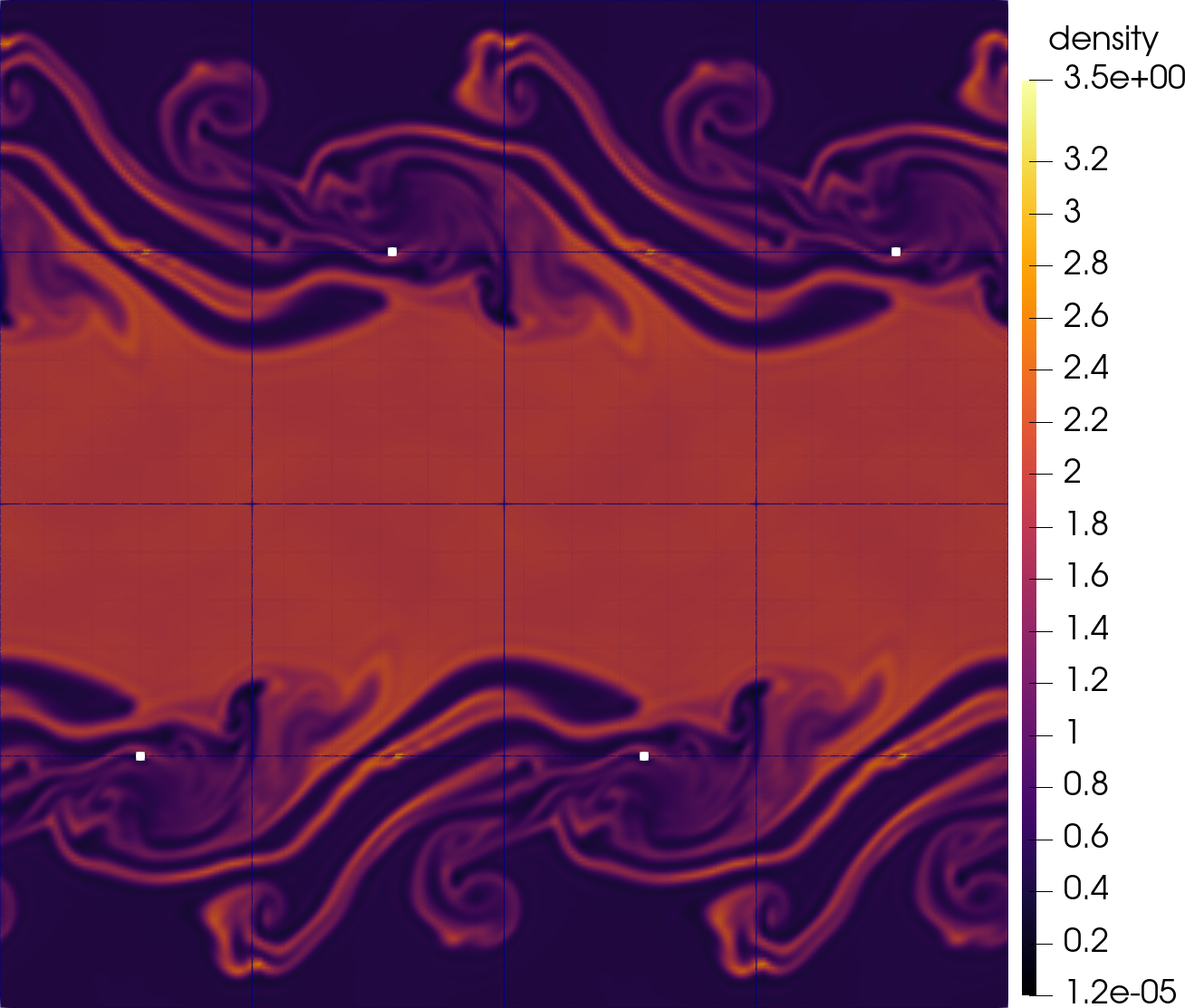}
    \caption{$4^2$ elements with $64^2$ nodes, crash at $t = 5.18$.}
  \end{subfigure}%
  \hspace{\fill}
  \begin{subfigure}{0.43\textwidth}
  \centering
    \includegraphics[width=\textwidth]{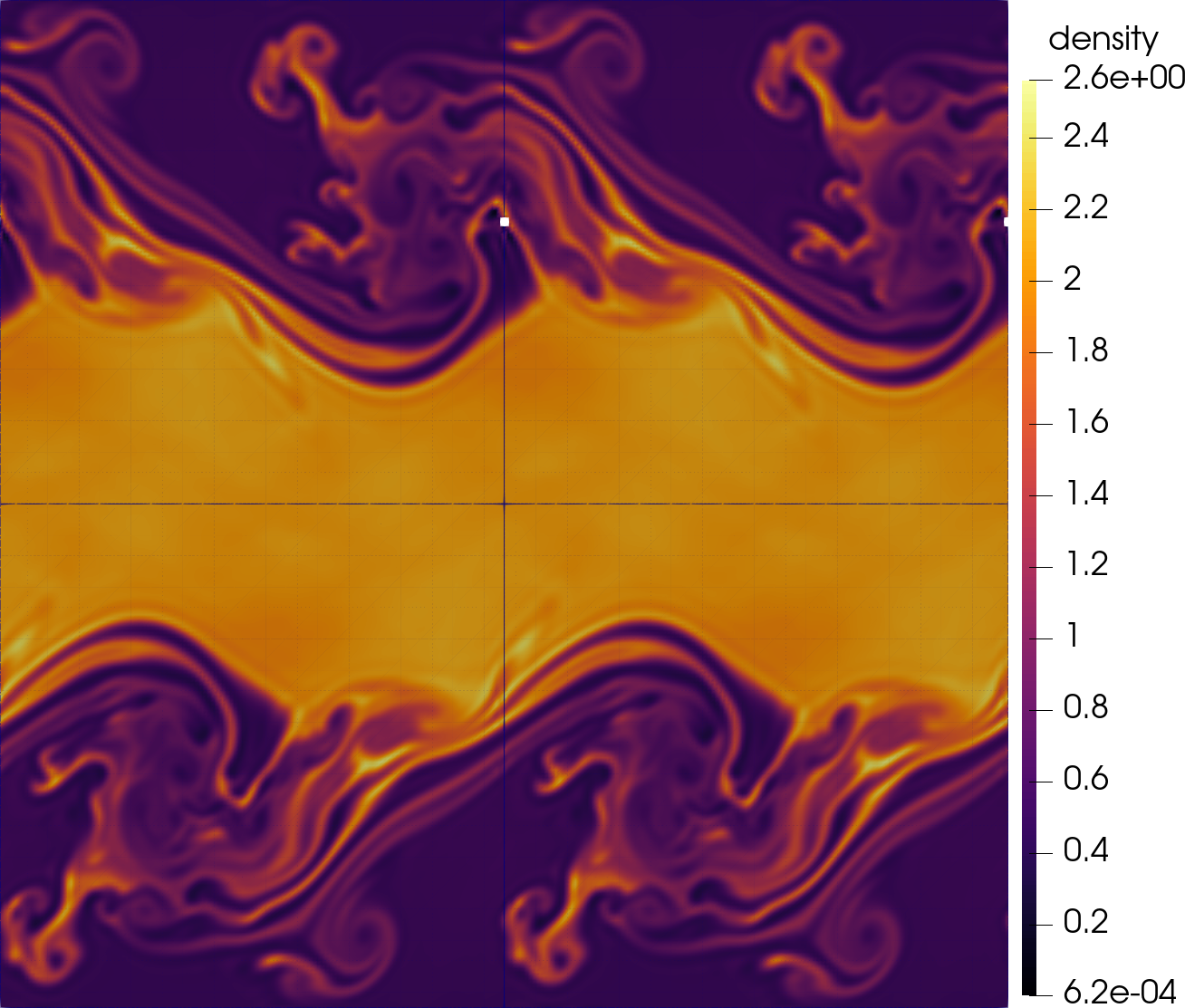}
    \caption{$2^2$ elements with $128^2$ nodes, crash at $t = 5.83$.}
  \end{subfigure}%
  \\
  \begin{subfigure}{0.43\textwidth}
  \centering
    \includegraphics[width=\textwidth]{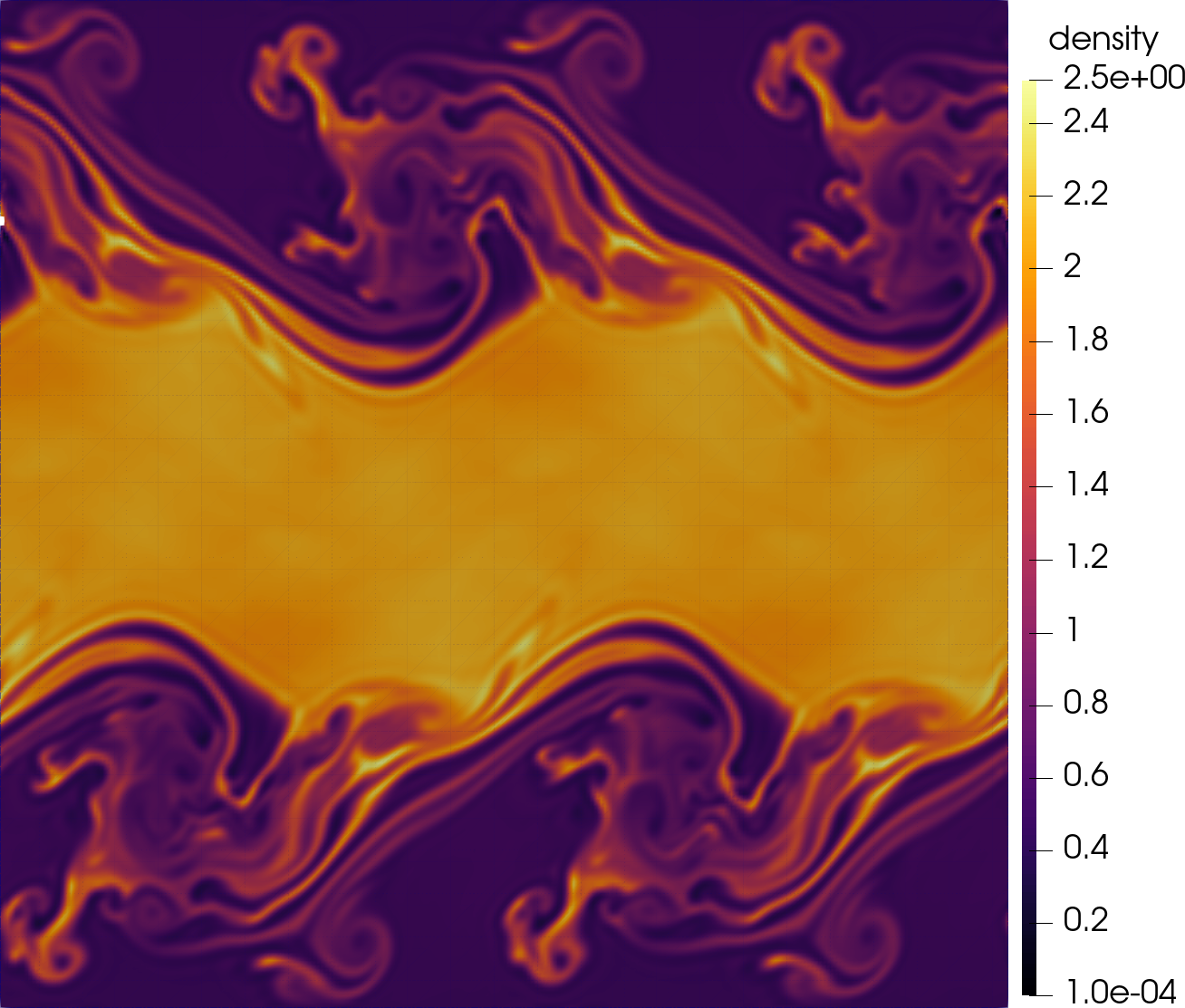}
    \caption{$1$ element with $256^2$ nodes, crash at $t = 5.83$.}
  \end{subfigure}%
  \hspace{\fill}
  \begin{subfigure}{0.43\textwidth}
  \centering
    \includegraphics[width=\textwidth]{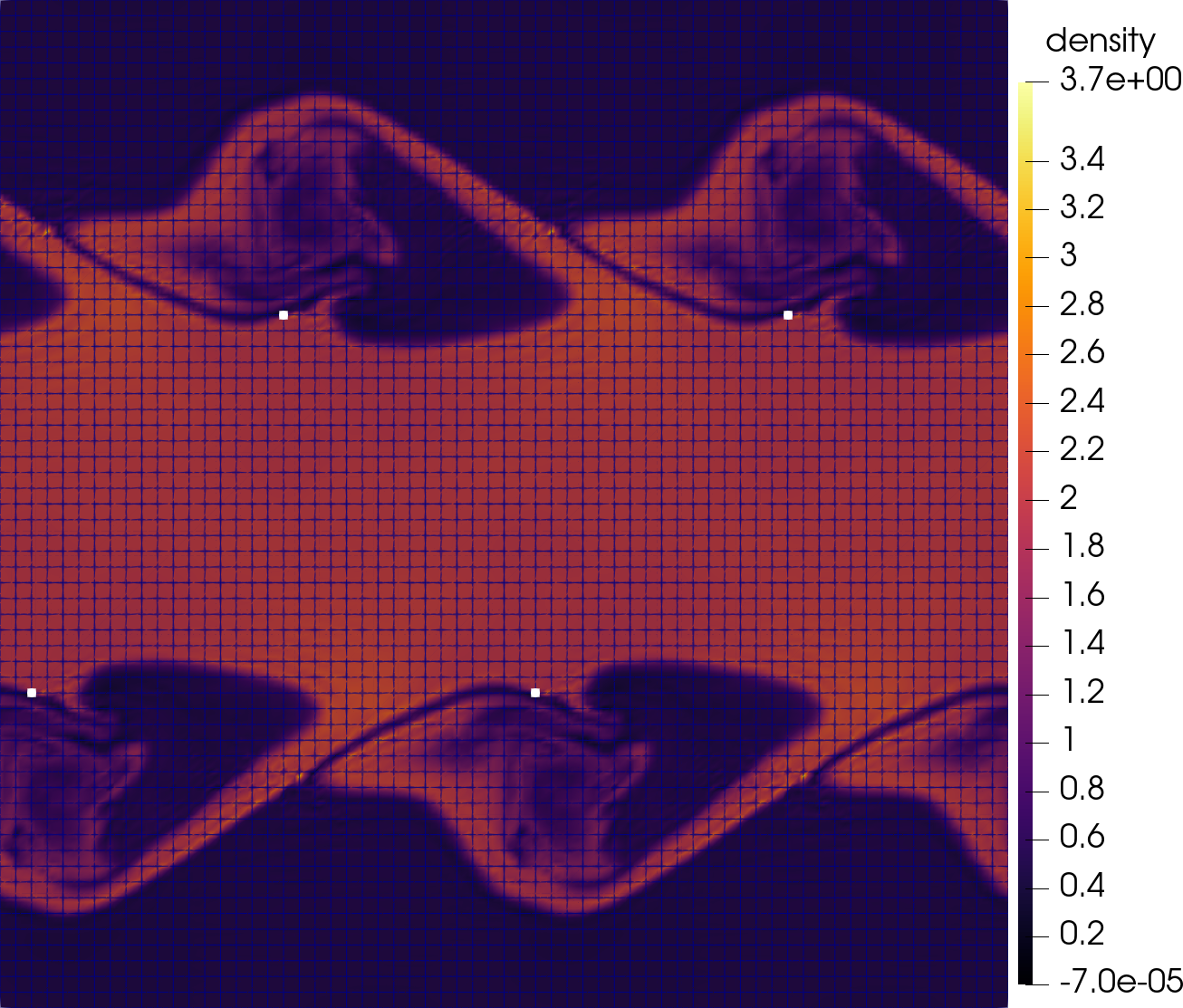}
    \caption{DGSEM, $64^2$ elements with $p = 3$, crash at $t = 3.66$.}
  \end{subfigure}%
  \caption{Visualization of numerical solutions when the simulations
           of the Kelvin-Helmholtz instability crashed. All simulations
           use the same number of DOFs --- with varying numbers of elements/nodes
           for the upwind SBP methods with an interior order of accuracy 6
           using the van Leer-Hänel splitting
           \cite{vanleer1982flux,hanel1987accuracy,liou1991high}. For
           comparison, results obtained by entropy-stable flux differencing
           DGSEM are also shown.
           The white spots mark points where the pressure (upwind SBP) or the
           density (DGSEM) is negative.}
  \label{fig:kelvin_helmholtz_visualization}
\end{figure}

Figure~\ref{fig:kelvin_helmholtz_visualization} shows the numerical solutions
corresponding to the constant DOF setup of
Table~\ref{tab:kelvin_helmholtz_constant_dofs} at the time the simulations
crashed. As usual, we plot the density of the numerical solutions to allow
a comparison with other publications.
The white spots mark the points where the pressure is negative
(for the upwind SBP methods) or where the density is negative (for the DGSEM).
Clearly, the problematic nodes are always located at interfaces between
elements. This appears to be correlated with the earlier crash times of
simulations with more elements and fewer nodes per element. Note that the
periodic boundary conditions are enforced weakly. Thus, even the setup with
a single element has internal interfaces --- and the negative pressure occurs
exactly at one of these boundary points (top left quadrant).

\begin{table}[htbp]
\centering
  \caption{Final times of numerical simulations of the Kelvin-Helmholtz
           instability with a single element using purely periodic upwind
           methods, i.e., only the interior stencils of the corresponding
           upwind SBP operators. Final times less than 15 indicate that the
           simulation crashed.}
  \label{tab:kelvin_helmholtz_periodic_upwind}
  \begin{subtable}{0.49\textwidth}
  \centering
    \caption{Van Leer-Hänel splitting \cite{vanleer1982flux,hanel1987accuracy,liou1991high}.}
    \begin{tabular}{r rrrrrr}
      \toprule
      \multicolumn{1}{c}{\#nodes} & \multicolumn{6}{c}{interior order of accuracy} \\
          & \multicolumn{1}{c}{2}
          & \multicolumn{1}{c}{3}
          & \multicolumn{1}{c}{4}
          & \multicolumn{1}{c}{5}
          & \multicolumn{1}{c}{6}
          & \multicolumn{1}{c}{7} \\
      \midrule
          \num{256} & 15.0 & 15.0 & 15.0 & 15.0 & 15.0 & 15.0 \\
         \num{1024} & 15.0 & 15.0 & 15.0 & 15.0 & 15.0 & 15.0 \\
         \num{4096} & 15.0 & 15.0 & 15.0 & 15.0 & 15.0 & 15.0 \\
        \num{16384} & 15.0 & 15.0 & 15.0 & 15.0 & 15.0 & 15.0 \\
        \num{65536} & 15.0 & 15.0 & 15.0 & 15.0 & 15.0 & 4.77 \\
      \bottomrule
    \end{tabular}
  \end{subtable}%
  \hspace{\fill}
  \begin{subtable}{0.49\textwidth}
  \centering
    \caption{Steger-Warming splitting \cite{steger1979flux}.}
    \begin{tabular}{r rrrrrr}
      \toprule
      \multicolumn{1}{c}{\#nodes} & \multicolumn{6}{c}{interior order of accuracy} \\
          & \multicolumn{1}{c}{2}
          & \multicolumn{1}{c}{3}
          & \multicolumn{1}{c}{4}
          & \multicolumn{1}{c}{5}
          & \multicolumn{1}{c}{6}
          & \multicolumn{1}{c}{7} \\
      \midrule
          \num{256} & 15.0 & 15.0 & 15.0 & 15.0 & 15.0 & 15.0 \\
         \num{1024} & 15.0 & 15.0 & 15.0 & 15.0 & 15.0 & 15.0 \\
         \num{4096} & 15.0 & 15.0 & 15.0 & 15.0 & 15.0 & 15.0 \\
        \num{16384} & 15.0 & 15.0 & 15.0 & 15.0 & 15.0 & 15.0 \\
        \num{65536} & 15.0 & 15.0 & 15.0 & 15.0 & 15.0 & 15.0 \\
      \bottomrule
    \end{tabular}
  \end{subtable}%
\end{table}

To investigate this claim, we considered purely periodic upwind methods using only the interior
coefficients of the upwind SBP operators. As shown in
Table~\ref{tab:kelvin_helmholtz_periodic_upwind}, this version is much
more robust. Indeed, all of the simulations ran successfully except the
van Leer-Hänel splitting with an interior order of accuracy 7 and
$256^2 = \num{65536}$ nodes in total.

\subsection{Inviscid Taylor-Green vortex}
\label{sec:tgv}

Next, we consider the classical inviscid Taylor-Green vortex for the 3D
compressible Euler equations of an ideal gas following
\cite{gassner2016split}. Specifically, we consider the initial conditions
\begin{equation}
\begin{gathered}
%
%
  \rho = 1, \;
  v_1 =  \sin(x_1) \cos(x_2) \cos(x_3), \;
  v_2 = -\cos(x_1) \sin(x_2) \cos(x_3), \;
  v_3  = 0, \\
  p = \frac{\rho^0}{\mathrm{Ma}^2 \gamma} + \rho^0 \frac{\cos(2 x_1) \cos(2 x_3) + 2 \cos(2 x_2) + 2 \cos(2 x_1) + \cos(2 x_2) \cos(2 x_3)}{16},
\end{gathered}
\end{equation}
where $\mathrm{Ma} = 0.1$ is the Mach number.
We consider the domain $[-\pi, \pi]^3$ with periodic boundary conditions
and a time interval $[0, 20]$.
We integrate the semidiscretizations in time with the third-order,
four-stage SSP method of \cite{kraaijevanger1991contractivity} with embedded
method of \cite{conde2022embedded} and error-based step size controller
developed in \cite{ranocha2021optimized} with absolute and relative tolerance
chosen as $10^{-6}$.

\begin{table}[htbp]
\centering
  \caption{Final times of numerical simulations of the inviscid Taylor-Green
           vortex with $\mathrm{Ma} = 0.1$, $K$ elements, and upwind SBP
           methods using the Steger-Warming splitting \cite{steger1979flux}.
           Final times less than 20 indicate that the simulation crashed.}
  \label{tab:taylor_green}
  \begin{subtable}{0.49\textwidth}
  \centering
    \caption{Constant number of nodes per direction (= 16).}
    \begin{tabular}{r rrrrrr}
      \toprule
      \multicolumn{1}{c}{$K$} & \multicolumn{6}{c}{interior order of accuracy} \\
          & \multicolumn{1}{c}{2}
          & \multicolumn{1}{c}{3}
          & \multicolumn{1}{c}{4}
          & \multicolumn{1}{c}{5}
          & \multicolumn{1}{c}{6}
          & \multicolumn{1}{c}{7} \\
      \midrule
        1 & 20.0 & 20.0 & 20.0 & 20.0 & 20.0 & 20.0 \\
        8 & 20.0 & 20.0 & 20.0 & 20.0 & 20.0 & 13.5 \\ 
       64 & 20.0 & 20.0 & 20.0 & 20.0 & 20.0 & 6.12 \\
      \bottomrule
    \end{tabular}
  \end{subtable}%
  \hspace{\fill}
  \begin{subtable}{0.49\textwidth}
  \centering
    \caption{Constant number of DOFs (= \num{262144}).}
    \begin{tabular}{r rrrrrr}
      \toprule
      \multicolumn{1}{c}{$K$} & \multicolumn{6}{c}{interior order of accuracy} \\
          & \multicolumn{1}{c}{2}
          & \multicolumn{1}{c}{3}
          & \multicolumn{1}{c}{4}
          & \multicolumn{1}{c}{5}
          & \multicolumn{1}{c}{6}
          & \multicolumn{1}{c}{7} \\
      \midrule
        1 & 20.0 & 20.0 & 20.0 & 20.0 & 20.0 & 6.1 \\ 
        4 & 20.0 & 20.0 & 20.0 & 20.0 & 20.0 & 15.0 \\ 
       16 & 20.0 & 20.0 & 20.0 & 20.0 & 20.0 & 6.12 \\
      \bottomrule
    \end{tabular}
  \end{subtable}%
\end{table}

The final times of simulations using upwind SBP operators are shown in
Table~\ref{tab:taylor_green}. For 16 nodes per coordinate direction, only
the method with an interior order of accuracy 7 crashed --- all lower-order
methods completed the full simulation.

\begin{table}[htbp]
\centering
  \caption{Final times of numerical simulations of the inviscid Taylor-Green
           vortex with $\mathrm{Ma} = 0.4$, $K$ elements, and upwind SBP
           methods using the Steger-Warming splitting \cite{steger1979flux}
           with a constant number of nodes per direction (= 16).
           Final times less than 20 indicate that the simulation crashed.}
  \label{tab:taylor_green_mach04}
  \begin{subtable}{0.49\textwidth}
  \centering
    \begin{tabular}{r rrrrrr}
      \toprule
      \multicolumn{1}{c}{$K$} & \multicolumn{6}{c}{interior order of accuracy} \\
          & \multicolumn{1}{c}{2}
          & \multicolumn{1}{c}{3}
          & \multicolumn{1}{c}{4}
          & \multicolumn{1}{c}{5}
          & \multicolumn{1}{c}{6}
          & \multicolumn{1}{c}{7} \\
      \midrule
        1 & 20.0 & 20.0 & 20.0 & 20.0 & 5.79 & 5.31 \\
        8 & 20.0 & 20.0 & 20.0 & 4.79 & 4.00 & 3.89 \\
      64 & 20.0 & 20.0 & 8.40 & 5.70 & 4.42 & 4.18 \\
      \bottomrule
    \end{tabular}
  \end{subtable}%
\end{table}

Next, we repeat the numerical robustness experiments with an increased Mach
number $\mathrm{Ma} = 0.4$. The results are shown in
Table~\ref{tab:taylor_green_mach04}. The increased Mach number introduces
more compressibility effects, testing the robustness of the numerical
methods in another regime. For the upwind SBP methods considered here, this
leads to a reduced numerical robustness. Indeed, only the upwind SBP methods
with an interior order of accuracy two and three complete all simulations.
The higher-order methods crash for increased resolution. These results are
comparable to the robustness results we observed for the Kelvin-Helmholtz
instability in Section~\ref{sec:khi}.

\begin{table}[htbp]
\centering
  \caption{Final times of numerical simulations of the inviscid Taylor-Green
           vortex with a single element and periodic upwind SBP methods using the
           Steger-Warming splitting \cite{steger1979flux}.
           Final times less than 20 indicate that the simulation crashed.}
  \label{tab:taylor_green_periodic_upwind}
  \begin{subtable}{0.49\textwidth}
  \centering
    \caption{Mach number $\mathrm{Ma} = 0.1$.}
    \begin{tabular}{r rrrrrr}
      \toprule
      \multicolumn{1}{c}{\#nodes} & \multicolumn{6}{c}{interior order of accuracy} \\
          & \multicolumn{1}{c}{2}
          & \multicolumn{1}{c}{3}
          & \multicolumn{1}{c}{4}
          & \multicolumn{1}{c}{5}
          & \multicolumn{1}{c}{6}
          & \multicolumn{1}{c}{7} \\
      \midrule
          \num{4096} & 20.0 & 20.0 & 20.0 & 20.0 & 20.0 & 20.0 \\
         \num{32768} & 20.0 & 20.0 & 20.0 & 20.0 & 20.0 & 20.0 \\
        \num{262144} & 20.0 & 20.0 & 20.0 & 20.0 & 20.0 & 20.0 \\
      \bottomrule
    \end{tabular}
  \end{subtable}%
  \hspace{\fill}
  \begin{subtable}{0.49\textwidth}
  \centering
    \caption{Mach number $\mathrm{Ma} = 0.4$.}
    \begin{tabular}{r rrrrrr}
      \toprule
      \multicolumn{1}{c}{\#nodes} & \multicolumn{6}{c}{interior order of accuracy} \\
          & \multicolumn{1}{c}{2}
          & \multicolumn{1}{c}{3}
          & \multicolumn{1}{c}{4}
          & \multicolumn{1}{c}{5}
          & \multicolumn{1}{c}{6}
          & \multicolumn{1}{c}{7} \\
      \midrule
          \num{4096} & 20.0 & 20.0 & 20.0 & 20.0 & 20.0 & 20.0 \\
         \num{32768} & 20.0 & 20.0 & 20.0 & 20.0 & 20.0 & 20.0 \\
        \num{262144} & 20.0 & 20.0 & 20.0 & 20.0 & 20.0 & 20.0 \\
      \bottomrule
    \end{tabular}
  \end{subtable}%
\end{table}

Table~\ref{tab:taylor_green_periodic_upwind} shows results obtained by
fully periodic upwind SBP methods using only the interior coefficients
of the upwind SBP operators. As for the Kelvin-Helmholtz instability
considered before, this version is much more robust --- all of the
simulations run successfully to the final time.

\begin{figure}[htbp]
\centering
  \begin{subfigure}{0.49\textwidth}
  \centering
    \includegraphics[width=\textwidth]{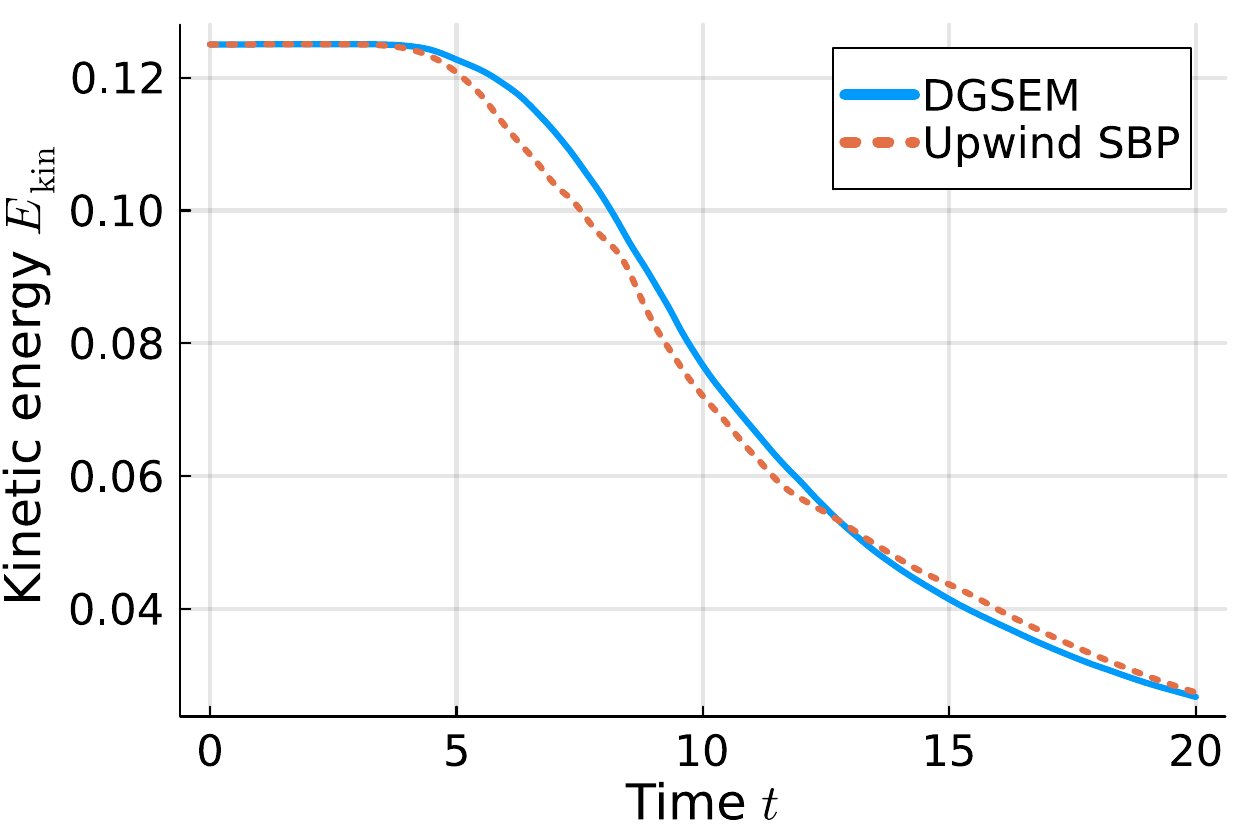}
    \caption{Discrete kinetic energy.}
  \end{subfigure}%
  \vspace{\fill}
  \begin{subfigure}{0.49\textwidth}
  \centering
    \includegraphics[width=\textwidth]{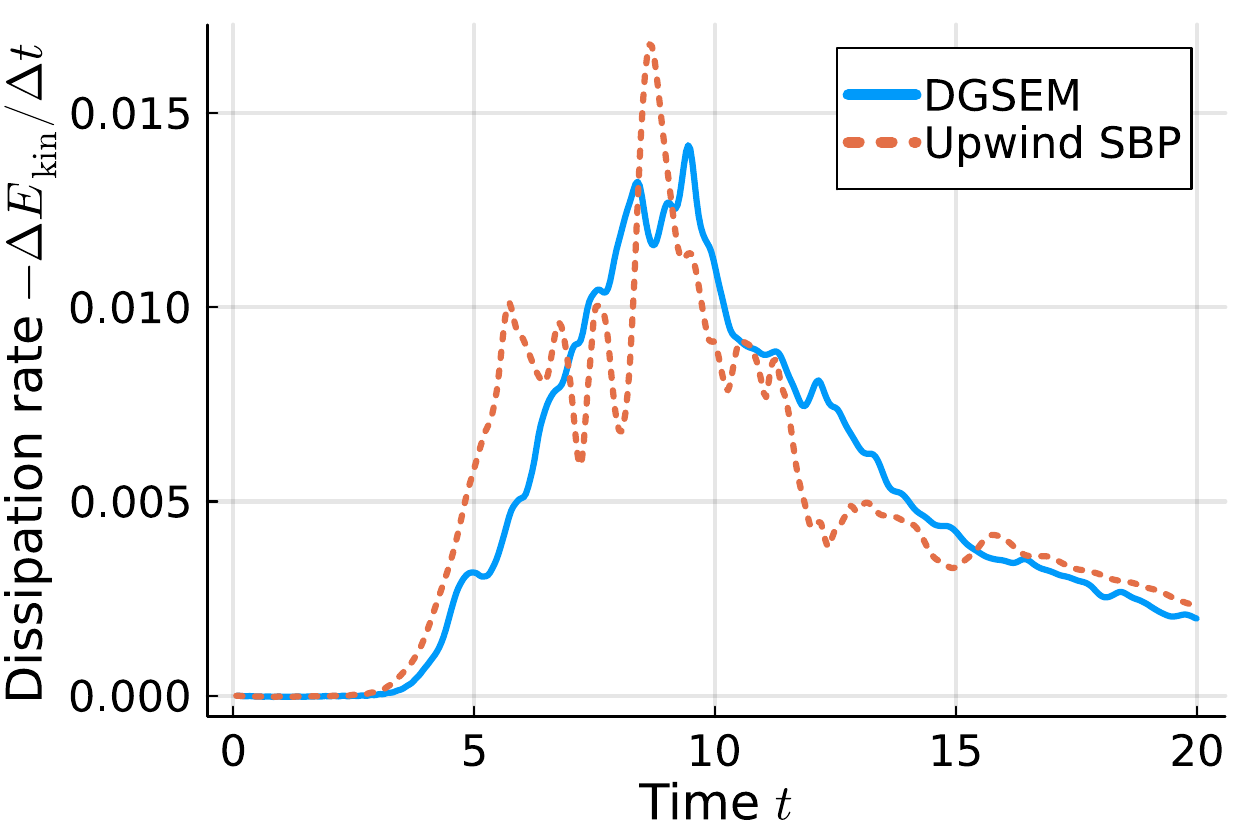}
    \caption{Discrete kinetic energy dissipation rate.}
  \end{subfigure}%
  \caption{Discrete kinetic energy and its dissipation rate for the inviscid
           Taylor-Green vortex. We compare the results of flux differencing
           DGSEM with upwind SBP methods. The DGSEM scheme uses 16 elements
           per coordinate direction, the entropy-conservative flux of Ranocha
           \cite{ranocha2018comparison,ranocha2020entropy,ranocha2021preventing}
           in the volume, and the local Lax-Friedrichs (Rusanov) flux at
           interfaces.
           The upwind SBP method uses 4 elements and 16 nodes per coordinate
           direction, the Steger-Warming splitting \cite{steger1979flux},
           and the operators of Mattsson \cite{mattsson2017diagonal}
           with an interior order of accuracy 6.
           Thus, both simulations use \SI{262144} DOFs.}
  \label{fig:tgv}
\end{figure}

Next, we follow the approach of \cite{gassner2016split} to compute the
kinetic energy dissipation rate for the Mach number $\mathrm{Ma} = 0.1$.
Specifically, we compute the discrete version of the
total kinetic energy
\begin{equation}
  E_\mathrm{kin}(t) = \int \frac{1}{2} \rho(t, x) v(t, x)^2 \dif x
\end{equation}
using the quadrature rule associated with the SBP mass matrix $M$ every
10 accepted time steps. Then, we use central finite differences to compute
the discrete kinetic energy dissipation rate $-\Delta E_\mathrm{kin} / \Delta t$
approximating $-\dif E_\mathrm{kin} / \dif t$. The results are visualized in
Figure~\ref{fig:tgv}.
The qualitative behavior of the kinetic energy and its dissipation rate are
the same for the flux differencing DGSEM method and the upwind SBP method.
The upwind SBP method tends to begin dissipating the kinetic energy earlier
than the DGSEM and shows a dissipation rate that is a bit more oscillatory.
The results of the upwind SBP method do not change if we use half the number
of elements but double the number of nodes per element (not shown in the
plots).
The results of the flux differencing DGSEM simulation match the results
of \cite{gassner2016split} for the same polynomial degree and resolution
(up to the smaller final time $T = 14$ used there).

We also measured the execution time of the upwind SBP method and the
flux differencing DGSEM used in this example on a MacBook with an Apple M2
CPU. The total time spent in the ODE right-hand side computation to simulate
the Taylor-Green vortex in the time interval $[0, 1]$ on a single thread
without any parallelism is roughly
\SI{12.34(0.01)}{s} 
for the flux differencing DGSEM and
\SI{13.73(0.11)}{s} 
for the upwind SBP method
(results of five runs, average and standard deviation,
same setup as used in Section~\ref{sec:tgv}).
Please note that this compares a highly tuned implementation of the DGSEM
using SIMD instructions as described in \cite{ranocha2023efficient} with
a first implementation of the upwind SBP methods in a research code.
Thus, we conclude that both methods are of comparable efficiency.

\section{Summary and conclusions}
\label{sec:summary}

We have discussed high-order upwind SBP methods for nonlinear
conservation laws. Introduced by Mattsson in \cite{mattsson2017diagonal},
these methods combine central-type classical SBP operators with
artificial dissipation and need a flux vector splitting for nonlinear
conservation laws. Lax-Friedrichs type splittings have been
predominantly considered in the literature
\cite{svard2005steady,mattsson2007high,mattsson2017diagonal,lundgren2020efficient,stiernstrom2021residual}.
To combine upwind SBP operators with multiple other flux vector splittings,
we have described a general way to design SATs as in discontinuous Galerkin
methods using numerical fluxes resulting from the chosen splitting
in Section~\ref{sec:formulations}.
Further, we discussed how to extend splittings other than those of Lax-Friedrichs
type into the high-order upwind SBP framework of Mattsson on unstructured curvilinear meshes.
Through this analysis we found an interplay between the dependency of
said splittings, like the van Leer-Hänel, on the metric terms and the boundary closure
accuracy of the upwind SBP operator. Only under specific conditions on the mapping,
the metric terms, and the boundary closure could the resulting method retain the important
free-stream preservation property in generalized coordinates.
We have proven the local linear/energy stability of upwind SBP methods
for Burgers' equation in Section~\ref{sec:stability}. This kind of stability
property is not the classical stability property of a numerical method
applied to a linearized problem, but a property of the linearization
thereof applied to a nonlinear problem. Since linearization and application
of a high-order method for conservation laws do not commute in general,
it is nontrivial to satisfy stability properties such as entropy stability
for the nonlinear problem and local linear/energy stability at the same
time. In particular, we are not aware of any numerical method that has
all three desirable properties
i) nonlinear entropy stability,
ii) local linear/energy stability, and
iii) high-order accuracy.
Methods based on classical SBP operators can be designed to be high-order
accurate and entropy-stable but lack local linear/energy stability
\cite{gassner2022stability}. We have complemented these results by
proving that high-order upwind SBP methods satisfy local linear/energy
stability. We have also discussed the relation to a very special case of
entropy stability. While this case is only an academic example, we hope
that it may lead the community in a way to solve the entropy/linear
stability issue.

We have applied upwind SBP methods with several flux vector splittings
in Section~\ref{sec:experiments}.
The robustness and computational efficiency of the upwind SBP
methods for nonlinear conservation laws are roughly comparable to
highly tuned flux differencing discontinuous Galerkin spectral element
methods, as demonstrated for several examples of compressible fluid flows
and under-resolved simulations.
The numerical tests demonstrated that the upwind SBP methods remained
high-order accurate on unstructured curvilinear domains and free-stream preservation
was retained provided any curved boundaries were approximated with an
appropriate polynomial order dictated by the boundary closure accuracy of a given SBP operator.
These validation tests were performed on well-resolved simulation setups.
For under-resolved simulations, we have shown that
results for a classical inviscid Taylor-Green vortex are promising,
but more challenging tests such as a Kelvin-Helmholtz instability show
that upwind SBP methods do not fix all high-order robustness issues for
shock-free flows. In particular, some robustness (positivity) issues
manifest mainly at interfaces and corners. Thus, upwind SBP methods are
roughly comparable to other modern stabilizations for high-order schemes.

\section*{Acknowledgments}

HR was supported by the Deutsche Forschungsgemeinschaft
(DFG, German Research Foundation, project numbers 513301895 and 528753982
as well as within the DFG priority program SPP~2410 with project number 526031774)
and the Daimler und Benz Stiftung (Daimler and Benz foundation,
project number 32-10/22).
ARW was funded through Vetenskapsr{\aa}det, Sweden grant
agreement 2020-03642 VR.
MSL received funding through the DFG research unit FOR~5409 "Structure-Preserving Numerical Methods
for Bulk- and Interface Coupling of Heterogeneous Models (SNuBIC)" (project number 463312734),
as well as through a DFG individual grant (project number 528753982).
PÖ was supported by the DFG within the priority research program
SPP 2410, project OE 661/5-1
(525866748) and under the personal grant 520756621 (OE 661/4-1).
JG was supported by the US DOD (ONR MURI) grant \#N00014-20-1-2595.
GG acknowledges funding through the Klaus-Tschira Stiftung via the project
``HiFiLab'' and received funding through the DFG research unit FOR 5409 ``SNUBIC''
and through the BMBF funded project ``ADAPTEX''.

Some of the computations were enabled by resources provided by
the National Academic Infrastructure for Supercomputing in Sweden (NAISS), partially 
funded by the Swedish Research Council through grant agreement no. 2022-06725.

\printbibliography

\end{document}